\documentclass[12pt]{article}
\usepackage{amsthm, amsmath, amssymb, amsfonts, booktabs, graphicx, natbib, setspace}

\graphicspath{{Figures/}}

\newtheoremstyle{example}{\medskipamount}{\medskipamount}{\normalfont}{}{\scshape}{.}{1em}{}

\theoremstyle{example}
\newtheorem{example}{Example}

\newtheorem{thm}{Theorem}[section]
\newtheorem{lem}[thm]{Lemma}
\newtheorem{prop}[thm]{Proposition}
\newtheorem{cor}[thm]{Corollary}

\newtheorem{rem}[thm]{Remark}

\newcommand{\FF}{{\mathbb F}}

\newcommand{\NN}{{\mathbb N}}

\newcommand{\RR}{{\mathbb R}}

\newcommand{\UU}{{\mathbb U}}

\newcommand{\WW}{{\mathbb W}}

\newcommand{\ZZ}{{\mathbb Z}}
\newcommand{\mc}[1]{\mathcal {#1}}

\newcommand{\eps}{\epsilon}

\newcommand{\argmax}{\mathrm{argmax}}
\newcommand{\rear}{\mathrm{rear}}
\newcommand{\gren}{\mathrm{gren}}
\newcommand{\cov}{\mathrm{cov}}

\newcommand{\bb}[1]{\mathbb {#1}}
\newcommand{\limsupp}{{\overline{\mathrm{lim}}}}

\begin{document}

\title{Estimation of a Discrete Monotone Distribution}
\author{Hanna K. Jankowski and Jon A. Wellner}
\date{\today}
\maketitle

\begin{abstract}
We study and compare three estimators of a discrete monotone distribution:
(a) the (raw) empirical estimator;
(b) the ``method of rearrangements'' estimator; and
(c) the maximum likelihood estimator.
We show that the maximum likelihood estimator strictly dominates both the rearrangement
and empirical estimators in cases when the distribution has intervals of constancy.  For example, when the distribution is 
uniform on $\{ 0, \ldots , y \}$, the asymptotic risk of the 
method of rearrangements estimator (in squared $\ell_2$ norm)
 is $y/(y+1)$, while the asymptotic risk of the MLE is of order $(\log y)/(y+1)$.  For strictly decreasing distributions, the estimators are asymptotically equivalent.


\end{abstract}


\section{Introduction}\label{sec:Intro}

This paper is motivated in large part by the recent surge of acitivity concerning 
``method of rearrangement'' estimators for nonparametric estimation of monotone functions:  
see, for example, \cite{MR1486918},  \cite{MR2225147},  \cite{MR2232727},  
\cite{CFG} and \cite{anevski:fougeres:07}. 
Most of these authors study continuous settings and often start with a 
kernel type estimator of the density, which involves choices of a kernel 
and of a bandwidth.  Our goal here is to investigate method of rearrangement 
estimators and compare them to natural alternatives (including the maximum 
likelihood estimators with and without the assumption of monotonicity) in a 
setting in which there is less ambiguity in the choice of an initial or ``basic'' estimator, 
namely the setting of estimation of a monotone decreasing mass function on the 
non-negative integers $\NN = \{0,1,2, \ldots \}$.

Suppose that $p = \{ p_x \}_{x \in \NN}$ is a probability mass function; i.e.
$p_x \ge 0$ for all $x \in \NN$ and $\sum_{x \in \NN} p_x = 1$.
Our primary interest here is in the situation in which $p$ is
monotone decreasing:  $p_{x} \geq p_{x+1}$ for all $ x \in \NN$.
The three estimators of $p$ we study are:
\newcounter{count14}
\begin{list}{(\alph{count14}).}
        {\usecounter{count14}
        \setlength{\topsep}{6pt}
        \setlength{\parskip}{0pt}
        \setlength{\partopsep}{0pt}
        \setlength{\parsep}{0pt}
        \setlength{\itemsep}{3pt}
        \setlength{\leftmargin}{40pt}}
\item the (raw) empirical estimator,
\item the method of rearrangement estimator,
\item the maximum likelihood estimator.
\end{list}
Notice that the empirical estimator is also the maximum likelihood estimator 
when no shape assumption is made on the true probability mass function.

Much as in the continuous case our considerations here carry over to the case of estimation of unimodal mass functions
with a known (fixed) mode; see e.g. \cite{MR1486918}, \cite{MR902241}, and \cite{MR1246082}.  
For two recent papers discussing connections and trade-offs between discrete and continuous
models in a related problem involving nonparametric estimation of a monotone function,
see \cite{BanerjeeKosorokTang:09} and \cite{MaathuisHudgens:09}.
  
Distributions from the monotone decreasing family satisfy 
$\Delta p_x \equiv p_{x+1} - p_x \le 0$ for all $x \in \NN$, and may be written as mixtures of uniform mass functions
\begin{eqnarray}\label{line:mix_form}
p_x = \sum_{y\geq 0} \frac{1}{y+1} 1_{\{0, \ldots, y\}}(x) \ q_y.
\end{eqnarray}
Here, the mixing distribution $q$ may be recovered via
\begin{eqnarray}\label{line:mix_recover}
q_x = - (x+1) \Delta p_x,
\end{eqnarray}
for any $x\in \NN.$

\begin{rem}\label{rem:p_k_bound}
From the form of the mass function, it follows that $p_x \leq 1/(x+1)$ for all~$x\geq 0.$
\end{rem}

Suppose then that we observe $X_1 , X_2, \ldots , X_n$ i.i.d. random variables with values in
$\NN$ and with a monotone decreasing mass function $p$.  For $x\in\NN$, let
\begin{eqnarray*}
\widehat{p}_{n,x} \equiv n^{-1} \sum_{i=1}^n 1_{\{ x \}} (X_i ) 
\end{eqnarray*}
denote the (unconstrained) empirical estimator of the probabilities $p_x$.  
Clearly, there is  no guarantee that this estimator will also be monotone decreasing, 
especially for small sample size.  We next consider two estimators 
which do satisfy this property: the rearrangement estimator and the maximum likelihood estimator (MLE).

For a vector $w=\{w_0, \ldots, w_k\}$, let $\rear(w)$ denote the 
reverse-ordered vector such that $w'=\rear(w)$ satisfies 
$w'_0\geq w'_1 \geq \ldots \geq w'_k.$  The rearrangement estimator is then simply defined as
\begin{eqnarray*}
\widehat{p}_{n}^{R} = \rear(\widehat p_n).
\end{eqnarray*}
We can also write $\widehat{p}_{n,x}^{R} = \sup \{ u : Q_n (u) \le x \}$, 
where $Q_n (u) \equiv \# \{ x : \ \widehat{p}_{n,x} \ge u \}.$

To define the MLE we again need some additional notation.  
For a vector $w=\{w_0, \ldots, w_k\}$, let $\gren(w)$ be the operator which 
returns the vector of the $k+1$ slopes of the least concave majorant of the points
\begin{eqnarray*}
\left\{\left(j, \sum_{j=0}^j w_i\right): j=-1, 0, \ldots, k \right\}.
\end{eqnarray*}
Here, we assume that $\sum_{j=0}^{-1}w_j=0.$    
The MLE, also known as the Grenander estimator, is then defined as
\begin{eqnarray*}
\widehat p_n^G = \gren(\widehat p_n).
\end{eqnarray*}
Thus, $\widehat p^G_{n,x}$ is the left derivative at $x$ of the 
least concave majorant (LCM) of the empirical distribution function 
$\FF_n(x)=n^{-1} \sum_{i=1}^n 1_{[0,x]} (X_i)$ (where we include the point $(-1,0)$ to find the left derivative at $x=0$).  
Therefore, by definition, the MLE is a vector of local averages over a partition of 
$\{0, \ldots, \max\{X_1, \ldots, X_n\}\}$.  
This partition is chosen by the touchpoints of the LCM with $\FF_n$.   
It is easily checked that $\widehat p^G_n$ corresponds to the isotonic estimator 
for multinomial data as described in \cite{MR961262}, pages 7--8 and 38--39.

We begin our discussion with two examples: in the first, $p$ is the uniform 
distribution, and in the second $p$ is strictly monotone decreasing.  
To compare the three estimators, we consider several metrics: 
the $\ell_k$ norm for $1\leq k\leq \infty$ and the Hellinger distance.  
Recall that the Hellinger distance between two mass functions is given by
\begin{eqnarray*}
H^2(p,\tilde p) = 2^{-1} \int [ \sqrt{p} - \sqrt{\tilde p} ]^2 d \mu = 2^{-1} \sum_{x\geq0} [ \sqrt{p_x} - \sqrt{\tilde p_x} ]^2,
\end{eqnarray*}
while the $\ell_k$ metrics are defined as
\begin{eqnarray*}
||p-\tilde p||_k&=& \left\{
             \begin{array}{ll}
             \left(\sum_{x\geq 0} |p_x-\tilde p_x|^k\right)^{1/k} & 1\leq k <\infty,\\
             \sup_{x\geq0} |p_x-\tilde p_x|& k=\infty.
             \end{array}\right.
\end{eqnarray*}
In the examples, we compare the Hellinger norm and the $\ell_1$ and $\ell_2$ metrics, as the behaviour of these differs the most.

\begin{figure}[htb!]
\centering
\includegraphics[width=0.45\textwidth]{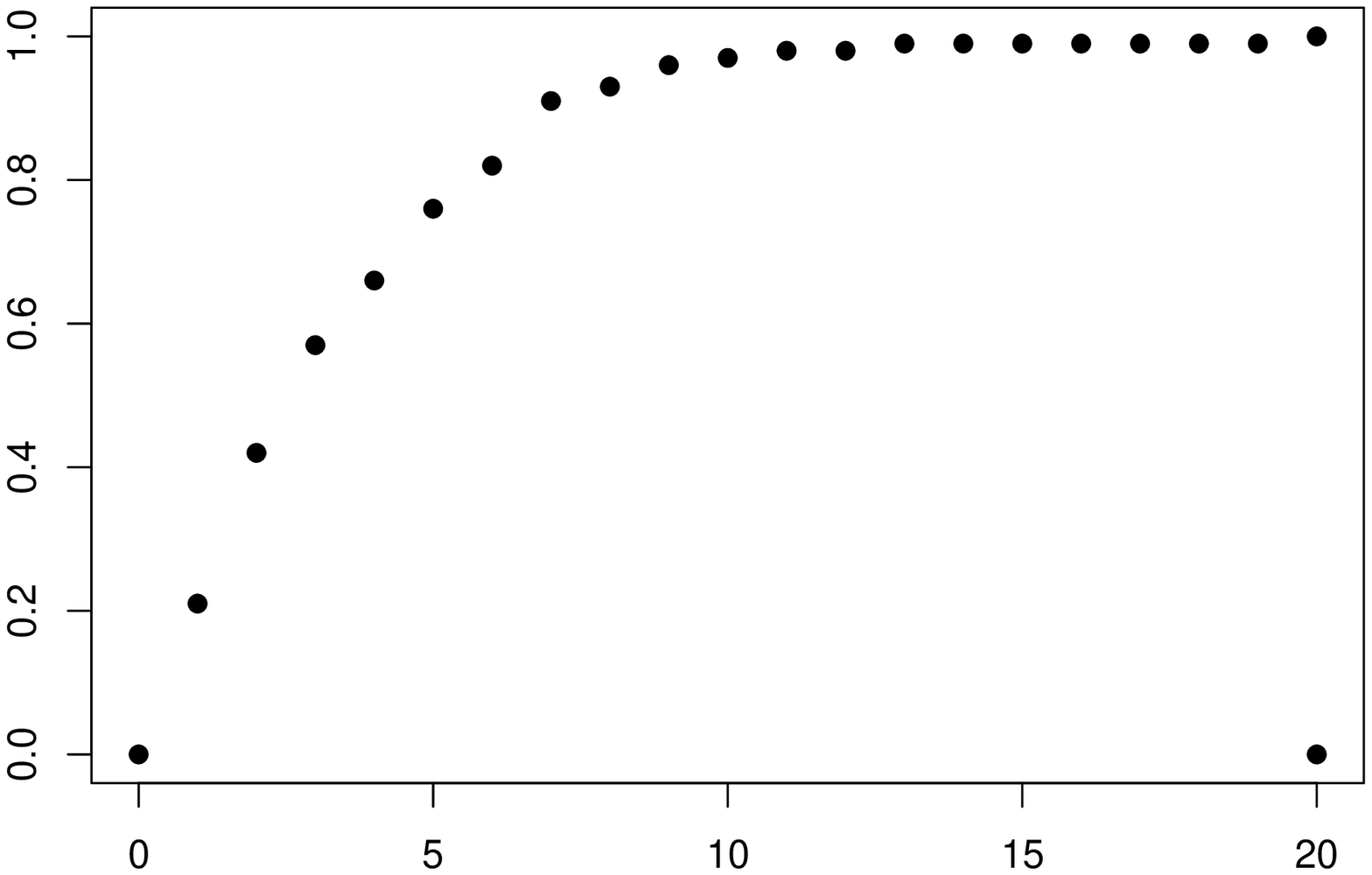}
\includegraphics[width=0.45\textwidth]{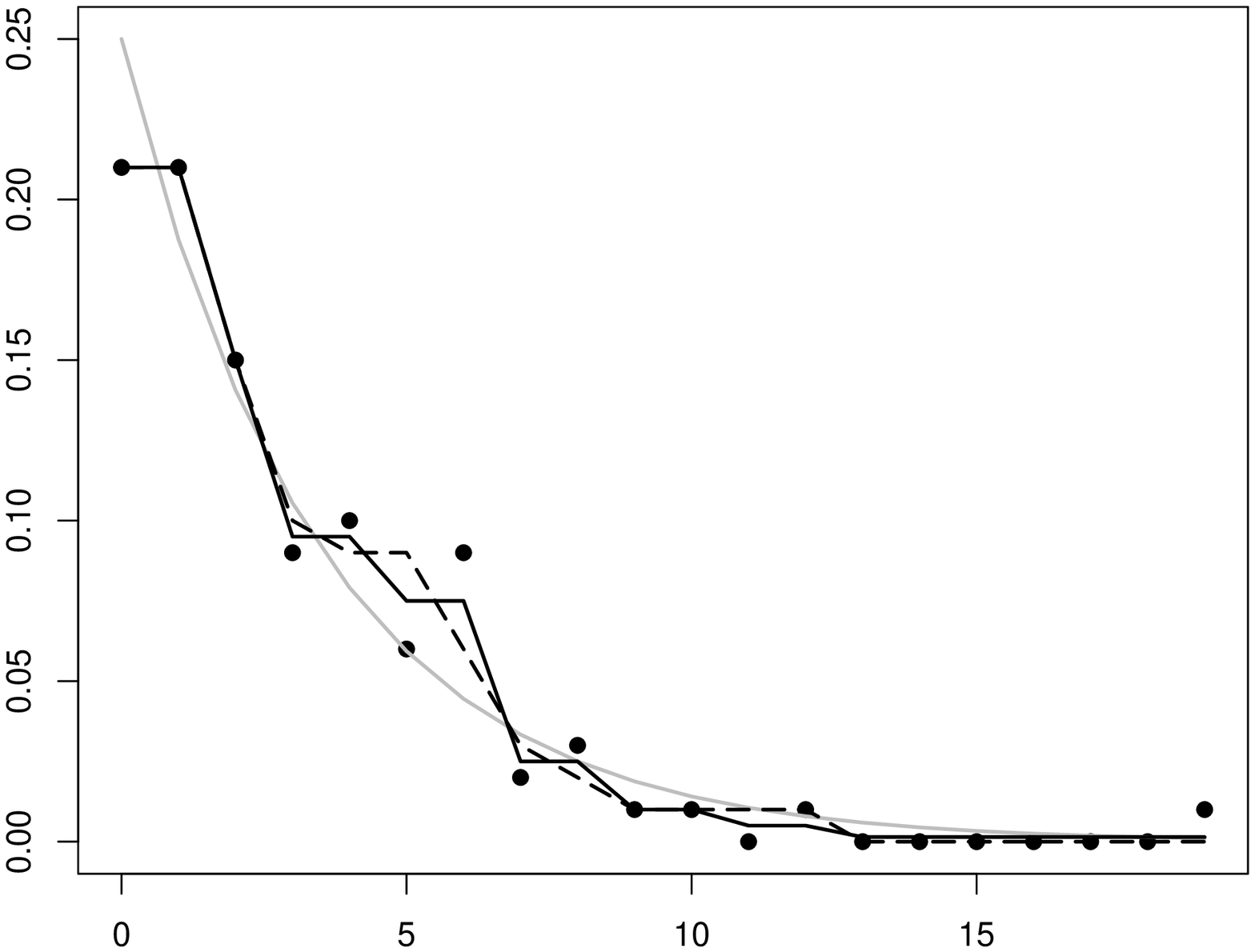}
\caption{Illustration of MLE and monotone rearrangement estimators:
empirical proportions (black dots), monotone rearrangement estimator (dashed line),
MLE (solid line), and the true mass function (grey line).
Left: the true distribution is the discrete uniform; and right: the true distribution
is the geometric distribution with $\theta=0.75$.
In both cases a sample size of $n=100$ was observed.}
\label{fig:comparisons}
\end{figure}

\begin{example}\label{eg:1}
\label{ExampleOneSampleFromUniformZeroToFive}
Suppose that $p$ is the uniform distribution on $\{ 0, \ldots , 5 \}$.
For $n=100$ independent draws from this
distribution we observe $\widehat{p}_{n} =  (0.20,0.14,0.11,0.22,0.15,0.18)$.
Then $\widehat{p}_{n}^R =  (0.22,0.20,0.18,0.15,0.14,0.11)$, and the MLE
may be calculated as $\widehat{p}_{n}^G =  (0.20,0.16,0.16,0.16,0.16,0.16)$.
The estimators are illustrated in Figure \ref{fig:comparisons} (left).
The distances of the estimators from the true mass function $p$ are
 given in Table \ref{tab:eg1} (left).  The maximum likelihood estimator
 $\widehat p_n^G$ is superior in all three metrics shown.
 To explore this relationship further, we repeated the estimation procedure
 for 1000 Monte Carlo samples of size $n=100$ from the uniform distribution.
 Figure \ref{fig:comparisons_box} (left) shows boxplots of the metrics for
 the three estimators.  The figure shows that here the rearrangement
 and empirical estimators have the same behaviour; a relationship which
 we establish rigorously in Theorem~\ref{thm:BasicInequalities}.
\end{example}

\begin{table}[htb!]
\caption{Distances between true $p$ and estimators}
\medskip
\centering
\begin{tabular}{lcccccc}
\toprule[1.5pt]
& \multicolumn{3}{c}{Example \ref{ExampleOneSampleFromUniformZeroToFive}}& \multicolumn{3}{c}{Example \ref{ExampleGeometricSSOneHundred}}\\\cmidrule(l){2-4}\cmidrule(l){5-7}
& $H(\tilde p,p)$ & $||\tilde p-p||_2$ & $||\tilde p-p||_1$ & $H(\tilde p,p)$ & $||\tilde p-p||_2$ & $||\tilde p-p||_1$ \\
\midrule
$\tilde p=\widehat p_n$     & 0.08043  & 0.09129  & 0.2     & 0.1641  & 0.07425  & 0.2299 \\
$\tilde p=\widehat p_n^R$   & 0.08043  & 0.09129  & 0.2     & 0.1290  & 0.06115  & 0.1821 \\
$\tilde p=\widehat p_n^G$   & 0.03048  & 0.03651  & 0.06667 & 0.09553 & 0.06302  & 0.1887 \\
\bottomrule[1.5pt]
\end{tabular}
\label{tab:eg1}
\end{table}

\begin{example}
\label{ExampleGeometricSSOneHundred}
Suppose that $p$ is the geometric distribution with
$p_x = (1-\theta) \theta^x$ for $x \in \NN$ and with $\theta = 0.75$.
For $n=100$ draws from this distribution we observe
$\widehat p_n, \widehat p_n^R$ and $\widehat p_n^G$ as shown in
Figure \ref{fig:comparisons} (right).  The distances of the estimators
from the true mass function $p$ are given in Table~\ref{tab:eg1} (right).
Here, $\widehat p_n$ is outperformed by $\widehat p_n^G$ and
$\widehat p_n^R$ in all the metrics, with $\widehat p_n^R$ performing
better in the $\ell_1$ and $\ell_2$ metrics, but not in the Hellinger distance.
These relationships appear to hold true in general, see Figure
\ref{fig:comparisons_box} (left) for boxplots of the metrics obtained
through Monte Carlo simulation.
\end{example}

\begin{figure}[htb!]
\centering
\includegraphics[width=0.45\textwidth]{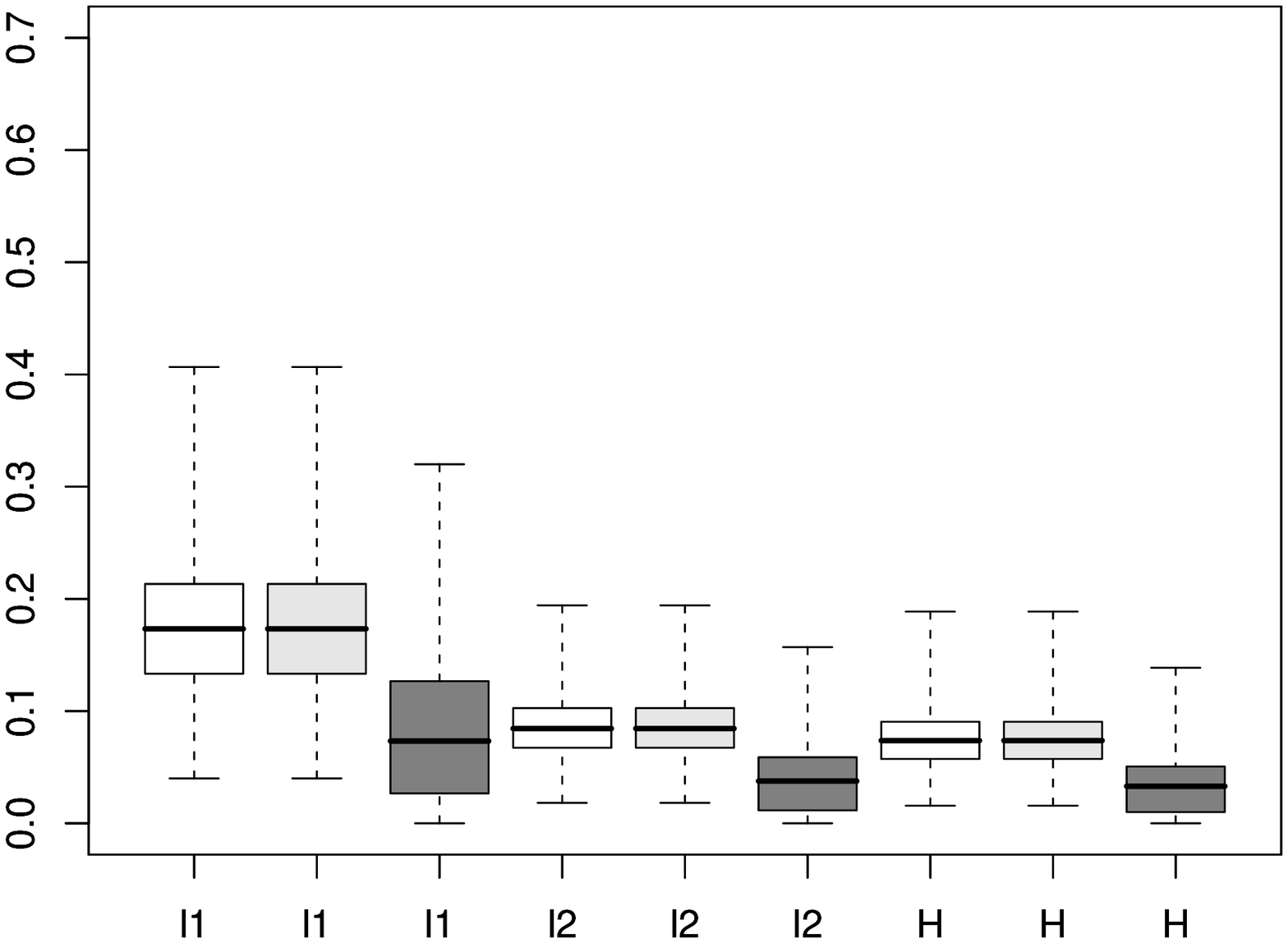}
\includegraphics[width=0.45\textwidth]{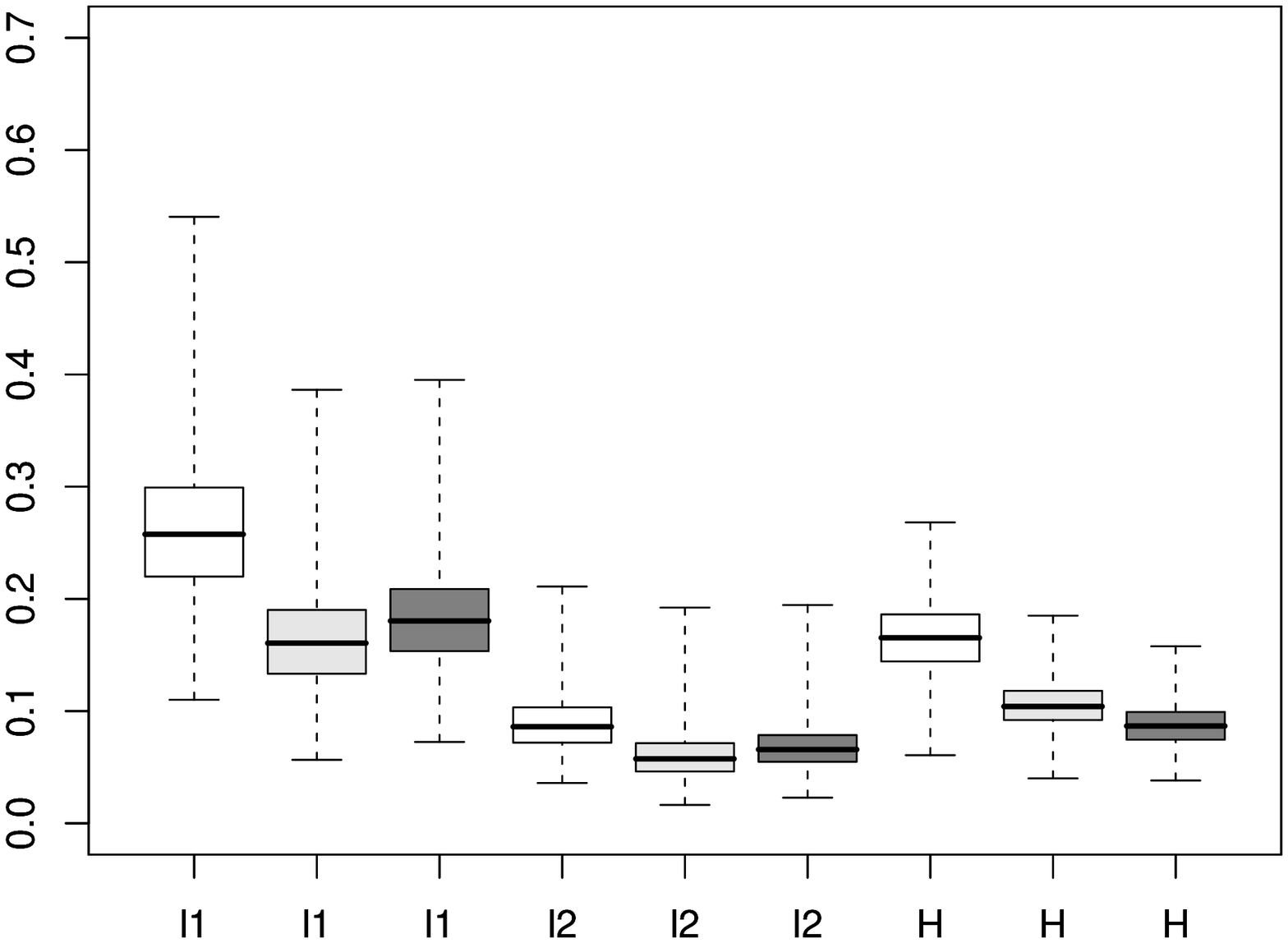}
\caption{Monte Carlo comparison of the estimators: boxplots of $m=1000$
distances of the estimators $\widehat p_n$ (white), $\widehat p_n^R$ (light grey) and $\widehat p_n^G$ (dark grey) from the truth for a sample size of $n=100$.
Left: the true distribution is the discrete uniform; and right: the true distribution
is the geometric distribution with $\theta=0.75$.}
\label{fig:comparisons_box}
\end{figure}

The above examples illustrate our main conclusion: the MLE preforms better when 
the true distribution $p$ has intervals of constancy, while the MLE and rearrangement 
estimators are competitive when $p$ is strictly monotone.  Asymptotically, it turns 
out that the MLE is superior if $p$ has any periods of constancy, while the empirical 
and rearrangement estimators are equivalent.  However, if $p$ is strictly monotone, 
then all three estimators have the same asymptotic behaviour.  

Both the MLE and monotone rearrangement estimators have been considered 
in the literature for the decreasing probability density function.  The MLE, or 
Grenander estimator, has been studied extensively, and much is known about its behaviour.  
In particular, if the true density is locally strictly decreasing, then the estimator converges 
at a rate of $n^{1/3}$, and if the true density is locally flat, then the estimator converges 
at a rate of $n^{1/2}$, cf. \cite{MR0267677, MR1745821},  and the references therein 
for a further history of the problem.  In both cases
the limiting distribution is characterized via the LCM of a Gaussian process.

The monotone rearrangement estimator for the continuous density was introduced 
by \cite{MR1486918} (see also \cite{MR2225147}). It is found by calculating the 
monotone rearrangement of a kernel density estimator (see e.g. \cite{MR1415616}).  
\cite{MR1486918} shows that this estimator also converges at the $n^{1/3}$ rate if the 
true density is locally strictly decreasing, and it is shown through Monte Carlo 
simulations that it has better behaviour than the MLE for small sample size.  
The latter is done by comparing the $L_1$ metrics for different, strictly decreasing, 
densities.  Unlike our Example \ref{ExampleGeometricSSOneHundred}, the 
Hellinger distance is not considered.

The outline of this paper is as follows.  
In Section \ref{sec:InequalitiesAndConsistency} we show that all 
three estimators are consistent.  We also establish some small sample size 
relationships between the estimators.  Section \ref{sec:LimitDistributions} is 
dedicated to the  limiting distributions of the estimators, where we show that the rate 
of convergence is $n^{1/2}$ for all three estimators.  Unlike the continuous case, 
the local behaviour of the MLE is equivalent to that of the empirical estimator 
when the true mass function is strictly decreasing.  In Section \ref{sec:limitdistributions_metrics} 
we consider the limiting behaviour of the $\ell_p$ and Hellinger distances of the estimators. 
In Section \ref{sec:mixing}, we consider the estimation of the mixing distribution $q$.  
Proofs and some technical results are given in Section \ref{sec:proofs}.  
R code to calculate the maximum likelihood estimator (i.e. $\gren(\widehat p_n^G)$) 
is available from the website of the first author: (will be provided). 

\section{Some inequalities and consistency results}
\label{sec:InequalitiesAndConsistency}

We begin by establishing several relationships between the three different estimators.

\begin{thm}\label{thm:BasicInequalities}
\begin{list}{}
        {\setlength{\topsep}{15pt}
        \setlength{\parskip}{0pt}
        \setlength{\partopsep}{0pt}
        \setlength{\parsep}{0pt}
        \setlength{\itemsep}{6pt}
        \setlength{\leftmargin}{15pt}}
\item[(i).]
Suppose that $p$ is monotone decreasing.  Then
\begin{eqnarray}
\max\{H(\widehat{p}_n^G,p), H(\widehat{p}_n^R,p)\} &\leq& H(\widehat{p}_{n},p),\label{ineq:hellinger}\\
\max\left\{||\widehat{p}_n^G-p||_k , ||\widehat{p}_n^R-p||_k\right\} &\leq& ||\widehat{p}_{n}-p||_k, \ \ 1\leq k\leq \infty.
\label{ineq:ellk}
\end{eqnarray}
\item[(ii).]
If $p$ is the uniform distribution on $\{0,\ldots, y \}$ for some integer $y$, then
\begin{eqnarray*}
H(\widehat p_n, p)&=&H(\widehat p_n^R, p),\\
||\widehat{p}_n^R-p||_k &=& ||\widehat{p}_{n}-p||_k, \ \ 1\leq k\leq \infty.
\end{eqnarray*}
\item[(iii).]
If  $\widehat{p}_n$ is monotone then $\widehat{p}_n^G = \widehat{p}_n^R = \widehat{p}_n$.
Under the discrete uniform distribution on $\{ 0, \ldots , y \}$, this occurs with probability
$$
P( \widehat{p}_{n,0} \ge \widehat{p}_{n,1} \ge \cdots \ge \widehat{p}_{n,y} ) \rightarrow \frac{1}{(y+1)!}
\qquad \mbox{as} \ \ n \rightarrow \infty .
$$
If $p$ is strictly monotone with the support of $p$ equal to $\{0, \ldots, y\}$ where $y\in\NN$, then
$$
P( \widehat{p}_{n,0} \ge \widehat{p}_{n,1} \ge \cdots \ge \widehat{p}_{n,y} ) \rightarrow 1,
$$
as $n\rightarrow\infty.$
\end{list}
\end{thm}

Let $\mc{P}$ denote the collection of all decreasing mass functions on 
$\NN$.  For any estimator $\widetilde{p}_n$ of $p \in \mc{P}$ and $k\geq 1$
let the loss function $L_k$ be defined by  
$L_k (p, \widetilde{p}_n) = \sum_{x \ge 0} | \widetilde{p}_{n,x} - p_x |^k$, 
with $L_\infty(p, \widetilde{p}_n) = \sup_{x \ge 0} | \widetilde{p}_{n,x} - p_x |$.    The risk of 
$\widetilde{p}_n$ at $p$ is then defined as
\begin{eqnarray}\label{def:risk_k}
R_k(p, \widetilde{p}_n) &=& E_p \left[ \sum_{x \ge 0} | \widetilde{p}_{n,x} - p_x |^k \right ].
\end{eqnarray}

\begin{cor}\label{cor:risk_finite}
When $k=2,$ and for any sample size $n$, it holds that
\begin{eqnarray*}
\sup_{\mc P} R_2(p, \widehat p^G_n) \leq \sup_{\mc P} R_2(p, \widehat p^R_n) = \sup_{\mc P} R_2(p, \widehat p_n).
\end{eqnarray*}
\end{cor}

\begin{figure}[htb!]
\centering
\includegraphics[width=0.43\textwidth]{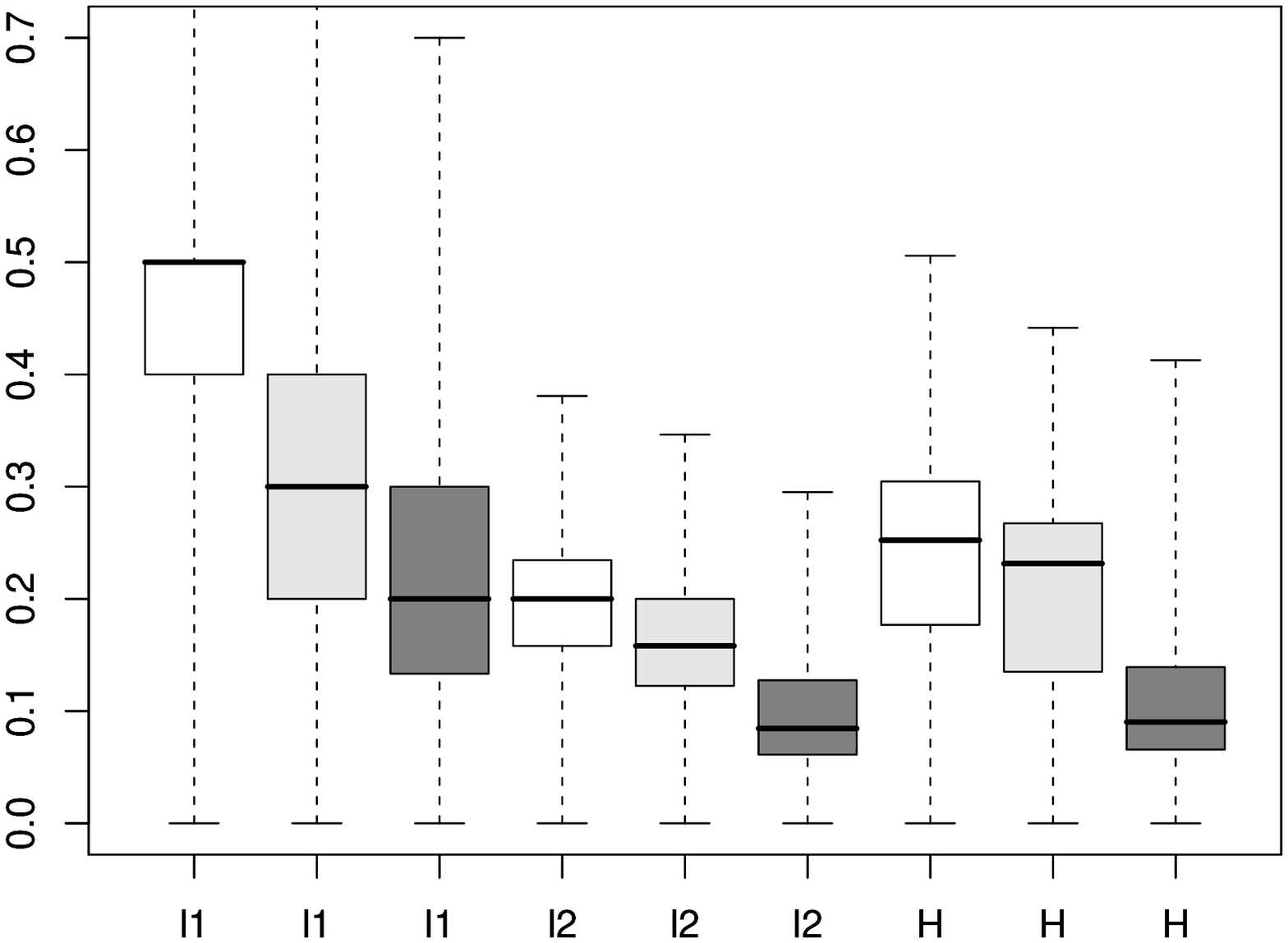}
\includegraphics[width=0.43\textwidth]{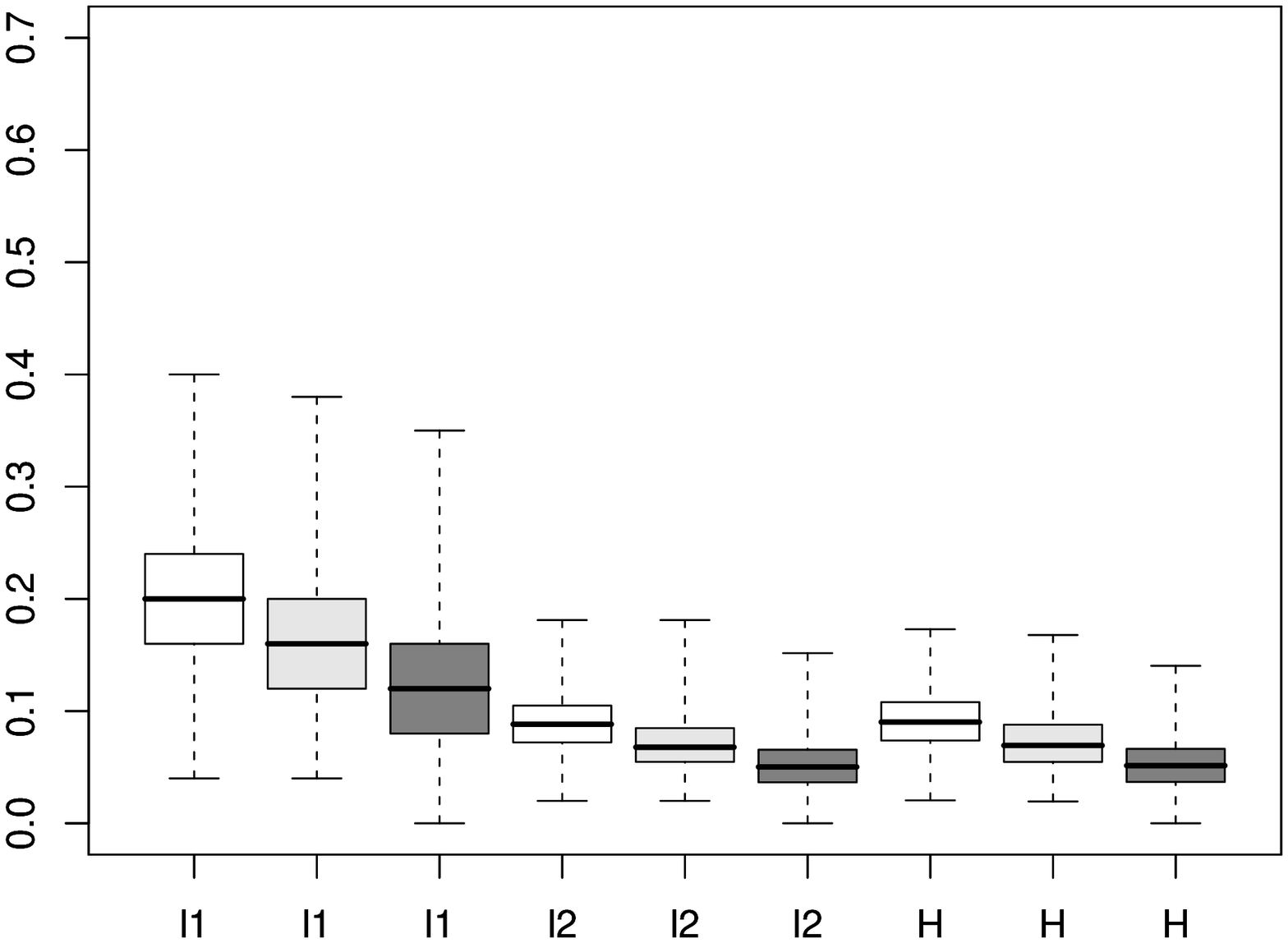}
\includegraphics[width=0.43\textwidth]{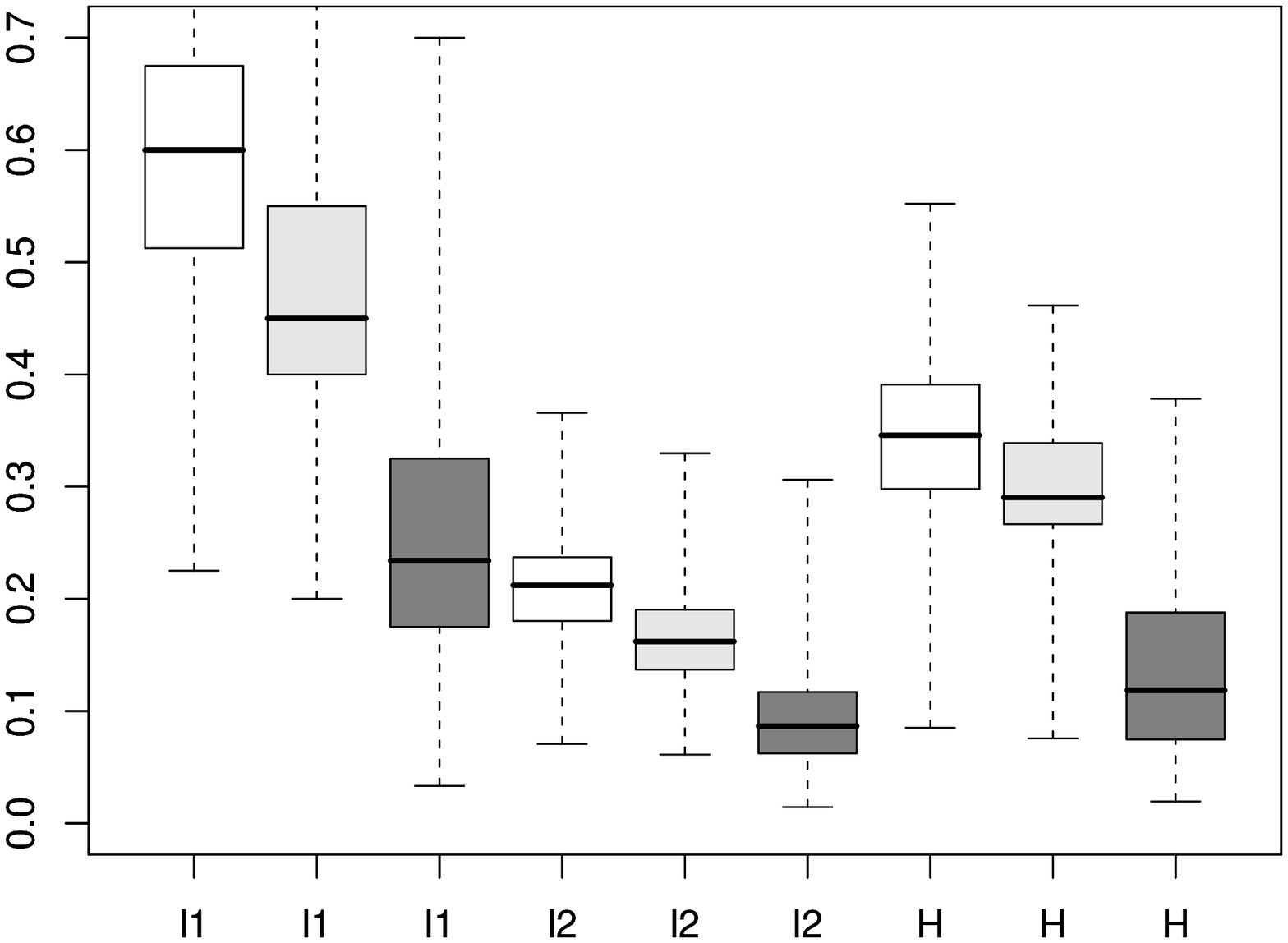}
\includegraphics[width=0.43\textwidth]{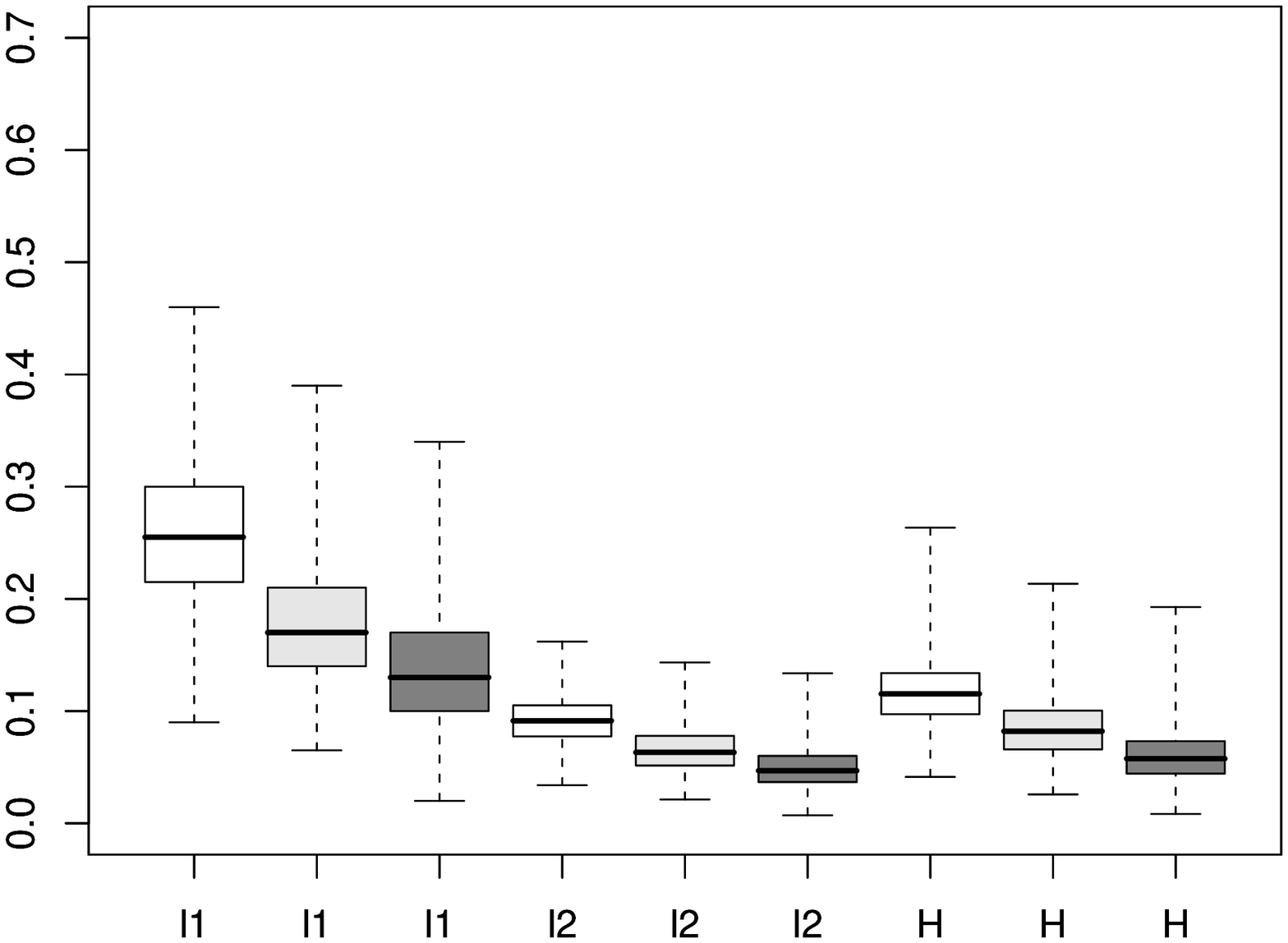}
\includegraphics[width=0.43\textwidth]{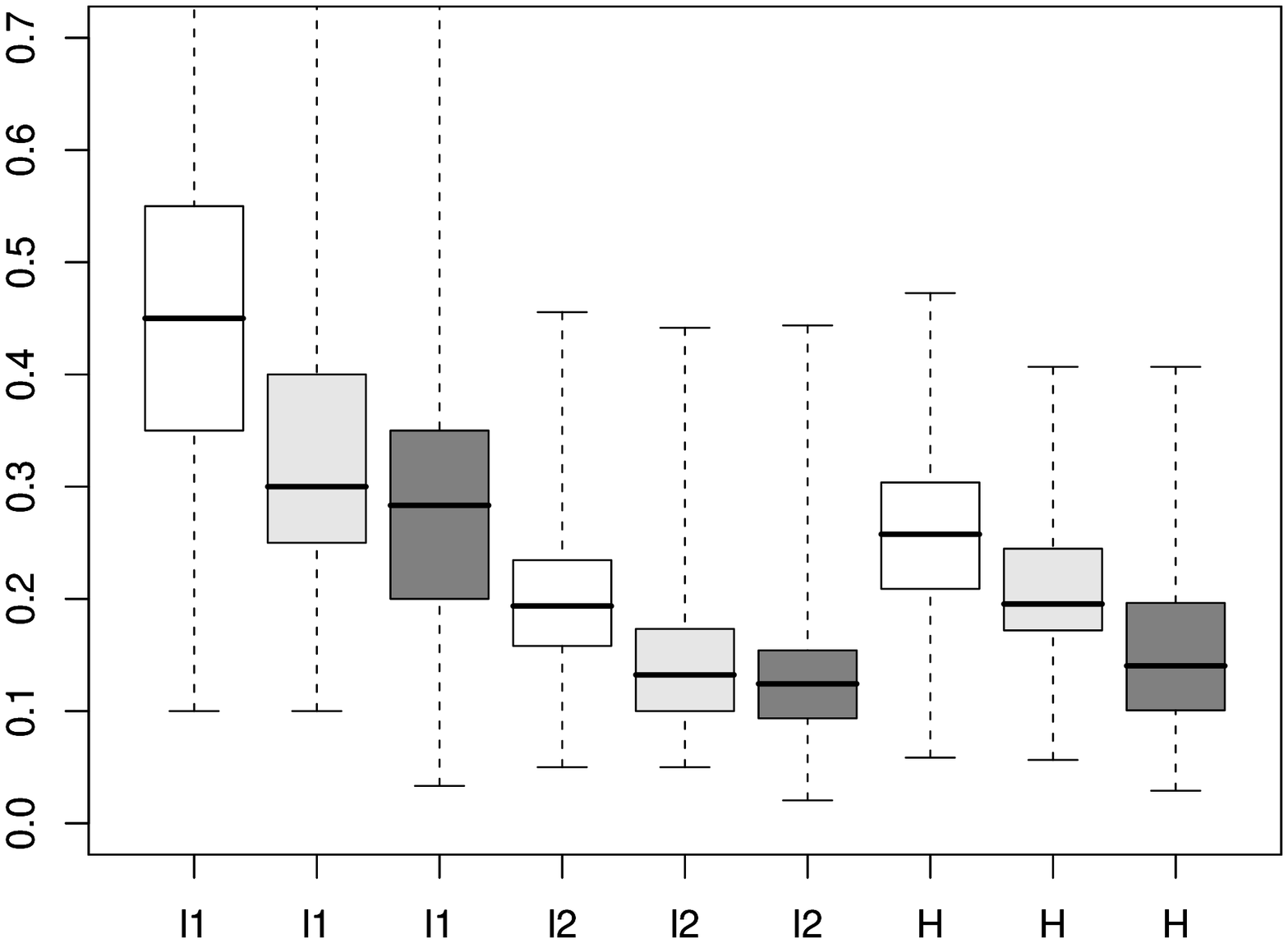}
\includegraphics[width=0.43\textwidth]{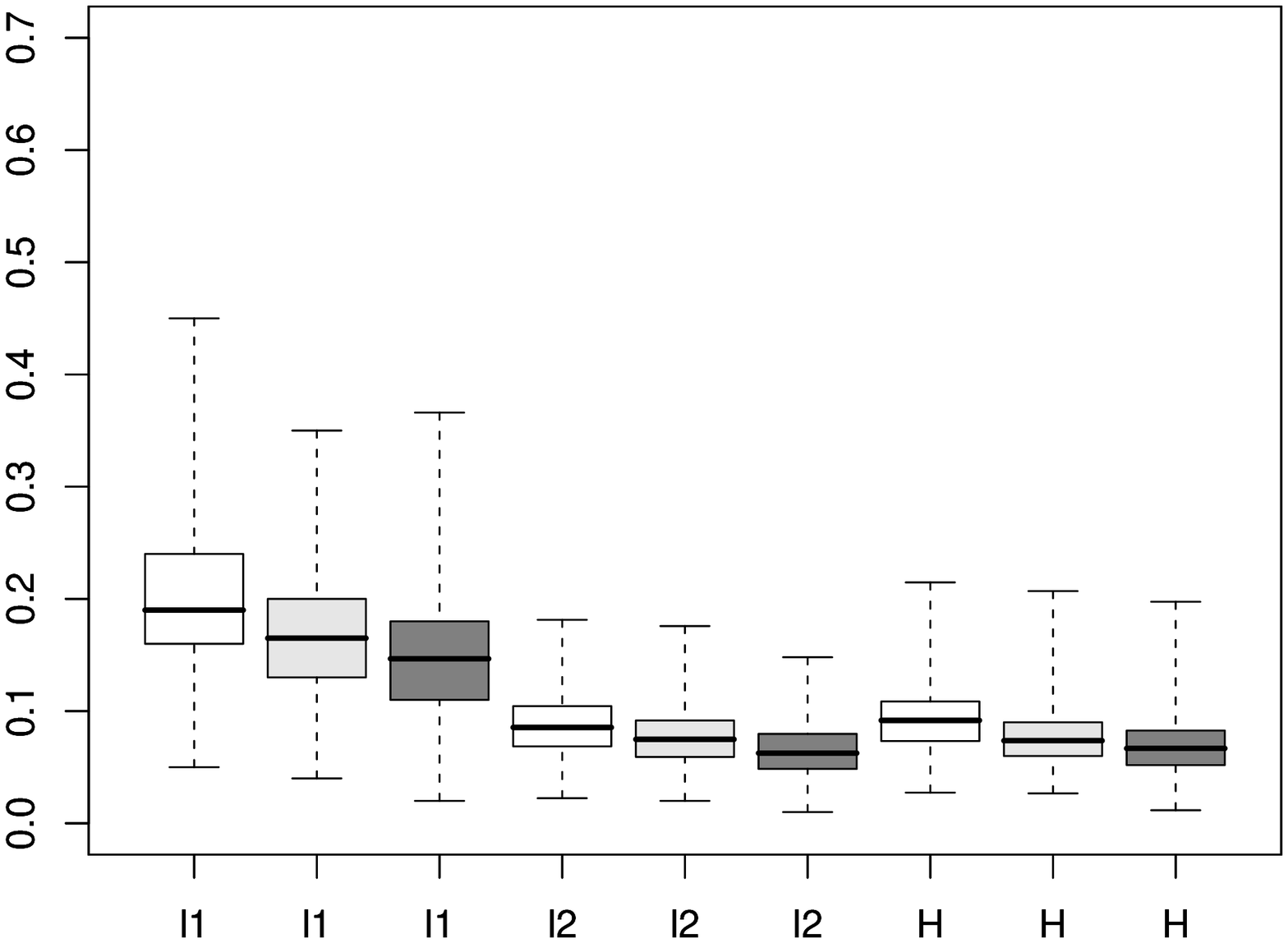}
\caption{Comparison of the estimators $\widehat p_n$ (white), $\widehat p_n^R$ (light grey) and $\widehat p_n^G$ (dark grey).} 
\label{fig:box}
\end{figure}

Based on these results, we now make the following remarks.
\newcounter{count}
\begin{list}{\arabic{count}.}
        {\usecounter{count}
        \setlength{\topsep}{6pt}
        \setlength{\parskip}{0pt}
        \setlength{\partopsep}{0pt}
        \setlength{\parsep}{0pt}
        \setlength{\itemsep}{3pt}
        \setlength{\leftmargin}{25pt}}
\item  It is always better to use a monotone estimator (either $\widehat p_n^R$ 
or $\widehat p_n^G$) to estimate a monotone mass function.
\item If the true distribution is uniform, then clearly the MLE is the better choice.
\item If the true mass function is strictly monotone, then the estimators $\widehat p_n^R$ 
and $\widehat p_n^G$ should be asymptotically equivalent.  We make this 
statement more precise in Sections \ref{sec:LimitDistributions} and~\ref{sec:limitdistributions_metrics}.  
Figure \ref{fig:comparisons_box} (right) shows that in this case $\widehat p_n^R$ 
and $\widehat p_n^G$ have about the same performance for $n=100$.
\item When only the monotonicity constraint is known about the true  $p$,  
then, by Corollary~\ref{cor:risk_finite}, $\widehat p_n^G$ is a better choice of estimator than $\widehat p_n^R.$
\end{list}

\begin{rem}
In continuous density estimation one of the most popular measures of distance 
is the $L_1$ norm, which corresponds to the $\ell_1$ norm on mass functions.  
However, for discrete mass functions, it is more natural to consider the $\ell_2$ norm.  
One of the reasons is made clear in the following sections (cf. Theorem \ref{thm:process}, 
Corollaries \ref{cor:ell2} and \ref{cor:ell1}, and Remark~\ref{rem:wrongspace}).  
The $\ell_2$ space is the smallest space in which we obtain convergence results, 
without additional assumptions on the true distribution $p$. 
\end{rem}

To examine more closely the case when the true distribution $p$ is neither uniform nor 
strictly monotone we turn to Monte Carlo simulations.    Let $p^{U(y)}$ denote the 
uniform mass function on $\{0, \ldots, y\}$.  Figure \ref{fig:box} shows boxplots of 
$m=1000$ samples of the estimators for three distributions:
\newcounter{count2}
\begin{list}{(\alph{count2}).}
        {\usecounter{count2}
        \setlength{\topsep}{6pt}
        \setlength{\parskip}{0pt}
        \setlength{\partopsep}{0pt}
        \setlength{\parsep}{0pt}
        \setlength{\itemsep}{3pt}
        \setlength{\leftmargin}{40pt}}
\item (top) $p=0.2 p^{U(3)}+0.8 p^{U(7)}$
\item (centre) $p=0.15 p^{U(3)}+0.1 p^{U(7)}+0.75 p^{U(11)}$
\item (bottom) $p=0.25 p^{U(1)}+0.2 p^{U(3)}+0.15 p^{U(5)}+0.4 p^{U(7)}$
\end{list}
On the left we have a small sample size of $n=20$, while on the right $n=100$.  
For each distribution and sample size, we calculate the three estimators 
(the estimators $\widehat p_n, \widehat p_n^R$ and $\widehat p_n^G$ are 
shown in white, light grey and dark grey, respectively) and compute their 
distance functions from the truth (Hellinger, $\ell_1$, and $\ell_2$).  Note that the 
MLE outperforms the other estimators in all three metrics, even for small sample sizes.  
It appears also that the more regions of constancy the true mass function has, the better 
the relative performance of the MLE, even for small sample size (see also Figure \ref{fig:comparisons_box}).  
By considering the asymptotic behaviour of the estimators, we are able to make 
this statement more precise in Section \ref{sec:limitdistributions_metrics}.

All three estimators are consistent estimators of the true distribution, regardless of their relative performance.

\begin{thm}\label{thm:GlobalConsistencyTheorem}
Suppose that $p$ is monotone decreasing.
Then all three estimators  $\widehat p_n, \widehat{p}_n^G$ and $\widehat{p}_{n}^R$ are consistent estimators of $p$ in the sense that
\begin{eqnarray*}
\rho(\tilde p_n, p) \rightarrow 0
\end{eqnarray*}
almost surely as $n\rightarrow\infty$ for $\tilde p_n=\widehat p_n, \widehat{p}_n^G$ and $\widehat{p}_n^R$, whenever $\rho(\tilde p,p)=H(\tilde p,p)$ or $\rho(\tilde p,p)=||\tilde p-p||_k, 1\leq k \leq \infty$.
\end{thm}

As a corollary, we obtain the following Glivenko-Cantelli type result.  

\begin{cor}\label{cor:glivenko}
Let $\widehat F_n^R(x)= \sum_{y=0}^x \widehat p_{n,y}^R$ and $\widehat F_n^G(x)= \sum_{y=0}^x \widehat p_{n,y}^G$, with $F(x)= \sum_{y=0}^x p_y.$  Then
\begin{eqnarray*}
\sup_{x\geq 0} |\widehat F_n^R(x)-F(x)|\rightarrow 0 & \mbox{ and } & 
\sup_{x\geq 0} |\widehat F_n^G(x)-F(x)|\rightarrow 0,
\end{eqnarray*}
almost surely.
\end{cor}

\section{Limiting distributions}
\label{sec:LimitDistributions}

Next, we consider the large sample behaviour of  $\widehat p_n, \widehat p_n^R$ 
and $\widehat p_n^G$.  To do this, define the fluctuation processes $Y_n, Y_n^R$, and $Y_n^G$ as
\begin{eqnarray*}
Y_{n,x}&=&\sqrt{n}(\widehat p_{n,x}-p_x),\\
Y_{n,x}^R&=&\sqrt{n}(\widehat p_{n,x}^R-p_x),\\
Y_{n,x}^G&=&\sqrt{n}(\widehat p_{n,x}^G-p_x).
\end{eqnarray*}
Regardless of the shape of $p$, the limiting distribution of $Y_n$ is well-known.  
In what follows we use the notation $Y_{n,x}\rightarrow_d Y_{n,x}$ to denote weak convergence of random variables in $\RR$ (we also use this notation for $\RR^d$), and  $Y_n \Rightarrow Y$ to denote that the \emph{process} 
$Y_n$ converges weakly to the process $Y$.  
Let $Y=\{Y_x\}_{x\in\NN}$ be a Gaussian process on the Hilbert space $\ell_2$ with 
mean zero and covariance operator $\mathbf{S}$ such that
$
\langle \mathbf{S} \, e_{(x)}, e_{(x')} \rangle = p_x\delta_{x, x'} -p_x p_{x'},
$
where $e_{(x)}$ denotes a sequence which is one at location $x$, and zero everywhere else.  
The process is well-defined, since
\begin{eqnarray*}
\mbox{trace}\, \mathbf{S} = E\left[||Y||_2^2\right]= \sum_{x\geq 0} p_x(1-p_x) < \infty.
\end{eqnarray*}
For background on Gaussian processes on Hilbert spaces we refer to 
\cite{MR0226684}.

\begin{thm}
\label{thm:l2Gaussian}
For any mass function $p,$ the process $Y_n$ satisfies $Y_n \Rightarrow Y$ in $\ell_2$.
\end{thm}

\begin{rem}
We assume that $Y$ is defined only on the support of the mass function $p$.  
That is, let $\kappa = \sup\{x: p_x >0\}$.  If $\kappa<\infty$ then $Y=\{Y_x\}_{x=0}^\kappa$.
\end{rem}

\subsection{Local Behaviour}

At a fixed point $x$ there are only two possibilities for the true mass function $p$:
either $x$ belongs to a flat region for $p$ (i.e. $p_r = \ldots = p_x = \ldots = p_s$ 
for some $r \le x \le s$), or $p$ is strictly decreasing at $x$:  $p_{x-1} > p_x > p_{x+1}$.
In the first case the three estimators exhibit different limiting behaviour,
while in the latter all three have the same limiting distribution.
In some sense, this result is not surprising.  Suppose that $x$ is such
that  $p_{x-1}>p_x>p_{x+1}.$  Then asymptotically this will hold also
for $\widehat p_n:$ $\widehat p_{n,x-k}>\widehat p_{n,x}>\widehat p_{n,x+k}$ for
$k\geq 1$ and for sufficiently large $n$.
Therefore, in the rearrangement of $\widehat p_n$ the values at $x$ will
always stay the same, i.e. $\widehat p_{n,x}^R=\widehat p_{n,x}$.
Similarly, the empirical distribution function $\FF_n$ will also be
locally concave at $x$, and therefore both $x, x-1$ will be touchpoints
of $\FF_n$ with its LCM.  This implies that $\widehat p_{n,x}^G=\widehat p_{n,x}$.

On the other hand, suppose that $x$ is such that $p_{x-1}=p_x=p_{x+1}.$
Then asymptotically the empirical density will have random order near $x$,
and therefore both re-orderings (either via rearrangement or via the LCM)
will be necessary to obtain $\widehat p_{n,x}^R$ and $\widehat p_{n,x}^G$.

\subsubsection{When $p$ is flat at $x$.}

We begin with some notation.
Let $q=\{q_x\}_{x\in\NN}$ be a sequence,
and let $r\leq s$ be positive integers.
We define $q^{(r,s)}=\{q_r, q_{r+1}, \ldots, q_{s-1}, q_s\}$
to be the $r$ through $s$ elements of $q$.  

\begin{prop}\label{prop:flat}
Suppose that for some $r,s \in \NN$ with $s-r\geq 1$ the probability mass function
$p$ satisfies $p_{r-1}> p_r = \cdots = p_s > p_{s+1}$.  Then
\begin{eqnarray*}
(Y_n^R)^{(r,s)}&\rightarrow_d& \rear(Y^{(r,s)}),\\
(Y_n^G)^{(r,s)}&\rightarrow_d& \gren(Y^{(r,s)}).
\end{eqnarray*}
\end{prop}

\begin{figure}[htb!]
\centering
\includegraphics[width=0.49\textwidth]{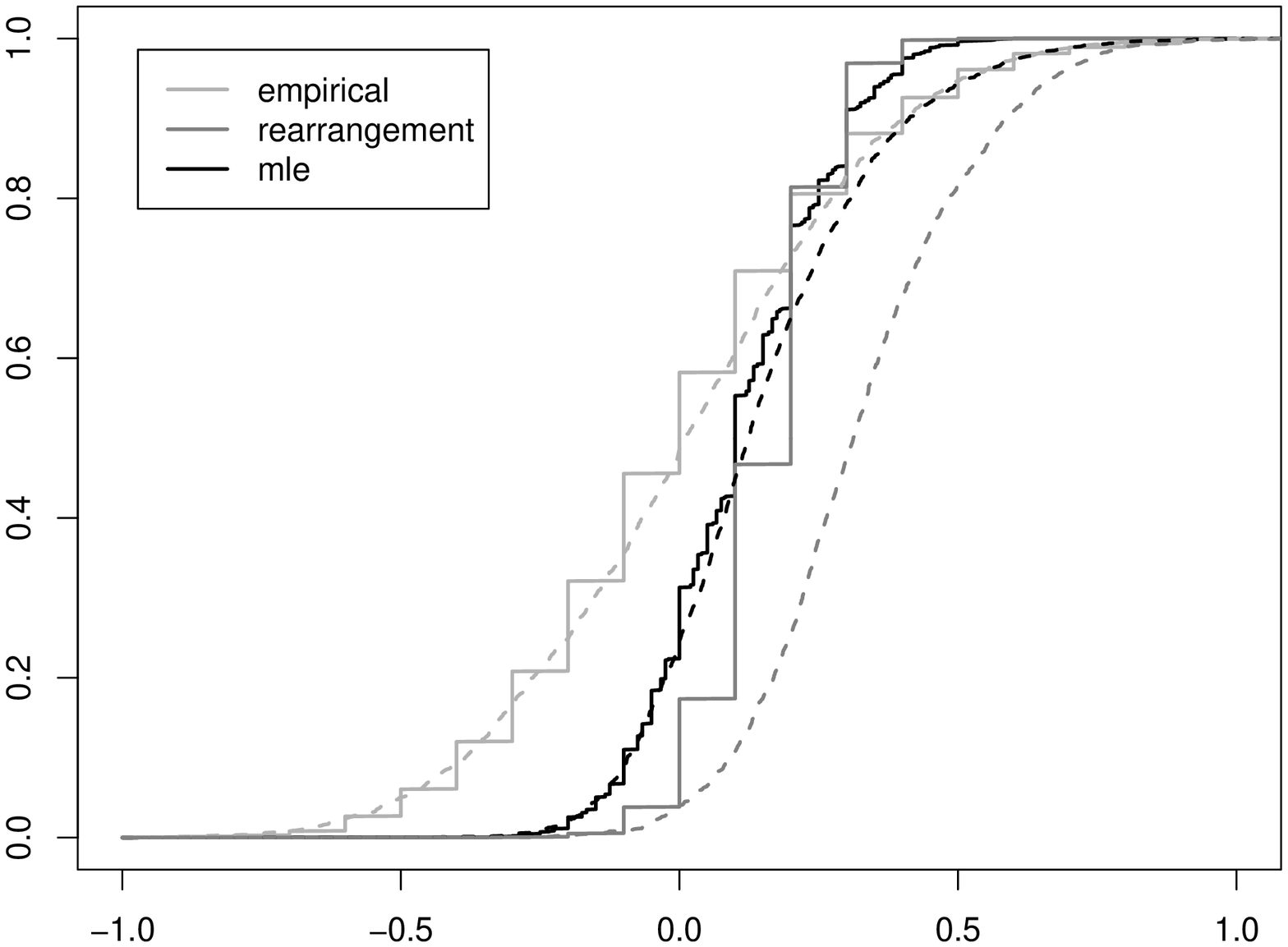}
\includegraphics[width=0.49\textwidth]{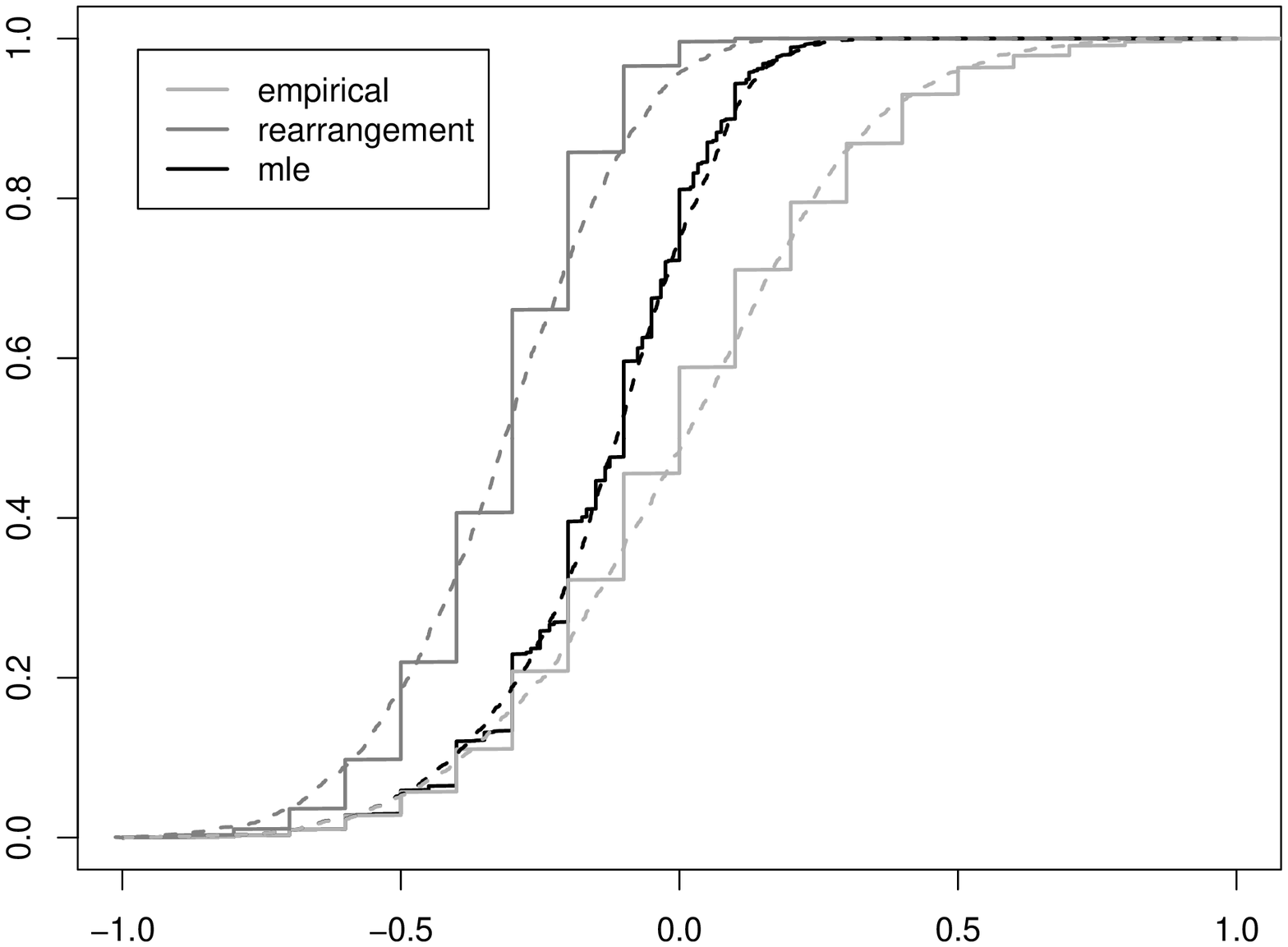}
\includegraphics[width=0.49\textwidth]{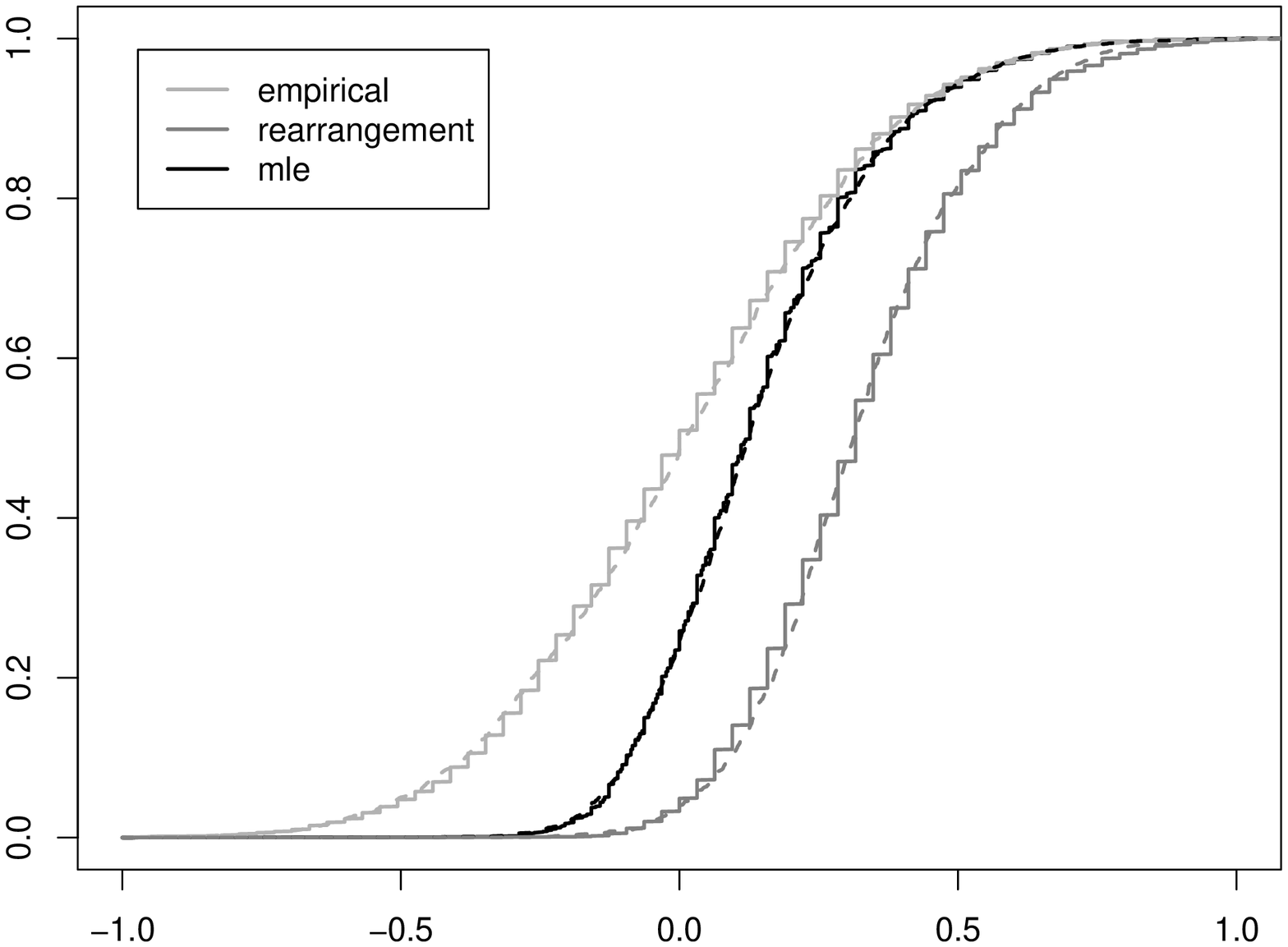}
\includegraphics[width=0.49\textwidth]{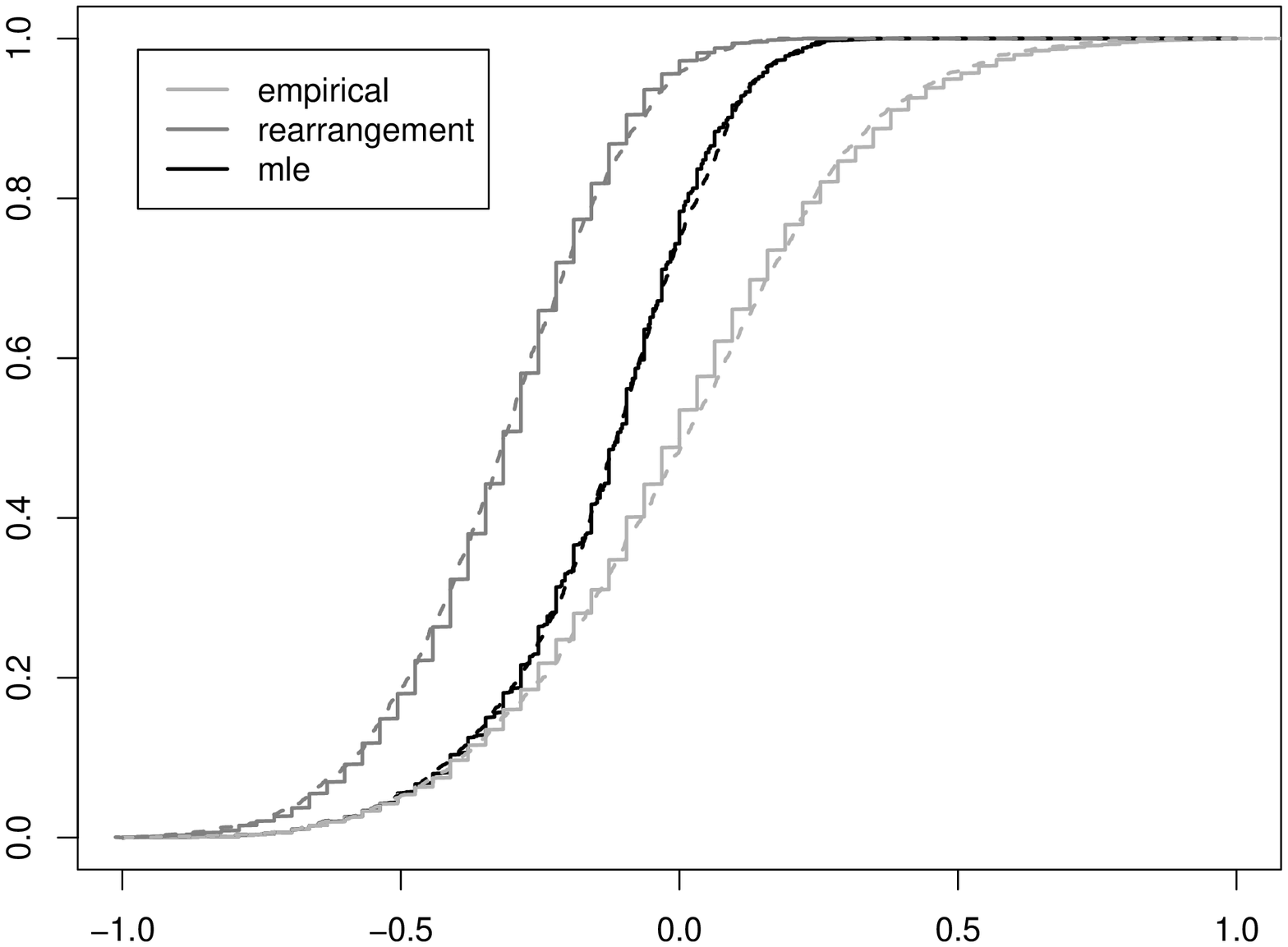}
\caption{The limiting distributions at $x=4$ (left) and at $x=7$ (right) when
$p=0.2 p^{U(3)}+0.8 p^{U(7)}$: the limiting distributions are shown 
(dashed) along with the exact distributions (solid) of $Y_n, Y_n^R, Y_n^G$ for $n=100$ (top) and $n=1000$ (bottom).}
\label{fig:limflat}
\end{figure}

The last statement of the above theorem is the discrete version of the same 
result in the continuous case due to \cite{MR1745821} for a density with locally flat regions.  
Thus, both the discrete and continuous settings have similar behaviour in this situation.  
Figure \ref{fig:limflat} shows the exact and limiting cumulative distribution functions 
when $p=0.2 p^{U(3)}+0.8 p^{U(7)}$ (same as in Figure \ref{fig:box}, top) at locations $x=4$ and $x=7.$
Note the significantly ``more discrete" behaviour of the empirical and 
rearrangement estimators in comparison with the MLE.  
Also note the lack of accuracy in the approximation at $x=4$ when $n=100$ (top left), 
which is more prominent for the rearrangement estimator.  
This occurs because $x=4$ is a boundary point, in the sense that $p_3>p_4$, 
and is therefore least resilient to any global changes in $\widehat p_n$.   
Lastly, note that the distribution functions satisfy  
$F_{Y_4}>F_{Y^G_4}>F_{Y^R_4}$ at $x=4$ while at $x=7$, 
$F_{Y^R_7}>F_{Y^G_7}>F_{Y_7}$.   It is not difficult to see that the 
relationships $Y_4^R\geq Y_4^G\geq Y_4$ and 
$Y^R_7\leq Y^G_7 \leq Y_7$ must hold from the definition of 
$(Y^R)^{(4,7)}=\rear(Y^{(4,7)})$ and $(Y^G)^{(4,7)}=\gren(Y^{(4,7)})$.

\begin{prop}\label{prop:carolandykstra}
Let $\theta = p_r = \ldots = p_s$, and let $\widetilde Y^{(r,s)}$ denote a 
multivariate normal vector with mean zero and variance matrix $\{\sigma_{i,j}\}_{i,j=r}^s$ where
\begin{eqnarray*}
\sigma_{i,j}=\alpha \delta_{i,j}- \alpha ^2,
\end{eqnarray*}
for $\alpha^{-1}=s-r+1.$  Let $Z$ be a standard normal random variable independent of 
$\widetilde Y^{(r,s)}$, and let $\tau = s-r+1$.  Then
\begin{eqnarray*}
\gren(Y^{(r,s)})&=_d& \sqrt{\frac{\theta}{\tau}}\left(\sqrt{1-\theta \tau}\ Z +\tau \ \gren(\widetilde Y^{(r,s)})\right).
\end{eqnarray*}
\end{prop}

Note that the behaviour of $\gren(Y^{(r,s)})$ and $\gren(\widetilde Y^{(r,s)})$ 
will be quite different since $\sum_{x=r}^s \widetilde Y^{(r,s)}_x=0$ almost surely, but the same is not true for $Y^{(r,s)}$.

\begin{rem}
To match the notation of \cite{MR1745821}, note that $\tau \ \gren(\widetilde Y^{(r,s)})$ is 
equivalent to the left slopes at the points $\{1, \ldots, \tau\}/\tau$ of the least 
concave majorant of standard Brownian bridge at the points $\{0, 1, \ldots, \tau\}/\tau$.  
This random vector most closely matches the left derivative of the least concave majorant 
of the Brownian bridge on $[0,1],$ which is the process that shows up in the limit for the continuous case.
\end{rem}

\subsubsection{When $p$ is strictly monotone at $x$.}

In this situation, the three estimators $\widehat p_{n,x}, \widehat p_{n,x}^R$ 
and $\widehat p_{n,x}^G$ have the same asymptotic behaviour.  This is 
considerably different than what happens for continuous densities, and 
occurs because of the inherent discreteness of the problem for probability mass functions.

\begin{prop}\label{prop:monotone}
Suppose that for some $r,s \in \NN$ with $s-r\geq 0$ the probability mass function
$p$ satisfies $p_{r-1}> p_r > \ldots > p_s > p_{s+1}$.  Then
\begin{eqnarray*}
(Y_n^R)^{(r,s)}&\rightarrow_d& Y^{(r,s)} \ \ \mbox{in} \ \ \RR^{s-r+1},\\
(Y_n^G)^{(r,s)}&\rightarrow_d& Y^{(r,s)} \ \ \mbox{in} \ \ \RR^{s-r+1}.
\end{eqnarray*}
\end{prop}

\begin{rem}
We note that the convergence results of Propositions \ref{prop:flat} 
and \ref{prop:monotone} also hold jointly.  
That is, convergence of the three processes $(Y_n^{(r,s)}, (Y_n^R)^{(r,s)}, (Y_n^G)^{(r,s)} )$ 
may also be proved jointly in $\RR^{3(s-r+1)}.$
\end{rem}

\subsection{Convergence of the Process}

We now strengthen these results to obtain convergence of the processes 
$Y_n^R$ and $Y_n^G$ in $\ell_2$.
  Note that the limit of $Y_n$ has already been stated in Theorem \ref{thm:l2Gaussian}.

\begin{thm}\label{thm:process}
Let $Y$ be the Gaussian process  defined in Theorem \ref{thm:l2Gaussian}, 
with $p$ a monotone decreasing distribution.  
Define $Y^R$ and $Y^G$ as the processes obtained by the following transforms of $Y$: 
for all periods of constancy of $p$, i.e. for all $s\geq r$ with $s-r\geq 1$ 
such that $p_{r-1}>p_r = \ldots = p_x = \ldots = p_s > p_{s+1}$ let
\begin{eqnarray*}
(Y^R)^{(r,s)}&=& \rear(Y^{(r,s)})\\
(Y^G)^{(r,s)}&=& \gren(Y^{(r,s)}).
\end{eqnarray*}
Then $Y_n^R \Rightarrow Y^R$, and $Y_n^G \Rightarrow Y^G$ in $\ell_2.$
\end{thm}

The two extreme cases, $p$ strictly monotone decreasing
and $p$ equal to the uniform distribution, may now be
considered as corollaries.  By studying the uniform case, we also study the behaviour of $Y^G$
(via Proposition \ref{prop:carolandykstra}), and therefore we consider this case in detail.

\begin{cor}\label{cor:decrasing}
Suppose that $p$ is strictly monotone decreasing.
That is, suppose that $p_x>p_{x+1}$ for all $x\geq 0$.
Then $Y^R_n \Rightarrow Y$ and $Y^G_n \Rightarrow Y$ in $\ell_2$.
\end{cor}

\subsubsection{The Uniform Distribution}

Here, the limiting distribution $Y$ is a vector of length $y+1$ having a
multivariate normal distribution with  $E[Y_x]=0$ and
$\cov(Y_x, Y_z)= (y+1)^{-1}\delta_{x,z}-(y+1)^{-2}$.

\begin{cor} \label{cor:uniformprocess}
Suppose that $p$ is the uniform probability mass function on $\{ 0, \ldots , y \}$,
where $y \in \NN$.  Then $Y_n^R \rightarrow_d \rear(Y)$ and $Y_n^G \rightarrow_d \gren(Y)$.
\end{cor}

\begin{figure}[htb!]
\begin{center}
\includegraphics[width=0.9\textwidth]{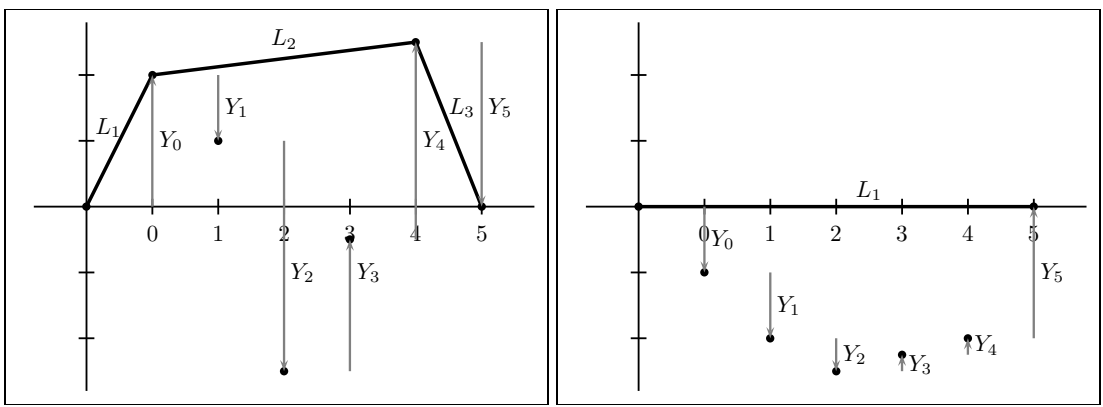}
\caption{The relationship between the limiting process $Y$ and the least
concave majorant of its partial sums for the uniform distribution on
$\{0, \ldots, 5\}$.  Left: the slopes of the lines $L_1, L_2$ and $L_3$
give the values $\gren(Y)_0$, $\gren(Y)_1 = \ldots = \gren(Y)_4$
and $\gren(Y)_5$, respectively. Right: the discrete Brownian bridge
lies entirely below zero.  Therefore, its LCM is zero, and also
$\gren(Y)\equiv0$.
This event occurs with positive probability (see also Figure \ref{fig:mle_lim}).}
\label{fig:lcm}
\end{center}
\end{figure}

The limiting process $\gren(Y)$ may also be described as follows.
Let $\UU(\cdot)$ denote the standard Brownian bridge process
on $[0,1]$, and write $U_k = \sum_{j=0}^k Y_j$ for $k=-1, \ldots, y$.
Then we have equality in distribution of
\begin{eqnarray*}
U = \{U_{-1}, U_0, \ldots, U_{y-1}, U_y\} 
\stackrel{d}{=} \left\{\UU\left(\frac{k+1}{y+1}\right): k= -1, \ldots, y\right\}.
\end{eqnarray*}
In particular we have that $U_{-1}=U_y = \sum_{j=0}^y Y_j=0.$
Thus, the process $U$ is a discrete analogue of the Brownian bridge, and 
$\gren(Y)$ is the vector of (left) derivatives of the least concave majorant 
of $\{ (j, U_j): \ j = -1, \ldots , y \}$.
Figure \ref{fig:lcm} illustrates two different realizations of the processes $Y$ and $\gren(Y).$

\begin{rem}\label{rem:prob_equal}
Note that if the discrete Brownian Bridge is itself convex,
then the limits $Y, \rear(Y)$ and $\gren(Y)$ will be equivalent.
This occurs with probability
\begin{eqnarray*}
P\left(Y\equiv \rear(Y)\equiv \gren(Y)\right)=\frac{1}{(y+1)!}.
\end{eqnarray*}
The result matches that in part (iii) of Theorem \ref{thm:BasicInequalities}.
\end{rem}

\begin{figure}[htb!]
\begin{center}
\includegraphics[width=0.64\textwidth, height=0.3\textheight]{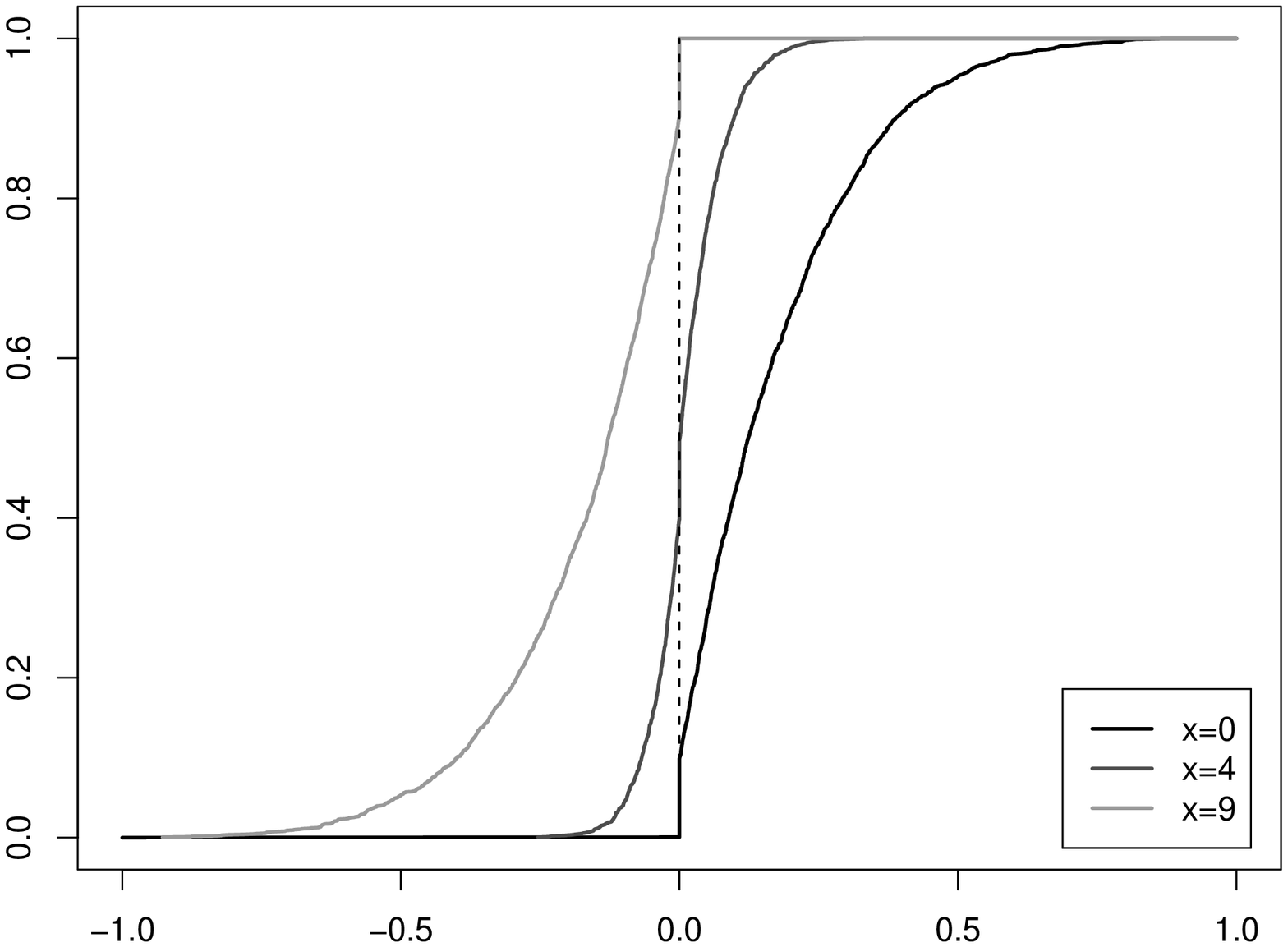}
\includegraphics[width=0.34\textwidth, height=0.3\textheight]{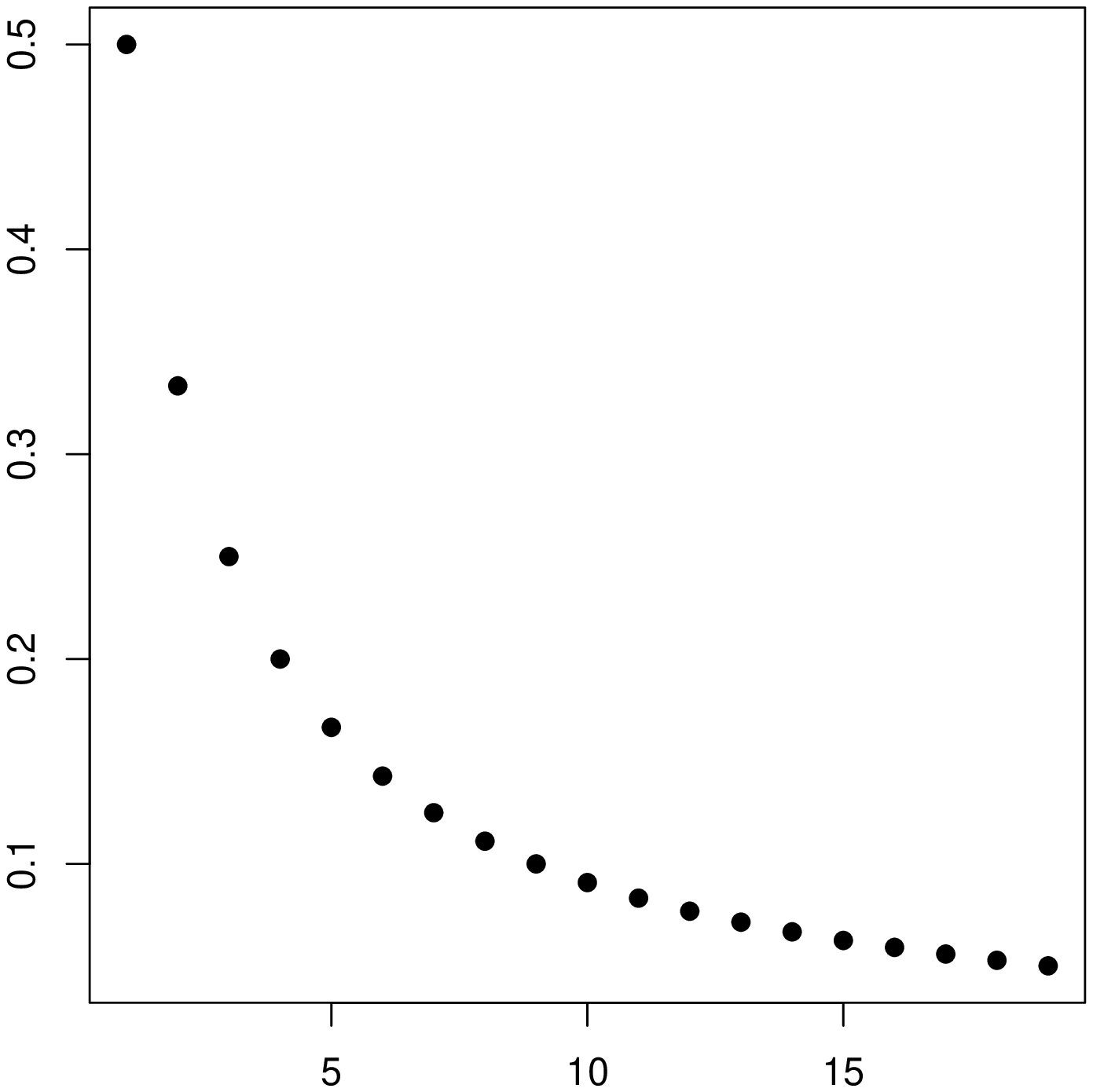}
\caption{Limiting distribution of the MLE  for the uniform case with $y=9$:
marginal cumulative distribution functions at $x=0,4,9$ (left).
The probability that $\gren(Y)\equiv0$ is plotted for different
values of $y$ (right).  For $y=9$, it is equal to 0.0999.}
\label{fig:mle_lim}
\end{center}
\end{figure}

Figure \ref{fig:mle_lim} examines the behaviour of the limiting
distribution of the MLE for several values of $x$.   Since this is
found via the LCM of the discrete Brownian bridge, it maintains
the monotonicity property in the limit:  that is, $\gren(Y)_x \geq \gren(Y)_{x+1}$.
This can easily be seen by examining the marginal distributions of
$\gren(Y)$ for different values of $x$ (Figure \ref{fig:mle_lim}, left).
For each $x$, there is a positive probability that $\gren(Y)_x=0$.
This occurs if the discrete Brownian bridge lies entirely below zero
and then the least concave majorant is identically zero, in which case
$\gren(Y)_x=0$ for all $x=0, \ldots, y$ (as in Figure \ref{fig:lcm}, right).
The probability of this event may be calculated exactly using the distribution
function of the multivariate normal.  Figure \ref{fig:mle_lim} (right),
shows several values for different $y$.

\section{Limiting distributions for the  metrics}
\label{sec:limitdistributions_metrics}

In the previous section we obtained asymptotic distribution
results for the three estimators.  To compare the estimators,
we need to also consider convergence of the Hellinger and $\ell_k$ metrics.
Our results show that $\widehat p_n^R$ and $\widehat p_n$ are asymptotically
equivalent (in the sense that the metrics have the same limit).
The MLE is also asymptotically equivalent, but if and only if $p$ is strictly monotone.
If $p$ has any periods of constancy, then the MLE has better asymptotic behaviour.
Heuristically, this happens because, by definition, $Y^G$ is a sequence of local
averages of $Y$, and averages have smaller variability.  Furthermore, the more
and larger the periods of constancy, the better the MLE performs, see, in particular,
Proposition \ref{prop:means} below.  These results quantify, for large sample
size, the observations of Figure \ref{fig:box}.

The rate of convergence of the $\ell_2$ metric is an immediate consequence
of Theorem~\ref{thm:process}.  Below, the notation $Z_1 \leq_S Z_2$ denotes
stochastic ordering: i.e. $P(Z_1>x)\leq P(Z_2>x)$ for all $x\in \RR$ (the ordering
\ is strict if both inequalities are replaced with strict inequalities).

\begin{cor}\label{cor:ell2}
Suppose that $p$ is a monotone decreasing distribution.  Then, for any $2 \leq k \leq \infty$,
\begin{eqnarray*}
\sqrt{n}||\widehat p_n -p||_k = ||Y_n||_k &\rightarrow_d& ||Y||_k,\\
\sqrt{n}||\widehat p_n^R -p||_k = ||Y_n^R||_k &\rightarrow_d& ||Y||_k,\\
\sqrt{n}||\widehat p_n^G -p||_k = ||Y_n^G||_k &\rightarrow_d& ||Y^G||_k\leq_S ||Y||_k.
\end{eqnarray*}
If $p$ is not strictly monotone, then $\leq_S$ may be replaced with $<_S$.
The above convergence also holds in expectation
(that is, $E[||Y_n||_k^k]\rightarrow E[||Y||_k^k ]$ and so forth).  Furthermore,
\begin{eqnarray*}
E\left[||Y^G||^2_2\right]\leq E\left[||Y||^2_2\right]=\sum_{x\geq 0} p_x(1-p_x),
\end{eqnarray*}
with equality if and only if $p$ is strictly monotone.
\end{cor}

Convergence of the other two metrics is not as immediate,
and depends on the tail behaviour of the distribution $p$.

\begin{cor}\label{cor:ell1}
Suppose that $p$ is such that $\sum_{x\geq 0}\sqrt{p_x}<\infty$.  Then
\begin{eqnarray*}
\sqrt{n}||\widehat p_n -p||_1 = ||Y_n||_1 &\rightarrow_d& ||Y||_1,\\
\sqrt{n}||\widehat p_n^R -p||_1 = ||Y_n^R||_1 &\rightarrow_d& ||Y||_1,\\
\sqrt{n}||\widehat p_n^G -p||_1 = ||Y_n^G||_1 &\rightarrow_d& ||Y^G||_1 \leq _S ||Y||_1.
\end{eqnarray*}
If $p$ is not strictly monotone, then $\leq_S$
may be replaced with $<_S$.  The above convergence also holds in expectation, and
\begin{eqnarray*}
E[||Y^G||_1]\leq E[||Y||_1]= \sqrt{\frac{2}{\pi}}\sum_{x\geq 0} \sqrt{p_x(1-p_x)},
\end{eqnarray*}
 with equality if and only if  $p$ is strictly monotone.
\end{cor}

Convergence of the Hellinger distance requires an even more stringent condition.
\begin{cor}\label{cor:hell}
Suppose that  $\kappa = \sup\{x: p_x >0\}<\infty$.  Then
\begin{eqnarray*}
nH^2(\widehat p_n, p)
&\rightarrow_d& \frac{1}{8}\sum_{x=0}^\kappa \frac{Y_x^2}{p_x}\\
nH^2(\widehat p_n^R, p)
&\rightarrow_d& \frac{1}{8}\sum_{x=0}^\kappa \frac{Y_x^2}{p_x}\\
nH^2(\widehat p_n^G, p)
&\rightarrow_d& \frac{1}{8}\sum_{x=0}^\kappa \frac{(Y^G_x)^2}{p_x}
                           \leq _S \frac{1}{8}\sum_{x=0}^\kappa \frac{Y_x^2}{p_x}.
\end{eqnarray*}
If $p$ is not strictly monotone, then $\leq_S$ may be replaced with $<_S$.
The distribution of $\sum_{x=0}^\kappa Y_x^2/p_x$ is chi-squared
with $\kappa$ degrees of freedom.  The above convergence also holds in expectation, and
\begin{eqnarray*}
E\left[\sum_{x=0}^\kappa \frac{(Y^G_x)^2}{p_x}\right]
\leq E\left[\sum_{x=0}^\kappa \frac{Y_x^2}{p_x}\right]= \kappa,
\end{eqnarray*}
with equality if and only if  $p$ is strictly monotone.
\end{cor}

\begin{rem}\label{rem:wrongspace}
We note that if  $\sum_{x\geq 0}\sqrt{p_x}=\infty$, then $\sum_{x\geq 0}|Y_x|=\infty$ 
almost surely, and if $\kappa=\infty,$ then $\sum_{x\geq 0} Y_x^2/p_x$ is also infinite 
almost surely.  This implies that for the empirical and rearrangement estimators, 
the conditions in Corollaries \ref{cor:ell1} and \ref{cor:hell} are also necessary 
for convergence.   The same is true for the Grenander estimator, when the true 
distribution is strictly decreasing.
\end{rem}

\begin{prop}\label{prop:means}
Let $p$ be a decreasing distribution, and write it in terms of its intervals of constancy.  That is, let
\begin{eqnarray*}
p_x &=& \theta_i \mbox{ if } x\in C_i,
\end{eqnarray*}
where where $\theta_i>\theta_{i+1}$ for all $i=1, 2, \ldots$,
and where $\{C_i\}_{i\geq 1}$ forms a partition of $\NN$.  Then
\begin{eqnarray*}
E\left[\sum_{x\geq 0} (Y^G_x)^2 \right]
&=& \sum_{i\geq 1} \sum_{j=1}^{|C_i|}\theta_i\left(\frac{1}{j}-\theta_i\right).
\end{eqnarray*}
Also, if $\kappa = \sup\{x: p_x >0\}<\infty$, then
\begin{eqnarray*}
E\left[\sum_{x=0}^\kappa \frac{(Y^G_x)^2}{p_x} \right]
&=& \sum_{i\geq 1} \sum_{j=1}^{|C_i|}\left(\frac{1}{j}-\theta_i\right).
\end{eqnarray*}
\end{prop}

This result allows us to explicitly calculate exactly how much ``better" the 
performance of the MLE is, in comparison to $Y$ and $Y^R$.    
With $\RR$--valued random variables, it is standard to compare the 
asymptotic variance to evaluate the relative efficiency of two estimators.  
We, on the other hand, are dealing with $\RR^\NN$--valued processes.  
Consider some process $W\in \RR^\NN,$ and let $\Sigma_W$ denote its 
covariance matrix (of size $\NN\times\NN$).  Then the trace norm of 
$\Sigma_W$ is equal to the expected squared $\ell_2$ norm of $W$, 
\begin{eqnarray*}
E[||W||_2^2]  = ||\Sigma_W||_{\mathrm{trace}} = \sum_{i\geq 1} \lambda_i,
\end{eqnarray*}
where $\{\lambda_i\}_{i\geq 1}$ denotes the eigenvalues of $\Sigma_W.$  
Therefore, Corollary \ref{cor:ell2} tells us that, asymptotically, $Y^G$ is more 
efficient than $Y^R$ and $Y$, in the sense that
\begin{eqnarray*}
||\Sigma_{Y^G}||_{\mathrm{trace}} \leq ||\Sigma_{Y^R}||_{\mathrm{trace}} = ||\Sigma_Y||_{\mathrm{trace}},
\end{eqnarray*}
with equality if and only if $p$ is strictly decreasing.   Furthermore, Proposition \ref{prop:means} 
allows us to calculate exactly how much more efficient $Y^G$ is  for any given mass function $p.$

Suppose that $p$ has exactly one period of constancy on $r \leq x \leq s$, 
and let $\tau=s-r+1\geq 2$.  Further, suppose that $p_x=\theta^*$ for $r\leq x\leq s.$  Then 
\begin{eqnarray*}
E\left[||Y^R||_2^2\right]-E\left[||Y^G||_2^2\right] &=& E\left[||Y||_2^2\right]-E\left[||Y^G||_2^2\right] \\
&=& \theta^*\left(\tau-\sum_{i=1}^\tau \frac{1}{i}\right).
\end{eqnarray*}
In particular, if $p$ is the uniform distribution on $\{0, \ldots, y\},$ 
then we find that $E\left[||Y^R||_2^2\right]=y/(y+1)$, whereas $E\left[||Y^G||_2^2\right]$
behaves like $\log y/(y+1),$ and is much smaller.

Note that if $p$ is strictly monotone, then we obtain
\begin{eqnarray*}
E\left[\sum_{x\geq 0} (Y^G_x)^2 \right]=
\sum_{i\geq 1} \theta_i(1-\theta_i)
= E\left[\sum_{x\geq 0} Y_x^2 \right],
\end{eqnarray*}
as required.  Also, if $p$ is the uniform probability mass function
on $\{ 0, \ldots , y \}$, we conclude that
\begin{eqnarray*}
E\left[\sum_{x=0}^y \frac{\gren(Y)_x^2}{p_x}\right] &=& \sum_{i=1}^y \frac{1}{i+1},
\end{eqnarray*}
where $\log y-0.5 <\sum_{i=1}^y (i+1)^{-1} < \log(y+1).$

Lastly, consider a distribution with bounded support, and fix $r<s$
where $p$ is strictly monotone on $\{r, \ldots, s\}$. That is, we have that  
$p_{r-1}>p_r > \ldots > p_s > p_{s+1}$. Next define $\tilde p$ by
$\tilde p_x = p_x$ for $x<r$ and $x>s$, and 
$\tilde p_x=\sum_{x=r}^s p_x/(s-r+1)$ for $x\in\{r, \ldots, s\}.$
Then the difference in the expected Hellinger metrics under the two distributions is
\begin{eqnarray*}
E_p\left[\sum_{x=0}^\kappa \frac{(Y^G_x)^2}{p_x} \right]
    -E_{\tilde p}\left[\sum_{x=0}^\kappa \frac{(Y^G_x)^2}{\tilde p_x} \right]
    &=& \tau -\sum_{j=1}^\tau \frac{1}{j}
\end{eqnarray*}
where $\tau=s-r+1.$  
Therefore, the longer the intervals of constancy in a distribution, 
the better the performance of the MLE.

\begin{rem}\label{rem:pointwise_ineq}
From Theorem 1.6.2 of \cite{MR961262} it follows that for any $x\geq 0$
\begin{eqnarray*}
E[(Y_{x}^G)^2] &\leq & E[Y_x^2]=p_x(1-p_x).
\end{eqnarray*}
This result may also be proved using the method used to show Proposition \ref{prop:means}.  
Note that this pointwise inequality does not hold in general for $Y^G$ replaced with $Y^R.$
\end{rem}

Corollaries 4.1 and 4.2 then translate into statements concerning the limiting risks of 
the three estimators $\widehat{p}_n$, $\widehat{p}_n^R$, and 
$\widehat{p}_n^G$ as follows, where the risk was defined in \eqref{def:risk_k}.   
In particular, we see that, asymptotically, both $\widehat p_n^R$ and 
$\widehat p_n$ are inadmissible, and are dominated by the maximum likelihood estimator $\widehat p_n^G.$

\begin{cor}\label{cor:LimitingRisks}
For any $2 \leq k \leq \infty$, and any $p\in\mc P$, the class of decreasing probability mass functions on $\NN,$
\begin{eqnarray*}
&& n^{k/2} R_{k} (p, \widehat{p}_n) \rightarrow E [\| Y \|_k^k], \nonumber \\ 
&& n^{k/2} R_{k} (p, \widehat{p}_n^R) \rightarrow E[\| Y \|_k^k], \nonumber \\ 
&& n^{k/2} R_{k} (p, \widehat{p}_n^G) \rightarrow E[ \| Y^G \|_k^k] \leq E[ \| Y \|_k^k ].
\label{GrenanderWinsForKin2Infty}
\end{eqnarray*}
The inequality in the last line is strict if $p$ is not strictly monotone.   
The statements also hold for $k=1$ under the additional hypothesis that $\sum_{x \ge 0} \sqrt{p}_x < \infty$. 
\end{cor}

\section{Estimating the mixing distribution}\label{sec:mixing}

Here, we consider the problem of estimating the mixing distribution
$q$ in \eqref{line:mix_form}.  This may be done directly via the
estimators of $p$ and the formula \eqref{line:mix_recover}.
Define the estimators of the mixing distribution as follows
\begin{eqnarray*}
\widehat q_{n,x} &=& -(x+1)\triangle \widehat p_{n,x},\\
\widehat q_{n,x}^R &=& -(x+1)\triangle \widehat p_{n,x}^R,\\
\widehat q_{n,x}^G &=& -(x+1)\triangle \widehat p_{n,x}^G.
\end{eqnarray*}
Each of these estimators sums to one by definition, however
$\widehat q_n$ is not guaranteed to be positive.   The main results
of this section are consistency and $\sqrt{n}$--rate of convergence of these estimators.

\begin{thm}\label{thm:mixing_consistent}
Suppose that $p$ is monotone decreasing and satisfies
$\sum_{x\geq 0} x p_x <\infty$. Then all three estimators
$\widehat q_n, \widehat{q}_n^G$ and $\widehat{q}_{n}^R$ are consistent
estimators of $q$ in the sense that
\begin{eqnarray*}
\rho(\tilde q_n, q) \rightarrow 0
\end{eqnarray*}
almost surely as $n\rightarrow\infty$ for $\tilde q_n=\widehat q_n, \widehat{q}_n^G$
and $\widehat{q}_n^R$, whenever $\rho(\tilde q,q)=H(\tilde q,q)$ or
$\rho(\tilde q,q)=||\tilde q-q||_k, 1\leq k \leq \infty$.
\end{thm}

To study the rates of convergence we define the the fluctuation
processes $Z_n, Z_n^R$, and $Z_n^G$ as
\begin{eqnarray*}
Z_{n,x}&=&\sqrt{n}(\widehat q_{n,x}-q_x),\\
Z_{n,x}^R&=&\sqrt{n}(\widehat q_{n,x}^R-q_x),\\
Z_{n,x}^G&=&\sqrt{n}(\widehat q_{n,x}^G-q_x),
\end{eqnarray*}
with limiting processes  defined as 
\begin{eqnarray*}
Z_{x}&=& -(x+1)(Y_{x+1}-Y_x),\\
Z_{x}^R&=&-(x+1)(Y_{x+1}^R-Y_x^R),\\
Z_{x}^G&=&-(x+1)(Y_{x+1}^G-Y_x^G).
\end{eqnarray*}

\begin{thm}\label{thm:mixing_rate}
Suppose that $p$ is such that $\kappa = \sup\{x\geq 0: p_x>0\} <\infty.$  Then 
$Z_n \Rightarrow Z, Z_n^R \Rightarrow Z^R$ and $Z_n^G \Rightarrow Z^G.$
Furthermore, for any $k\geq 1$, $||Z_n||_k \rightarrow_d ||Z||_k, ||Z_n^R||_k \rightarrow_d ||Z^R||_k$
and $||Z_n^G||_k \rightarrow_d ||Z^G||_k$, and these convergences also hold in expectation.    
Also, $n H^2 (\widehat q_n, q) \rightarrow_d \sum_{x=0}^k Z_x^2/q_x$, 
$nH^2 (\widehat q_n, q) \rightarrow_d  \sum_{x=0}^k (Z_x^R)^2/q_x$
and $nH^2 (\widehat q_n, q) \rightarrow_d  \sum_{x=0}^k (Z_x^G)^2/q_x$, 
and these again also hold in expectation.
\end{thm}

As before, we have asymptotic equivalence
of all three estimators if $p$ is strictly decreasing (cf. Corollary \ref{cor:decrasing}).
To determine the relative behaviour of the estimators $\widehat q_n^R$ and
$\widehat q_n^G$ we turn to simulations.  Since $\widehat q_n$ is not guaranteed
to be a probability mass function (unlike the other two estimators), we exclude it
from further consideration.

\begin{figure}[p]
\centering
\includegraphics[width=0.3\textwidth]{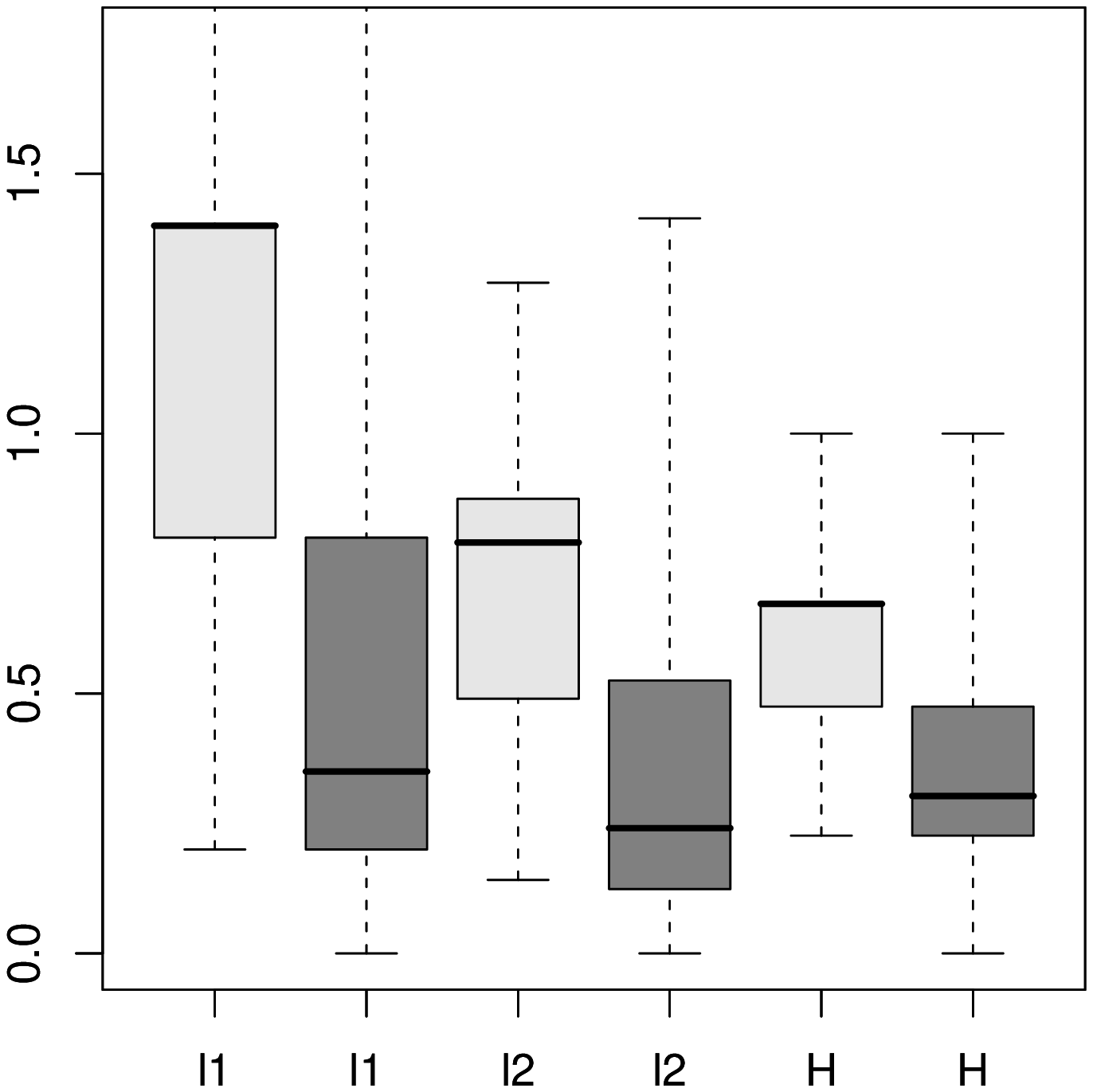}
\includegraphics[width=0.3\textwidth]{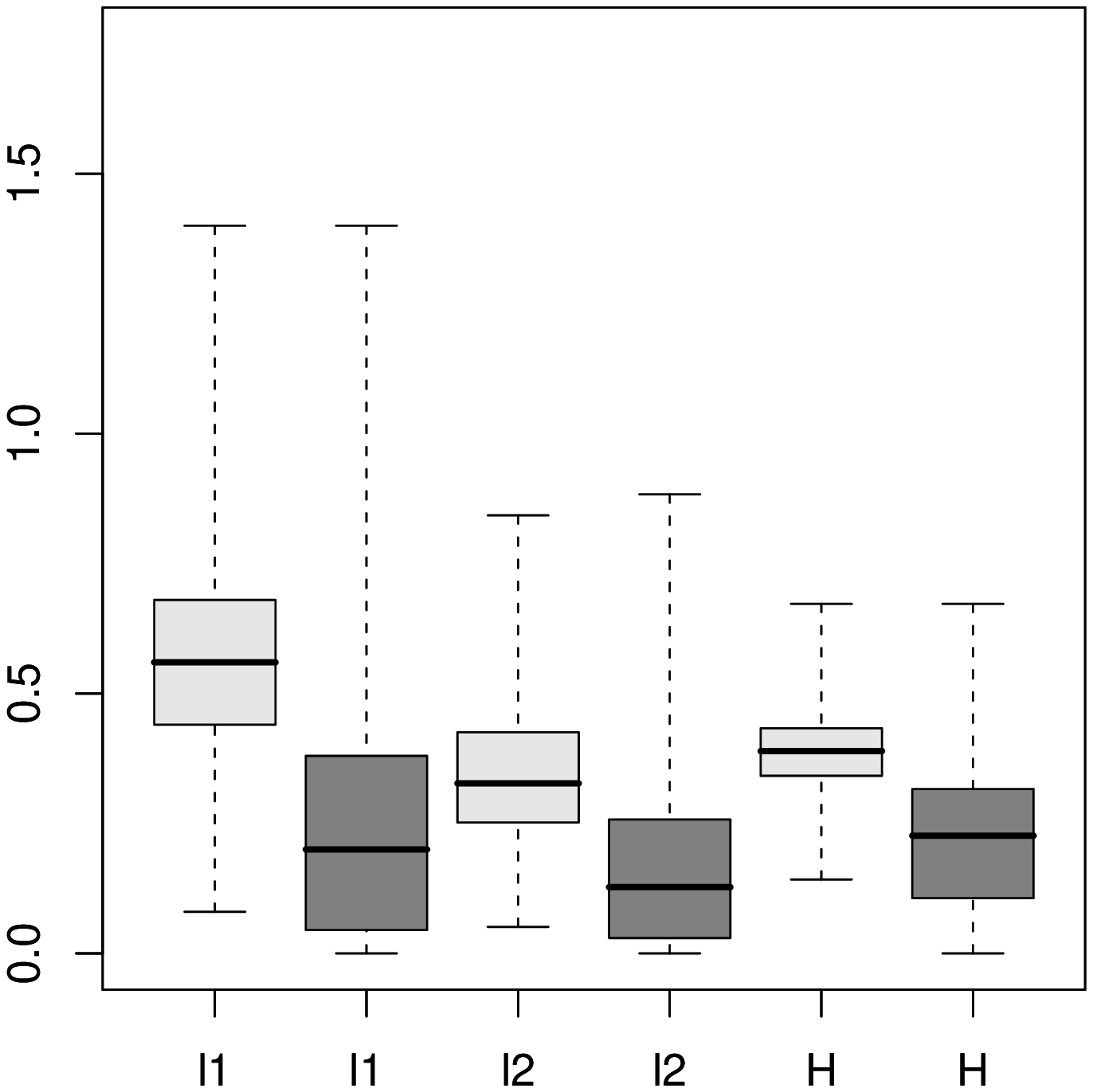}
\includegraphics[width=0.3\textwidth]{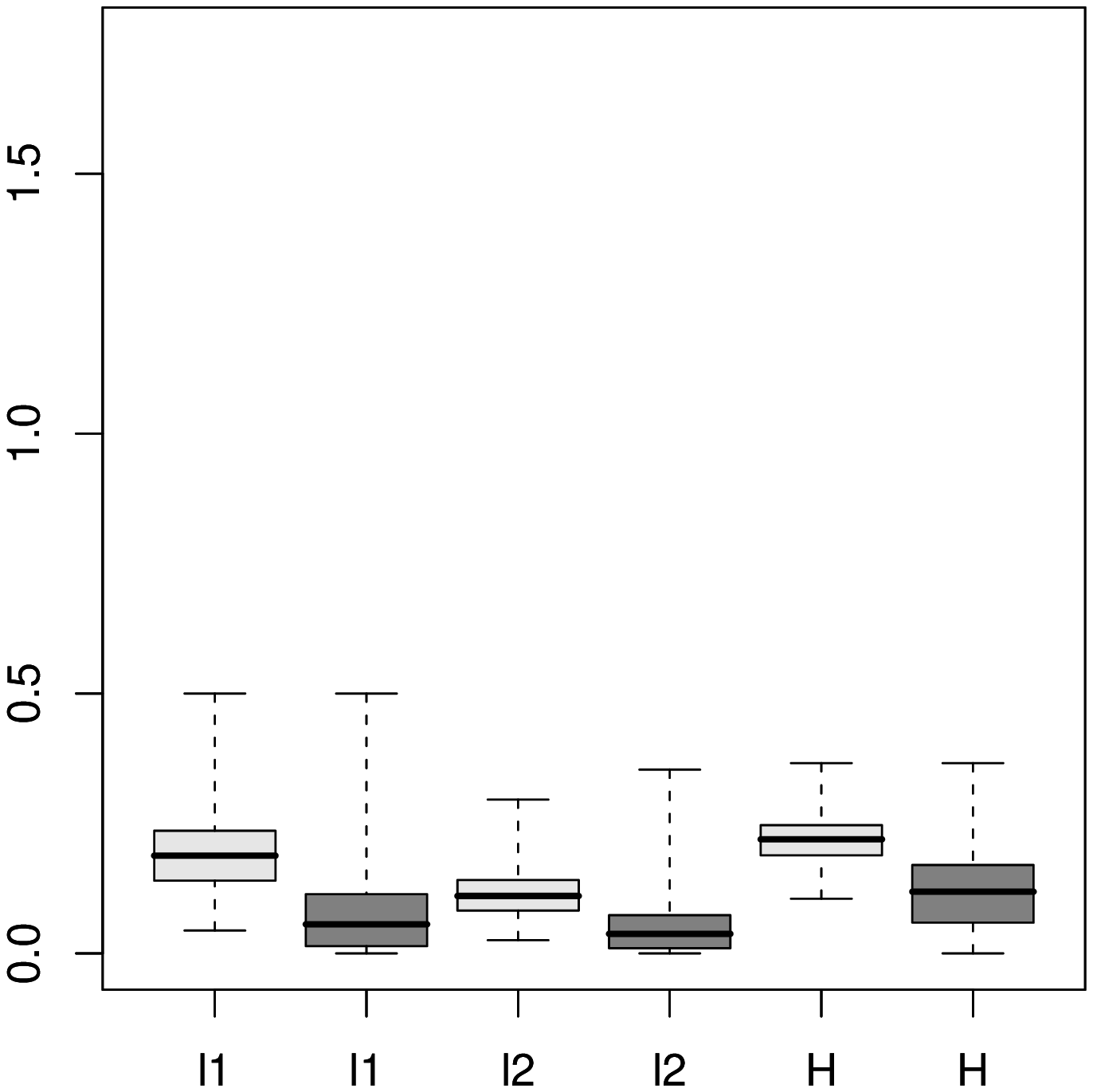}
\includegraphics[width=0.3\textwidth]{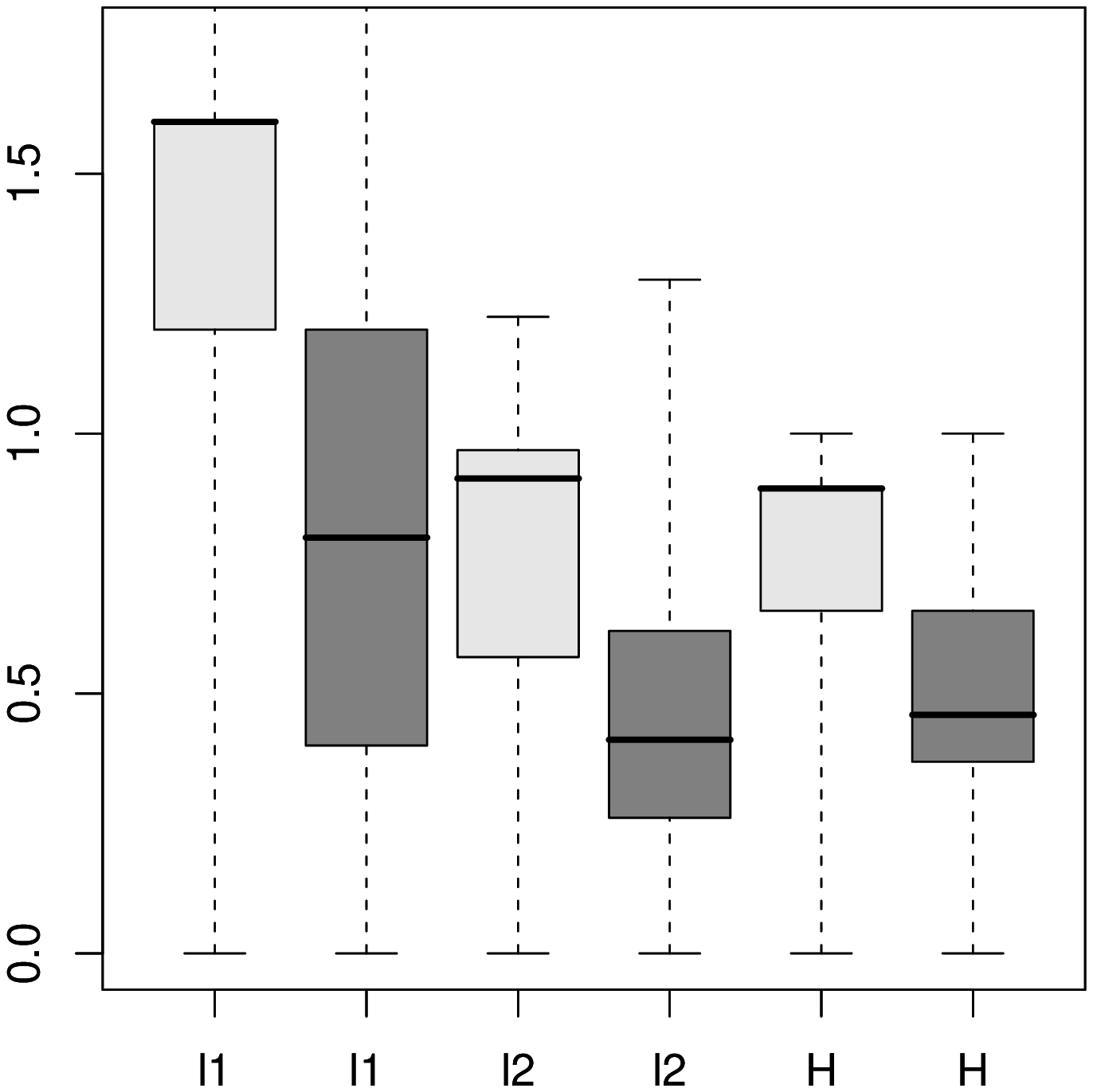}
\includegraphics[width=0.3\textwidth]{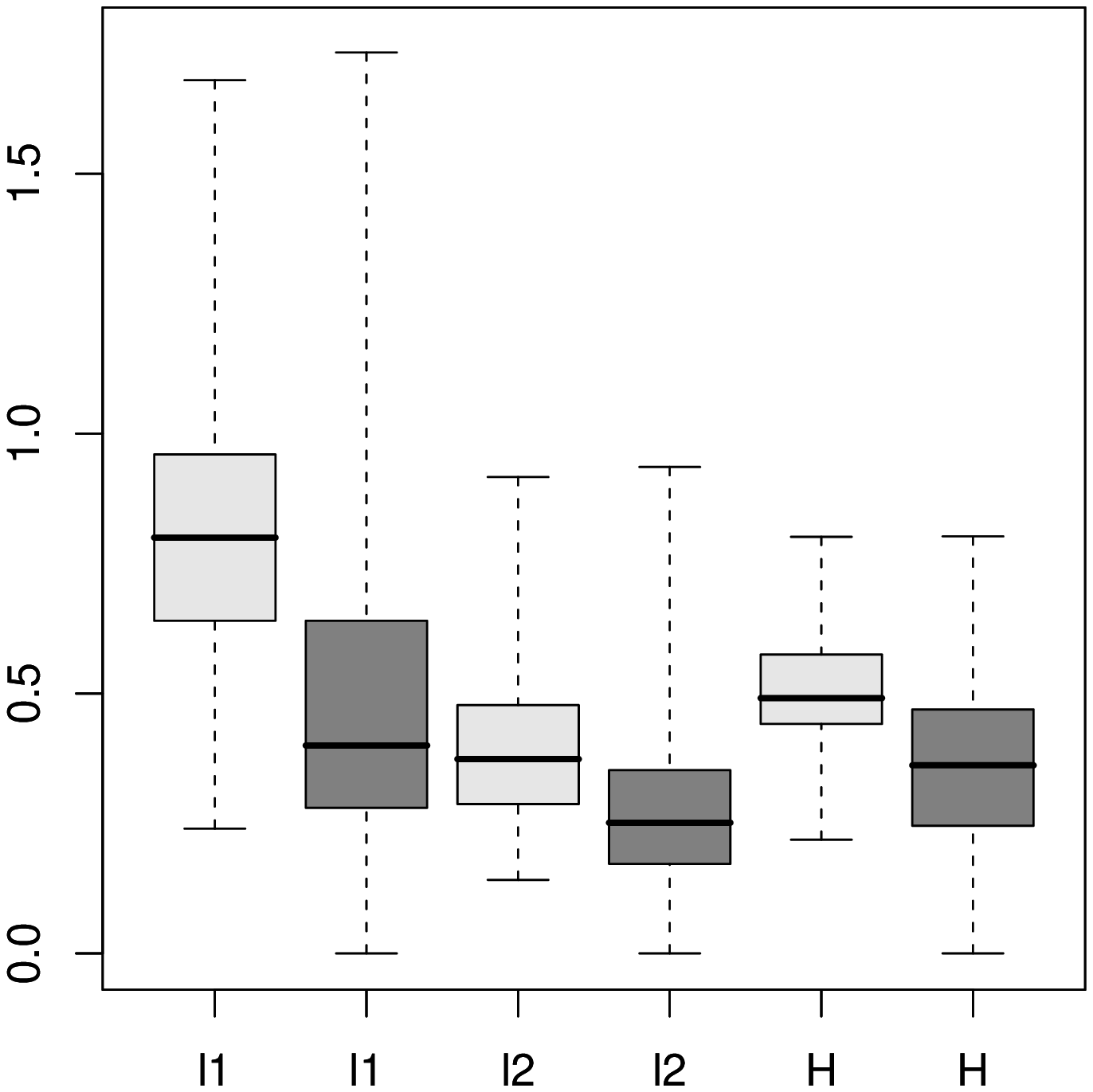}
\includegraphics[width=0.3\textwidth]{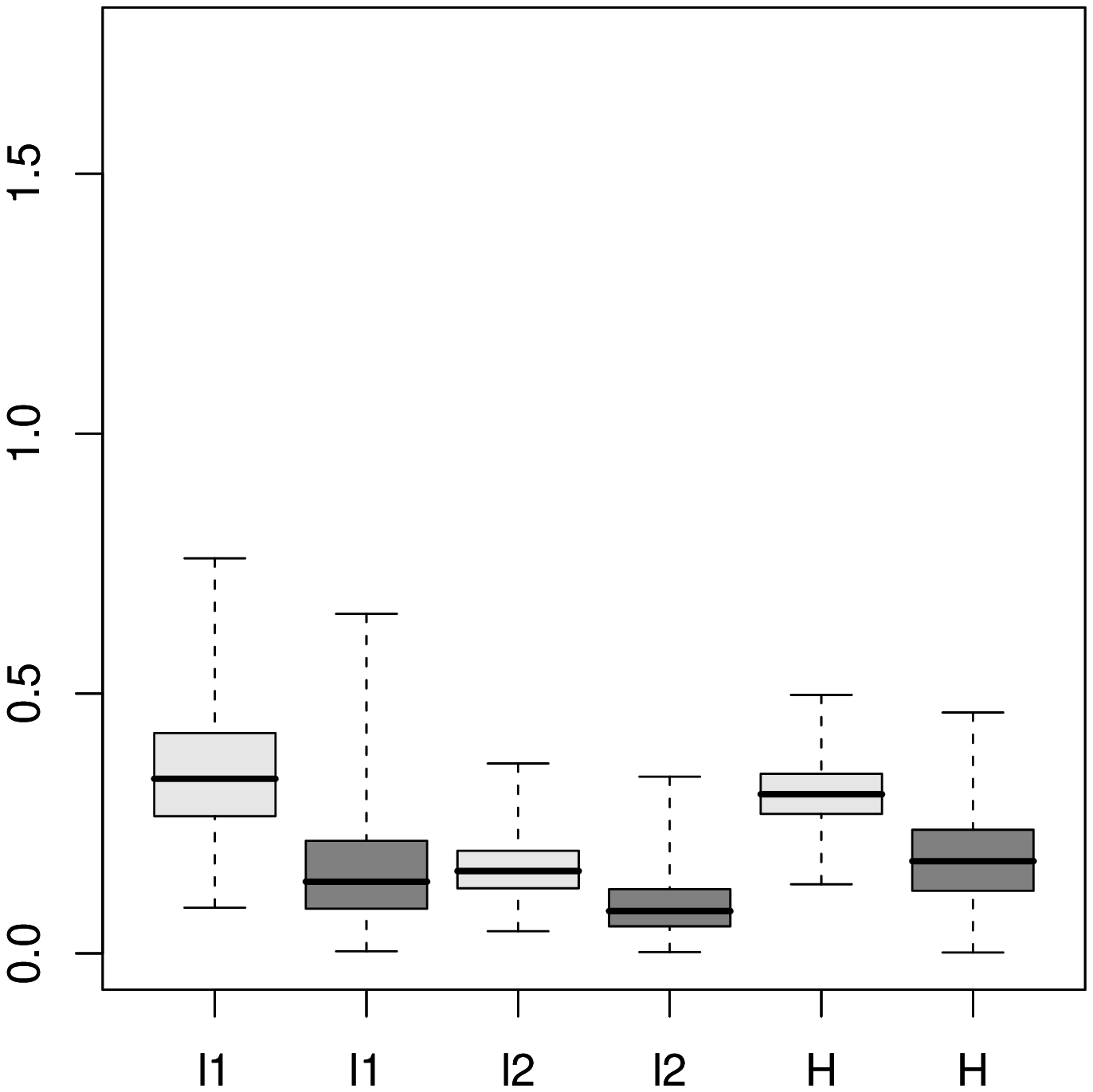}
\includegraphics[width=0.3\textwidth]{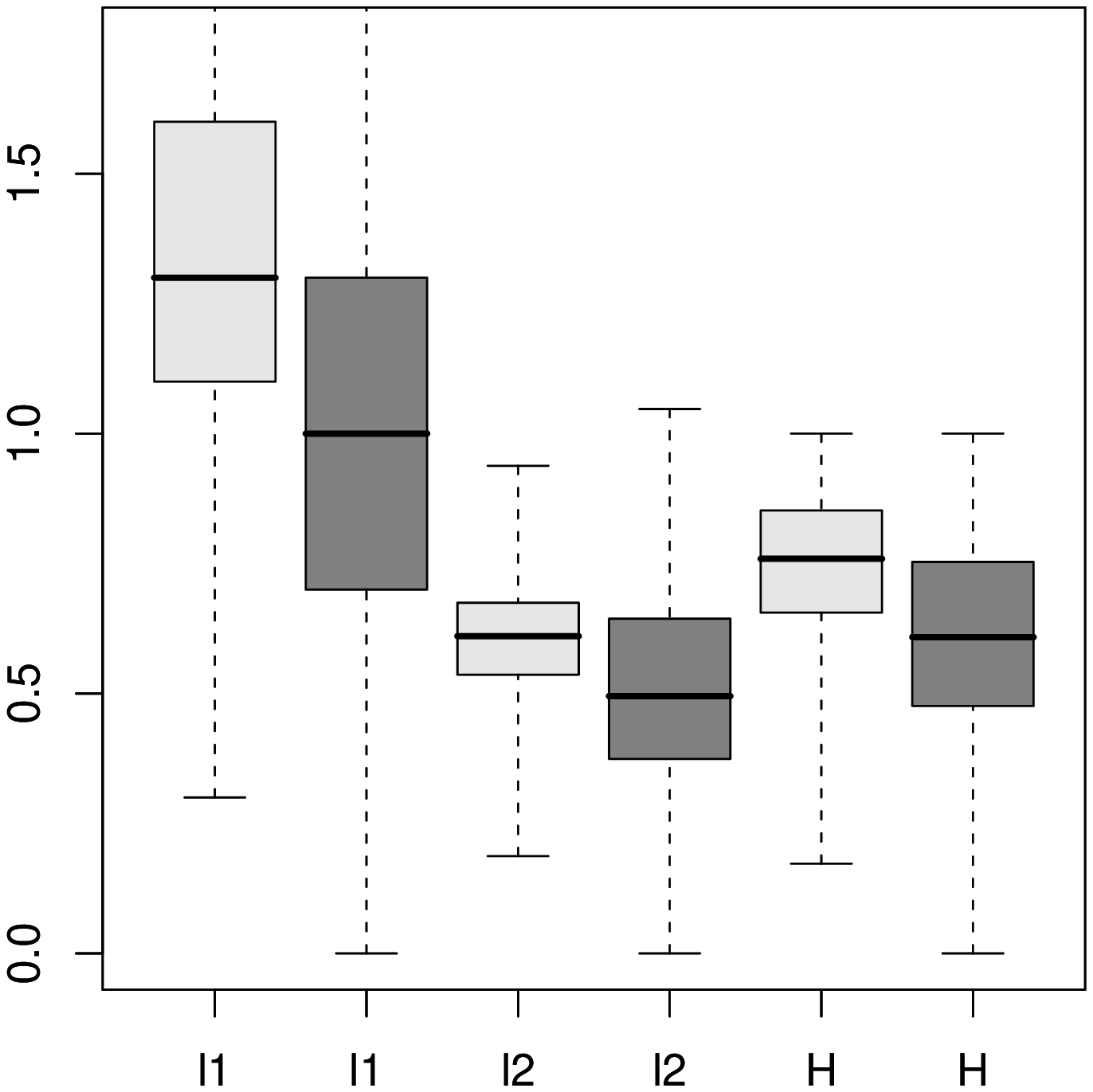}
\includegraphics[width=0.3\textwidth]{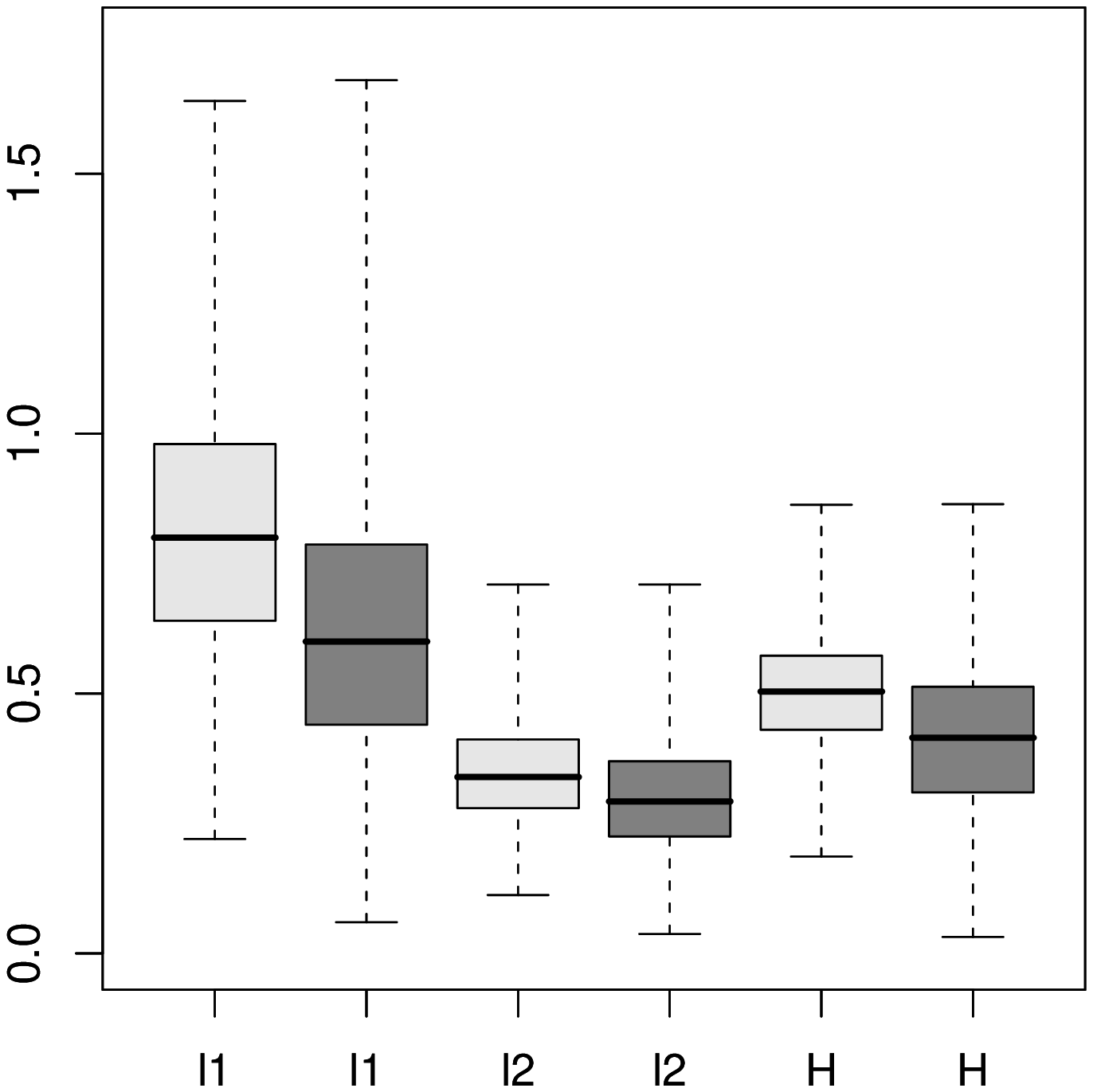}
\includegraphics[width=0.3\textwidth]{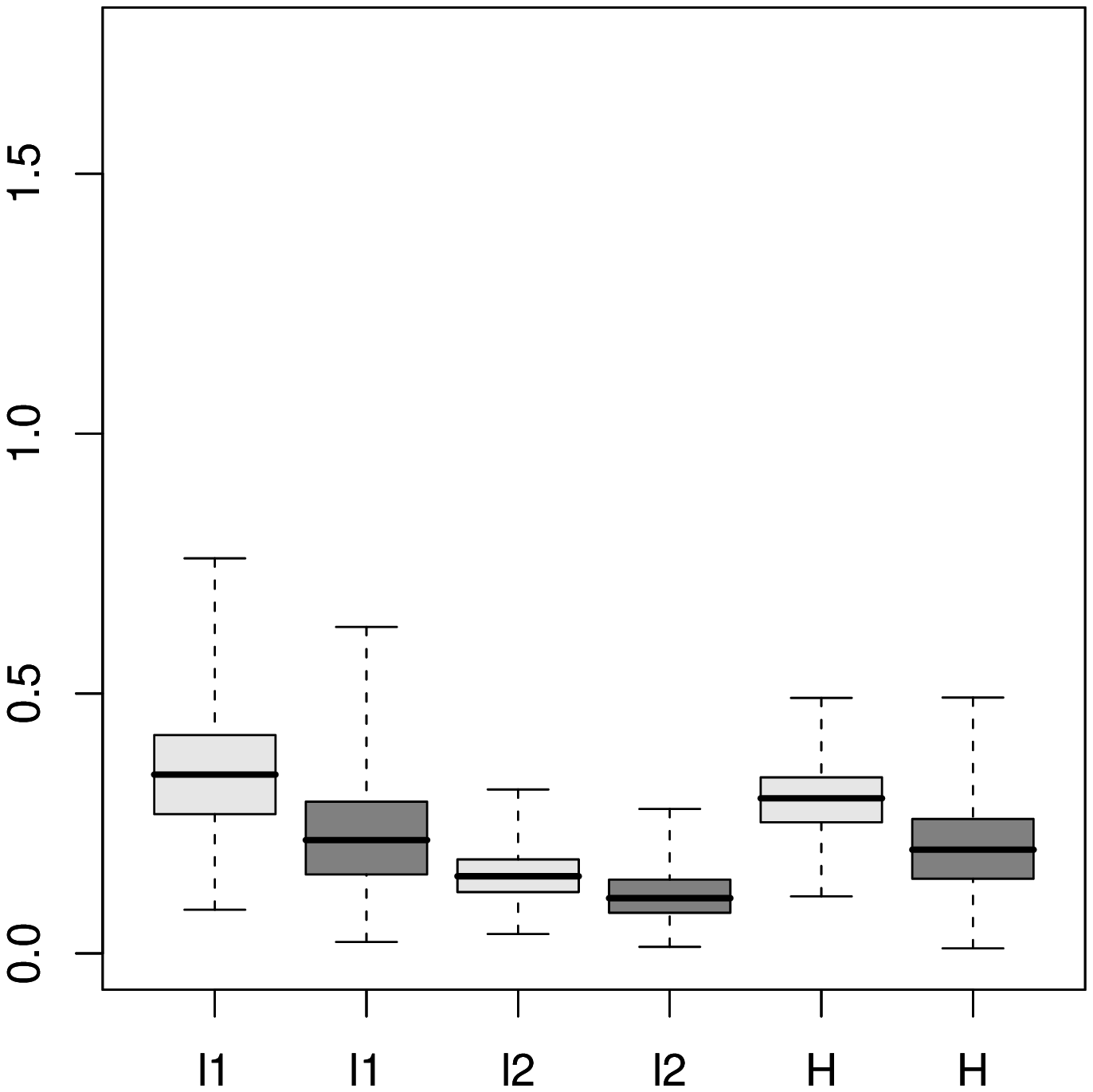}
\includegraphics[width=0.3\textwidth]{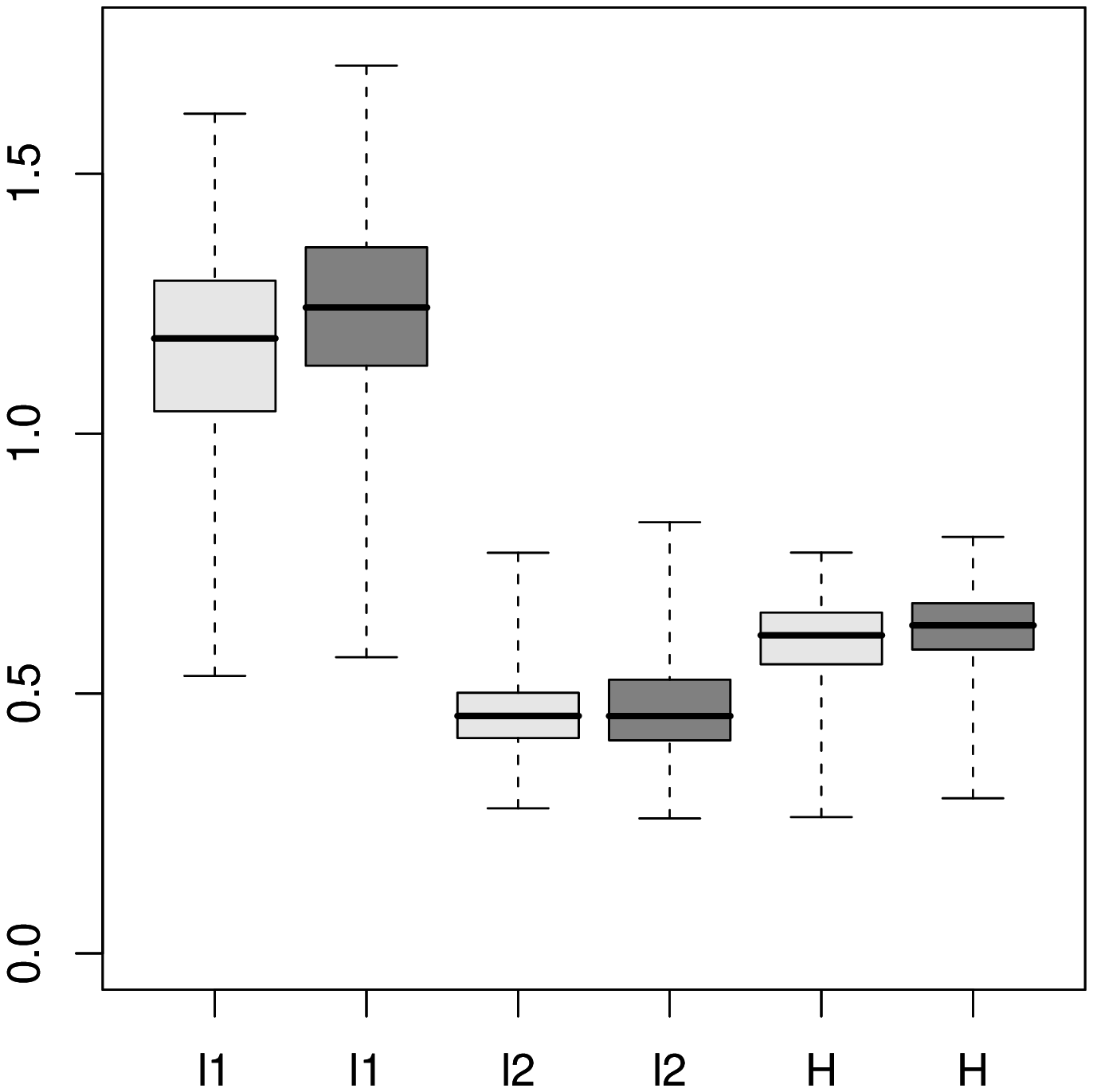}
\includegraphics[width=0.3\textwidth]{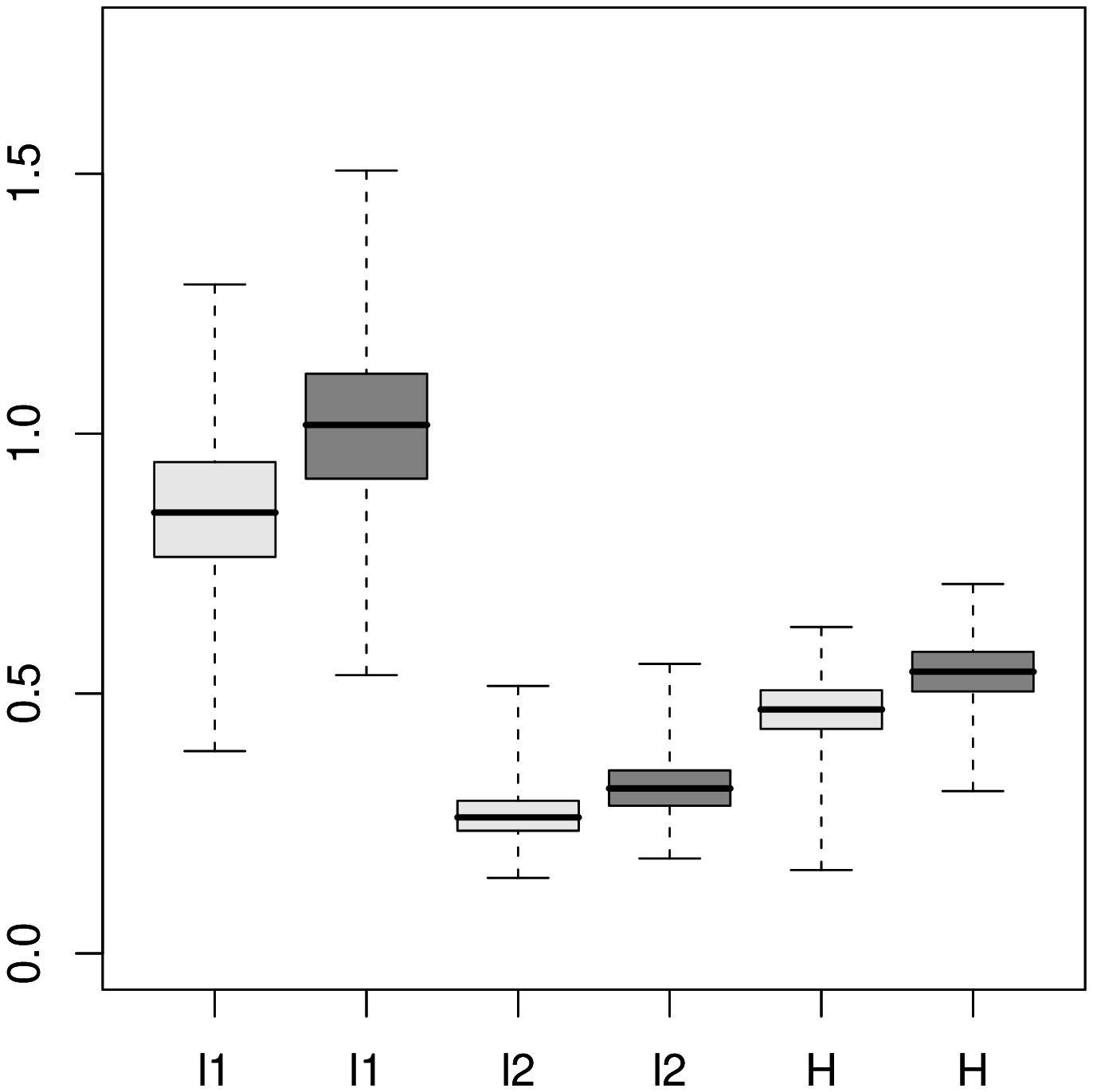}
\includegraphics[width=0.3\textwidth]{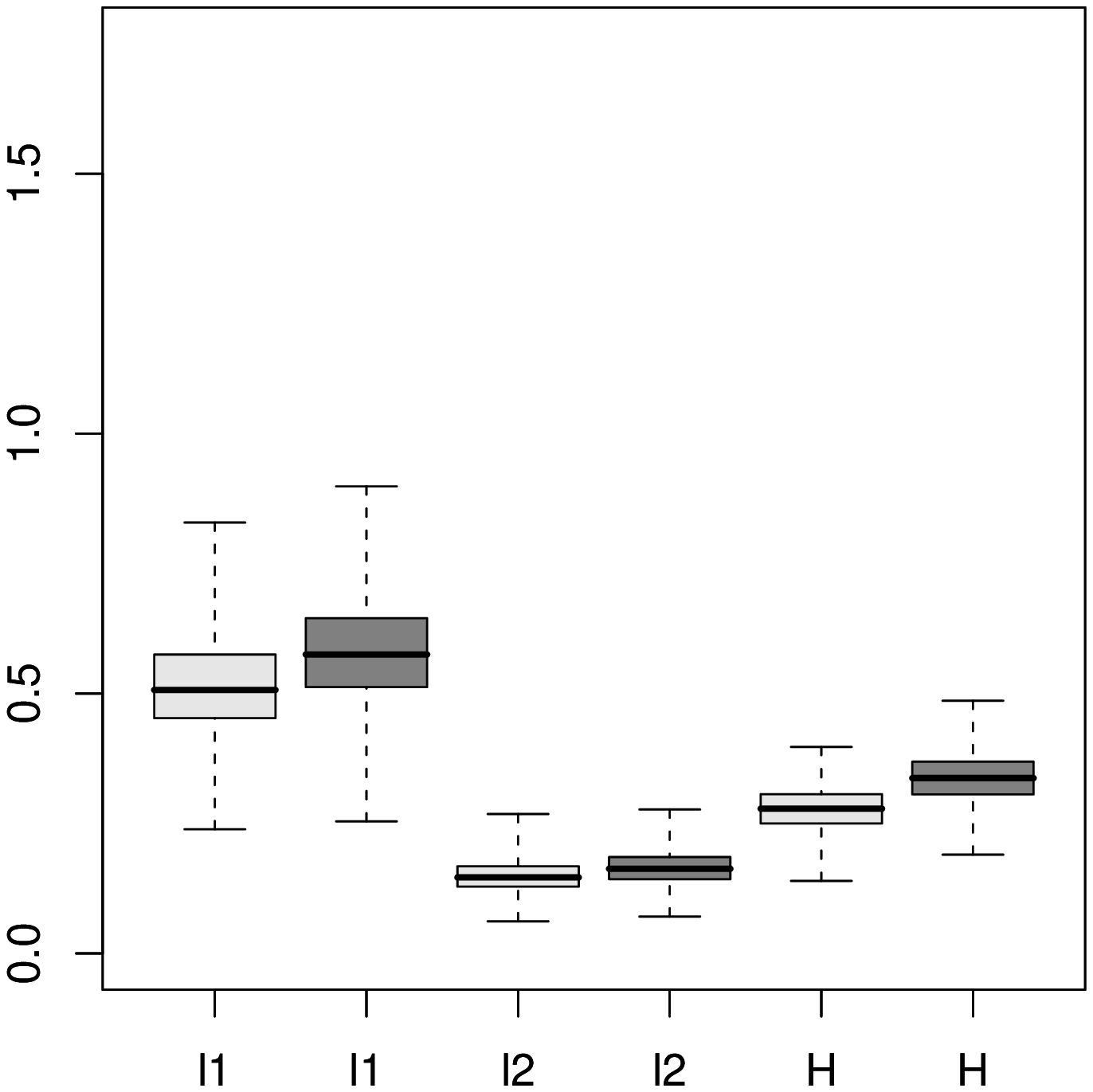}
\caption{Monte Carlo comparison of the estimators
$\widehat q_n^R$ (light grey) and $\widehat q_n^G$ (dark grey).}
\label{fig:mix_box}
\end{figure}

In Figure \ref{fig:mix_box}, we show boxplots of $m=1000$ samples of the distances $\ell_1(\tilde q,q), \ell_2(\tilde q,q)$
and $H(\tilde q,q)$ for $\tilde q=\widehat q_n^R$ (light grey) and
$\tilde q=\widehat q_n^G$ (dark grey) with $n=20$ (left), $n=100$ (centre) and
$n=1000$ (right). From top to bottom the true distributions are
\begin{list}{}
        {\setlength{\topsep}{3pt}
        \setlength{\parskip}{0pt}
        \setlength{\partopsep}{0pt}
        \setlength{\parsep}{0pt}
        \setlength{\itemsep}{0pt}
        \setlength{\leftmargin}{50pt}}
\item[(a)] $p=p^{U(5)}$,
\item[(b)] $p=0.2 p^{U(3)}+0.8 p^{U(7)}$,
\item[(c)] $p=0.25 p^{U(1)}+0.2 p^{U(3)}+0.15 p^{U(5)}+0.4 p^{U(7)}$, and
\item[(d)] $p$ is geometric with $\theta=0.75$.
\end{list}
We can see that $\widehat q_n^G$ has better performance in all metrics,
except for the case of the strictly decreasing distribution.
As before, the flatter the true distribution is, the
better the relative performance of $\widehat q_n^G$.  
Notice that by Corollary \ref{cor:decrasing} and Theorem \ref{thm:mixing_rate} 
the asymptotic behaviour (i.e. rate of convergence and limiting distributions) 
of the $l_2$ norm of $q_n^G$ and $q_n^R$ should be the same if $p$ is strictly decreasing.

\begin{rem}
For $\kappa = \infty$, the process $\{ x Y_{n,x} : x \in \NN\}$ is known to converge weakly
in $\ell_2$ if and only if $ \sum_{x \ge 0} x^2 p_x < \infty$, while the convergence 
is know to hold in $\ell_1$ if and only if $\sum_{x \ge 0} x \sqrt{p_x} < \infty$; 
see e.g. \citet[Exercise 3.8.14, page 205]{MR576407}.
We therefore conjecture that $Z_n^R$ and $Z_n^G$ converge weakly to 
$Z^R$ and $Z^G$ in $\ell_2$ (resp. $\ell_1$) if and only if $\sum_{x\ge 0} x^2 p_x < \infty$  
$\left(\mbox{resp. } \sum_{x \ge 0} x \sqrt{p_x} < \infty\right)$. 
\end{rem}

\section{Proofs}\label{sec:proofs}

\begin{proof}[Proof of Remark \ref{rem:p_k_bound}]
This bound follows directly from the definition of $p$, since
\begin{eqnarray*}
&p_x = \sum_{y > x}^\infty (y+1) q_y \le (x+1)^{-1} \sum_{y\ge x} q_y \le (x+1)^{-1}.&
\end{eqnarray*}
\end{proof}

In the next lemma, we prove several useful properties of both the 
rearrangement and Grenander operators.  

\begin{lem}\label{lem:key_ineq}
Consider two sequences $p$ and $q$ with support $\mathrm S$, and 
let $\varphi(\cdot)$ denote either the Grenander or rearrangement operator.  
That is,  $\varphi(p)=\gren(p)$ or $\varphi(p)=\rear(p)$.
\begin{list}{}
        {\setlength{\topsep}{15pt}
        \setlength{\parskip}{0pt}
        \setlength{\partopsep}{0pt}
        \setlength{\parsep}{0pt}
        \setlength{\itemsep}{6pt}
        \setlength{\leftmargin}{15pt}}
\item[1.]  For any increasing function $f:\mathrm S\mapsto \RR$,
\begin{eqnarray}\label{line:key_ineq1}
\sum_{x\in \mathrm{S}} f_x \varphi(p)_x \leq \sum_{x\in \mathrm{S}} f_x p_x.
\end{eqnarray}
\item[2.]  Suppose that $\Psi:\RR\mapsto \RR_+$ is a non--negative convex function 
such that $\Psi(0)=0$, and that $q$ is decreasing.   Then, 
\begin{eqnarray}\label{line:key_ineq2}
\sum_{x\in \mathrm{S}} \Psi (\varphi(p)_x-q_x) &\leq& \sum_{x\in \mathrm{S}} \Psi (p_x-q_x).
\end{eqnarray}
\item[3.]  Suppose that $|\mathrm S|$ is finite.    Then $\varphi(p)$ is a continuous function of $p$.
\end{list}
  
\end{lem}

\begin{proof}
\begin{list}{}
        {\setlength{\topsep}{15pt}
        \setlength{\parskip}{0pt}
        \setlength{\partopsep}{0pt}
        \setlength{\parsep}{0pt}
        \setlength{\itemsep}{6pt}
        \setlength{\leftmargin}{15pt}}
\item[1.] 
Suppose that $\mathrm S=\{s_1, \ldots, s_2\}$, where it is possible that $s_2=\infty.$  
Then it is clear from the properties of the rearrangement and Grenander operators that 
\begin{eqnarray*}
\sum_{x=s_1}^{s_2}\varphi(p)_x  = \sum_{x=s_2}^{s_1} p_x & \mbox{ and } & 
\sum_{x=s_1}^{y}\varphi(p)_x  \geq \sum_{x=s_1}^{y} p_x,
\end{eqnarray*}
for $y\in\mathrm S.$
These inequalities immediately imply \eqref{line:key_ineq1}, since,  by summation by parts, 
\begin{eqnarray*}
\sum_{x=s_1}^{s_2} f_x p_x 
&=& \sum_{x=s_1}^{s_2} \sum_{y=s_1}^{x-1} (f_{y+1} - f_{y}) p_x+ f_{s_1} \sum_{x=s_1}^{s_2}  p_x\\
&=&\sum_{y=s_1}^{s_2} (f_{y+1} - f_{y}) \sum_{x=y+1}^{s_2}  p_x+ f_{s_1} \sum_{x=s_1}^{s_2}  p_x,
\end{eqnarray*}
and $f$ is an increasing function.  
\item[2.] 
For the Grenander estimator this is simply Theorem 1.6.1 in \cite{MR961262}.   
For the rearrangement estimator, we adapt the proof from Theorem~3.5 in \cite{MR1415616}.    
We first write $\Psi = \Psi_+ + \Psi_-,$ where $\Psi_+(x) = \Psi(x)$ for $x\geq 0$ and 
$\Psi_-(x) = \Psi(x)$ for $x\leq 0$.  Now, since $\Psi_+$ is convex, there exists an increasing 
function $\Psi'_+$ such that $\Psi_+(x)= \int_0^x \Psi_+'(t)dt$.  Now, 
\begin{eqnarray*}
\Psi_+(p_x-q_x) &=& \int_{q_x}^{p_x} \Psi_+'(p_x-s) ds = \int_0^\infty \Psi_+'(p_x-s) \bb I_{[q_x\leq s]}ds.
\end{eqnarray*}
Applying Fubini's theorem, we have that 
\begin{eqnarray*}
\sum_{x\in \mathrm S} \Psi_+(p_x-q_x) &=& \int_0^\infty \left\{  \sum_{x\in \mathrm S} 
    \Psi_+'(p_x-s) \bb I_{[q_x\leq s]} \right\}ds.
\end{eqnarray*}
Now, the function $\bb I_{[q_x\leq s]}$ is an increasing function of $x$, and for 
$\varphi(p)=\rear(p),$ for each fixed $s$ we have that 
$\varphi(\Psi_+'(p-s))_x=\Psi_+'(\varphi(p)_x-s)$, since $\Psi_+'$ is an increasing function.   
Therefore, applying \eqref{line:key_ineq1}, we find that the last display above is bounded below by 
\begin{eqnarray*}
\int_0^\infty \left\{  \sum_{x\in \mathrm S} \Psi_+'(\varphi(p)_x-s) \bb I_{[q_x\leq s]} \right\}ds 
&=& \sum_{x\in \mathrm S} \Psi_+(\varphi(p)_x-q_x).
\end{eqnarray*}
The proof for $\Psi_-$ is the same, except that here we use the identity
\begin{eqnarray*}
\Psi_-(p_x-q_x) &=&  \int_0^\infty \Psi_-'(p_x-s) \left\{-\bb I_{[q_x\geq s]}\right\}ds.
\end{eqnarray*}
\item[3.]   Since $|\mathrm S|$ is finite, we know that $p$ is a finite vector, 
and therefore it is enough to prove continuity at any point $x\in \mathrm S.$  
For $\varphi=\rear$ this is a well--known fact.  Next, note that if $p_n\rightarrow p$, 
then the partial sums of $p_n$ also converge to the partial sums of $p$.  
From Lemma~2.2 of 
\cite{durot:tocquet:03}, it follows that the least concave majorant of $p_n$ 
converges to the least concave majorant of $p$, and hence, so do their differences.  
Thus $\varphi(p_n)_x\rightarrow \varphi(p)_x.$
\end{list}
\end{proof}

\subsection{Some inequalities and consistency results: proofs}

\begin{proof}[Proof of Theorem \ref{thm:BasicInequalities}]
\begin{list}{}
        {\setlength{\topsep}{15pt}
        \setlength{\parskip}{0pt}
        \setlength{\partopsep}{0pt}
        \setlength{\parsep}{0pt}
        \setlength{\itemsep}{6pt}
        \setlength{\leftmargin}{15pt}}
\item[(i).]
Choosing $\Psi(t)=|t|^k$ in \eqref{line:key_ineq2} of Lemma \ref{lem:key_ineq} 
proves \eqref{ineq:ellk}.  To prove \eqref{ineq:hellinger} recall that
\begin{eqnarray*}
H^2(\tilde p, p)&=& 1 - \sum_{x\geq 0}\sqrt{\tilde p_x p_x}.
\end{eqnarray*}
By \cite{MR0046395}, Theorem 368, page 261, (or Theorem 3.4 in \cite{MR1415616}) it follows that
$$
\sum_{x\geq 0} \sqrt{\widehat{p}_{n,x} p_x} \leq \sum_{x\geq 0} \sqrt{\widehat{p}_{n,x}^R p_x},
$$
which proves the result for the rearrangement estimator.   It remains to 
prove the same for the MLE.  Let $\{B_i\}_{i\geq 1}$ denote a partition of $\NN$.  By definition,
\begin{eqnarray*}
\widehat p_{n,x}^G = \frac{1}{|B_i|}\sum_{x\in B_i}\widehat p_{n,x}, \ \ x\in B_i
\end{eqnarray*}
for some partition.  Jensen's inequality now implies that
$$
\sum_{x\in B_i}\sqrt{\widehat p_{n,x}}\geq \sum_{x\in B_i}\sqrt{\widehat p_{n,x}^G},
$$
which completes the proof.
\item[(ii).] is obvious.
\item[(iii).] The second statement is obvious in light of \eqref{ineq:ellk} with $k=\infty$.  
To see that the probability of monotonicity of the  $\widehat{p}_{n,x}$'s  converges to $1/(y+1)!$ 
under the uniform distribution, note that the event in question is that same as the event 
that the components of the vector  $( \sqrt{n} ( \widehat{p}_{n,x} - (y+1)^{-1} ) : x \in \{ 0, \ldots , y \} \}$ 
are ordered in the same way.  This vector converges in distribution to $Z \sim N_{y+1} ( 0 , \Sigma)$ 
where $\Sigma = \mbox{diag} ( 1/(y+1)) - (y+1)^{-2} \underline{1} \underline{1}^T $, 
and the probability $P( Z_1 \ge Z_2 \ge \cdots \ge Z_{y+1} ) = 1/(y+1)!$ 
since the components of $Z$ are exchangeable.
\end{list}
\end{proof}

\begin{proof}[Proof of Corollary \ref{cor:risk_finite}]
For any $p\in \mc P,$ we have that 
\begin{eqnarray*}
n R_{2}(p, \widehat p_n^R) \leq n R_{2}(p, \widehat p_n) = 1-\sum_{x\geq 0} p_x^2 &\leq & 1.
\end{eqnarray*}
Plugging in the discrete uniform distribution on $\{0, \ldots, \kappa\}$, and 
applying part (ii) of Theorem \ref{thm:BasicInequalities}, we find that\begin{eqnarray*}
n R_2(p, \widehat p_n^R) = n R_2(p, \widehat p_n) = 1-(\kappa+1)^{-1}.
\end{eqnarray*}
Thus, for any $\eps>0$, there exists a $p\in \mc P$, such that 
\begin{eqnarray*}
nR_2(p, \widehat p_n^R) = nR_2(p, \widehat p_n) \geq 1-\eps.
\end{eqnarray*}
Since the upper bound on both risks is one, the result follows.   
\end{proof}

\begin{proof}[Proof of Theorem \ref{thm:GlobalConsistencyTheorem}]
The results of this theorem are quite standard, and we provide a proof only for 
completeness.  Let $\FF_n$ denote the empirical distribution function and $F$ 
the cumulative distribution function of the true distribution $p$.  For any $K$ (large), we have that for any $x> K$,
\begin{eqnarray*}
|\widehat p_{n,x}-p_x|&\leq & \widehat p_{n,x}+p_x\\
                  &\leq & (1-\FF_n(K))+(1-F(K))\\
                  &\leq & |\FF_n(K)-F(K)|+2(1-F(K)).
\end{eqnarray*}
Fix $\eps>0$, and choose $K$ large enough so that $(1-F(K))<\eps/6$.  
Next, there exists an $n_0$ sufficiently large so that 
$\sup_{0\leq x \leq K}|\widehat p_{n,x}-p_x|< \eps/3$ and $|\FF_n(K)-F(K)|<\eps/3$ 
for all $n\geq n_0$ almost surely.  Therefore, for $n\geq n_0$
\begin{eqnarray*}
\sup_{x\geq 0}|\widehat p_{n,x}-p_x|
&\leq &\sup_{0\leq x \leq K}|\widehat p_{n,x}-p_x|+|\FF_n(K)-F(K)|+2(1-F(K))\\
&<& \eps.
\end{eqnarray*}
This shows that $||\widehat p_n-p||_k\rightarrow 0$ almost surely for $k=\infty.$  
A similar approach proves the result for any $1\leq k <\infty.$  Convergence of 
$H(\widehat p_n, p)$ follows since for mass functions $H(p,q)\leq \sqrt{||p-q||_1}$ 
(see e.g. \cite{lecam:69}, page 35).  Consistency of the other estimators, 
$\widehat p_n^R$ and $\widehat p_n^G$ now follows from the inequalities 
of Theorem \ref{thm:BasicInequalities}.
\end{proof}

\begin{proof}[Proof of Corollary \ref{cor:glivenko}]
Note that by virtue of the estimators, we have that $\widehat F_n^R(x) \geq \FF_n(x)$ 
and $\widehat F_n^G(x) \geq \FF_n(x)$ for all $x\geq0.$  Now, fix $\eps>0.$  
Then there exists a $K$ such that  $\sum_{x>K}p_x < \eps/4.$  
By the Glivenko-Cantelli lemma, there exists an $n_0$ such that for all $n\geq n_0$
\begin{eqnarray*}
\sup_{x\geq 0}\left|\FF_n(x)-F(x)\right| < \eps/4,
\end{eqnarray*}
almost surely.  Furthermore, by Theorem \ref{thm:GlobalConsistencyTheorem}, $n_0$ 
can be chosen large enough so that for all $n\geq n_0$
\begin{eqnarray*}
\sup_{x\geq 0} \left|\widehat p_{n,x}^G-p_x\right| < \eps/4(K+1),
\end{eqnarray*}
almost surely.  Therefore,  for all $n\geq n_0$, we have that
\begin{eqnarray*}
\sup_{x\geq 0} |\widehat F_n^G(x)-F(x)| 
&\leq & \sum_{x=0}^K \left|\widehat p_{n,x}^G-p_x\right| + \sum_{x>K} \widehat p_{n,x}^G + \sum_{x>K}p_x\\
&\leq & \eps/4 + \sum_{x>K} \widehat p_{n,x} + \eps/4\\
&\leq & \eps/4 + \sum_{x>K} p_x + \eps/4 + \eps/4
\leq  \eps.
\end{eqnarray*}
The proof for the rearrangement estimator is identical.
\end{proof}

\subsection{Limiting distributions: proofs}

\begin{lem}\label{lem:tight}
Let $W_n$ be a sequence of processes in $\ell_k$ with $1\leq k < \infty$.  Suppose that
\newcounter{count4}
\begin{list}{\arabic{count4}.}
        {\usecounter{count4}
        \setlength{\topsep}{6pt}
        \setlength{\parskip}{0pt}
        \setlength{\partopsep}{0pt}
        \setlength{\parsep}{0pt}
        \setlength{\itemsep}{3pt}
        \setlength{\leftmargin}{25pt}}
\item $\sup_n E[||W_n||_k^k]<\infty$,
\item $\lim_{m\rightarrow\infty}\sup_n \sum_{x\geq m}E[|W_{n,x}|^k]=0$.
\end{list}  Then $W_n$ is tight in $\ell_k$.
\end{lem}

\begin{proof}
Note that for $k<\infty$, compact sets $K$ are subsets of $\ell_k$ such that there exists a sequence of real numbers $A_x$ for $x\in \NN$ and a sequence $\lambda_m\rightarrow 0$ such that
\newcounter{count3}
\begin{list}{\arabic{count3}.}
        {\usecounter{count3}
        \setlength{\topsep}{6pt}
        \setlength{\parskip}{0pt}
        \setlength{\partopsep}{0pt}
        \setlength{\parsep}{0pt}
        \setlength{\itemsep}{3pt}
        \setlength{\leftmargin}{25pt}}
\item $|w_x|\leq A_x$ for all $x\in \NN$,
\item $\sum_{k\geq m}|w_x|^k\leq \lambda_m$ for all $m$,
\end{list}
for all elements $w\in K.$   Clearly, if the conditions of the lemma are satisfied, then for each $\eps>0$, we have that
\begin{eqnarray*}
P\left(|W_{n,x}|\leq A_x \mbox{ for all } x\geq 0, \mbox{ and } 
\sum_{x\geq m}|W_{n,x}|^k \leq \lambda_m \mbox{ for all }m \right)\geq 1-\eps
\end{eqnarray*}
for all $n$.  Thus, $W_n$ is tight in $\ell_k$.
\end{proof}

\begin{proof}[Proof of Theorem \ref{thm:l2Gaussian}]
Convergence of the finite dimensional distributions is standard.  
It remains to prove tightness in $\ell_2$.  By Lemma \ref{lem:tight} this is straightforward, since
\begin{eqnarray*}
E[||Y_n||^2_2] =  \sum_{x\geq0}p_x(1-p_x) \ \ \mbox{ and } \ \ 
\sum_{x\geq m}E\left[Y_{n,x}^2\right]=\sum_{x\geq m}p_x(1-p_x).
\end{eqnarray*}
\end{proof}

Throughout the remainder of this section we make extensive use of a set 
equality for the least concave majorant known as the ``switching relation".  Let
\begin{eqnarray}
\widehat{s}_n (a)
& = & \inf \left \{ k \ge -1 : \ \FF_n (k) - a (k+1) = \sup \{ \FF_n (y) - a(y+1) \} \right \}\notag\\
&\equiv& \argmax_{k\geq -1}^L\{\FF_n (k) - a (k+1)\}
\label{CarefulDefnOfArgmaxProcessDiscreteCase}
\end{eqnarray}
denote the first time that the process $\FF_n(y)-a(y+1)$ reaches its maximum.  Then the following holds
\begin{eqnarray}
\{\widehat s_n(a)< x\}&=& \{\widehat s_n(a)\leq x-1/2\}\notag\\
                  &=& \{\widehat p_{n,x}^G < a\}. \label{line:switching}
\end{eqnarray}
For more background (as well as a proof) of this fact see, for example, \cite{bjpsw:09}.

\begin{proof}[Proof of Proposition \ref{prop:flat}]
Let $F$ denote the cumulative distribution function for the function $p$.  
For fixed $t \in \RR$ it follows from \eqref{line:switching} that
\begin{eqnarray}
P(Y_{n,x}^G< t)
&=& P(\widehat{s}_n(p_x + n^{-1/2}t)\leq x-1/2)\nonumber\\
& = & P( \mbox{argmax}^L_{y \geq -1} \{Z_n(y)\}\leq x-1/2)\label{line:calc1}
\end{eqnarray}
where $Z_n(y)=n^{1/2} \FF_n (y)- (n^{1/2} p_x + t)(y+1).$  
Note that for any constant $c$, $\argmax^L(Z_n(y))=\argmax^L(Z_n(y)+c)$, 
and therefore we instead take
\begin{eqnarray*}
Z_n(y)&=& n^{1/2}(\FF_n (y)-\FF_n(r-1))- (n^{1/2} p_x + t)(y-r+1)\\
      &=& V_n(y)+W_n(y)-t(y-r+1),
\end{eqnarray*}
where
\begin{eqnarray*}
V_n(y)&=& \sqrt{n}\left((\FF_n(y)-\FF_n(r-1))-(F(y)-F(r-1))\right),\\
n^{-1/2}W_n(y)&=&(F(y)-F(r-1))-p_x(y-r+1)\ \\
&=&\left\{
\begin{array}{ll}
=0 & \mbox{for} \ \ r-1 \leq y \leq s,\\
<0 & \mbox{otherwise.}
\end{array}\right.
\end{eqnarray*}
Let $\UU$ denote the standard Brownian bridge on $[0,1].$  
It is well-known that $V_n(y)\Rightarrow \UU(F(y))-\UU(F(r-1))$.  
Also, $W_n(y)\rightarrow\infty$ for $y\notin\{r-1, \ldots, s\},$ and it is identically zero otherwise.  
It follows that the limit of \eqref{line:calc1} is
\begin{eqnarray*}
&&P\left(\argmax^L_{r-1 \leq y \leq s}\{\UU(F(y))-\UU(F(r-1))-t(y-r+1)\} \leq x-1/2\right)\\
&=& P\left(\argmax^L_{r-1 \leq y \leq s}\{\UU(F(y))-\UU(F(r-1))-t(y-r+1)\} < x\right),
\end{eqnarray*}
for any $x\in\{r, \ldots, s\}.$  Note that the process
\begin{eqnarray*}
\{\UU(F(x))-\UU(F(r-1)), x=r-1, \ldots, s\}&=_d& \left\{\sum_{j=r}^x Y_j, x=r-1, \ldots, s\right\},
\end{eqnarray*}
and therefore the probability above is equal to
\begin{eqnarray*}
P\left(\gren(Y^{(r,s)})_x<t\right)
\end{eqnarray*}
for $x\in\{r, \ldots, s\}.$  Since the half-open intervals $[a,b)$ are 
convergence determining, this proves pointwise convergence of $Y_{n,x}^G$ to $\gren(Y)_x.$

To show convergence of the rearrangement estimator fluctuation process, 
note that for sufficiently large $n$ we have that 
$\widehat p_{n,r-k} > \widehat p_{n,x} > \widehat p_{n, s+k}$ for all 
$x\in \{r, \ldots, s\}$ and $k\geq 1$.  Therefore, 
$(\widehat p_n^R)^{(r,s)} = \rear((\widehat p_n)^{(r,s)})$ and furthermore, 
since $p_x$ is constant here,  $(Y_n^R)^{(r,s)} = \rear(Y_n^{(r,s)})$.  
The result now follows from the continuous mapping theorem.
\end{proof}

\begin{proof}[Proof of Proposition \ref{prop:carolandykstra}]
To simplify notation, let $\WW_m = \UU(F(m-r+1))-\UU(F(r-1))$ for 
$m=0, \ldots, s-r+1$.  Also, let $\theta=p_r=\ldots=p_s$ and then 
$G_m = F(m-r+1)-F(r-1)=\theta m$. Write
\begin{eqnarray*}
\WW_m &=& \frac{G_m}{G_s}\WW_s+ \left\{\WW_m - \frac{G_m}{G_s}\WW_s\right\} 
= \frac{m}{\bar s}\WW_s+ \left\{\WW_m - \frac{m}{\bar s}\WW_s\right\},
\end{eqnarray*}
where $\bar s=s-r+1$.  Let $\widetilde \WW_m=\WW_m - m\WW_s/\bar s$.  
Then $\widetilde \WW_0=\widetilde \WW_{\bar s}=0$ and 
some calculation shows that $E[\widetilde \WW_m]=0$ and
\begin{eqnarray*}
\cov(\widetilde \WW_m, \widetilde \WW_{m'})
&=& \theta \bar s \left\{\min\left(\frac{m}{\bar s},\frac{m'}{\bar s}\right)-\frac{m}{\bar s}\frac{m'}{\bar s}\right\}.
\end{eqnarray*}
Also, $\cov(\widetilde \WW_m, \WW_s)=0.$  
Let $Z$ be a standard normal random variable independent 
of the standard Brownian bridge $\UU$.  We have shown that
\begin{eqnarray*}
\WW_m &=_d&  \frac{m}{\bar s} \sqrt{\theta \bar s(1-\theta \bar s)}\ Z + \sqrt{\theta \bar s}\UU\left(\frac{m}{\bar s}\right).
\end{eqnarray*}
Next, let $\widetilde Y_m = \UU\left(\frac{m}{\bar s}\right)-\UU\left(\frac{m-1}{\bar s}\right)$ 
for $m=1, \ldots, \bar s$.  The vector $\widetilde Y = (\widetilde Y_1, \ldots, \widetilde Y_{\bar s})$ 
is multivariate normal with mean zero and 
$\cov(\widetilde Y_m,\widetilde Y_{m'})=\delta_{m, m'}/\bar s - 1 /(\bar s)^2.$  
To finish the proof, note that $\gren(c + \widetilde Y)=c+\gren(\widetilde Y)$ for any constant $c$.
\end{proof}

\begin{proof}[Proof of Proposition \ref{prop:monotone}]
The claim for the rearrangement estimator follows directly from Theorem 
\ref{thm:GlobalConsistencyTheorem} for $k=\infty$.  
To prove the second claim, we will show that 
$Y_{n,x}^G-Y_{n,x}=\sqrt{n}(\widehat p_{n,x}^G- \widehat p_{n,x})\stackrel{p}{\rightarrow}0$.  
To do this, we again use the switching relation \eqref{line:switching}.


Fix $\eps>0$.  Then
\begin{eqnarray}
P(Y_{n,x}^G-Y_{n,x}\geq\eps)
&=& P(\widehat{p}_{n,x}^G \geq \widehat{p}_{n,x} + n^{-1/2}\eps)\notag\\
&=& P(\widehat{s}_n (\widehat{p}_{n,x} + n^{-1/2} \eps) \geq x-1/2)\notag\\
&=& P(\argmax^L_{y\geq -1} \widetilde Z_n(y-x)\geq x-1/2)\notag\\
&=& P(\argmax^L_{h\geq -x-1} \widetilde Z_n(h)\geq -1/2),\label{line:calc2}
\end{eqnarray}
where $\widetilde Z_n(h)=n^{1/2} \FF_n (x+h)- (n^{1/2}\widehat p_{n,x} + t)(x+h+1).$  
Since for any constant $c$, $\argmax^L(\widetilde Z_n(y))=\argmax^L(\widetilde Z_n(y)+c)$, we instead take
\begin{eqnarray*}
\widetilde Z_n(h)
&=& n^{1/2}(\FF_n (x+h)-\FF_n(x-1))- (n^{1/2} \widehat p_{n,x} + \eps)(h+1)\\
&=& U_n(h)+V_n(h)+W_n(h)-\eps(h+1),
\end{eqnarray*}
where
\begin{eqnarray*}
U_n(h)&=& \sqrt{n}\left((\FF_n(x+h)-\FF_n(x-1))-(F(x+h)-F(x-1))\right),\\
(h+1)^{-1}V_n(h)&=& \sqrt{n}\left((\FF_n(x)-\FF_n(x-1))-(F(x)-F(x-1))\right),\\
n^{-1/2}W_n(h)&=&(F(x+h)-F(x-1))-p_x(h+1)\ \\
&=&\left\{
\begin{array}{ll}
=0 & \mbox{for} \ \ h=-1,0,\\
<0 & \mbox{otherwise.}
\end{array}\right.
\end{eqnarray*}
Let $\UU$ denote the standard Brownian bridge on $[0,1].$  
It is well-known that $U_n(h)\Rightarrow \UU(F(x+h))-\UU(F(x-1))$ and 
$V_n(h)\Rightarrow (h+1)(\UU(F(x))-\UU(F(x-1)))$. Also, $W_n(y)=0$ at $y=-1,0$ 
and $W_n(y)\rightarrow\infty$ for $y\notin\{-1,0\}.$  Define
\begin{eqnarray*}
\ZZ(h)= \UU(F(x+h))-\UU(F(x-1))+(h+1)(\UU(F(x))-\UU(F(x-1))),
\end{eqnarray*}
and notice that $\ZZ(0)=\ZZ(-1)=0$.  It follows that the limit of \eqref{line:calc2} is
\begin{eqnarray*}
P\left(\argmax^L_{y=-1,0}\{\ZZ(h)-\eps(h+1)\} \geq -1/2\right)&=& 0,
\end{eqnarray*}
since $\argmax^L_{y=-1,0}\{\ZZ(h)-\eps(h+1)\}=-1.$  A similar argument proves that
\begin{eqnarray*}
\lim_{n\rightarrow\infty}P(Y_{n,x}^G-Y_{n,x}<\eps)=0,
\end{eqnarray*}
showing that $Y_{n,x}^G-Y_{n,x}=o_p(1)$ and completing the proof.
\end{proof}

\begin{proof}[Proof of Theorem \ref{thm:process}]
Let $\varphi$ denote an operator on sequences in $l_2.$
Specifically, we take $\varphi=\gren$ or $\varphi=\rear$.
Also, for a fixed mass function $p$ let
$\mc T_p = \{x\geq 0: p_x-p_{x+1}>0\}=\{\tau_i\}_{i\geq 1}$.
Next,  define $\varphi_p$ to be the local version of the $\varphi$ operator.
That is, for each $i\geq 1$,
$\varphi_p(q)_x = \varphi(p^{(\tau_i+1, \tau_{i+1})})_x$
for all $\tau_{i}+1\leq x \leq \tau_{i+1}.$

Fix $\eps>0,$ and suppose that $q_n \rightarrow q$ in $\ell_2.$  
Then there exists a $K\in \mc T_p$ and an $n_0$ such that 
$\sup_{n\geq n_0}\sum_{x > K} q_{n,x}^2 < \eps/6.$  By Lemma \ref{lem:key_ineq},  $\varphi_p$ 
is continuous on finite blocks, and therefore it is continuous on $\{0,\ldots, K\}.$  
Hence, there exists a $n_0'$ such that for all $n\geq n_0'$ 
\begin{eqnarray*}
\sum_{x=0}^K (\varphi_p(q_n)_x-\varphi_p(q)_x)^2 \leq \eps/3.
\end{eqnarray*}
Applying \eqref{line:key_ineq2}, we find that for all $n\geq \max\{n_0, n_0'\}$
\begin{eqnarray*}
||\varphi_p(q_n)-\varphi_p(q)||^2_2 
&\leq & \sum_{x=0}^K (\varphi_p(q_n)-\varphi_p(q))^2 
             + 2\sum_{x>K} \varphi_p(q_n)_x^2 + 2\sum_{x>K} \varphi_p(q)_x^2 \\
&\leq & \eps/3 + 2 \sum_{x > K} q_{n,x}^2 +2\sum_{x > K} q_x^2 < \eps,
\end{eqnarray*}
which shows that $\varphi_p$ is continuous on $\ell_2.$ 
Since  $Y_n\Rightarrow Y$ in $\ell_2$, it follows, by the continuous mapping theorem,
that $\varphi_p(Y_n) \Rightarrow \varphi_p(Y)$.
   However, both $Y_n^G$ and $Y_n^R$ are of the form
$\sqrt{n}(\varphi(\widehat p_n)-p)\neq \varphi_p(Y_n)$.
To complete the proof of the theorem it is enough to show that
\begin{eqnarray*}
\bb E_n = ||\sqrt{n}(\varphi(\widehat p_n)-p)- \varphi_p(Y_n)||^2_2,
\end{eqnarray*}
converges to zero in $L_1$; that is, we will show that $E[\bb E_n] \rightarrow 0$. 

By Skorokhod's theorem, there exists a probability triple and
random processes $Y$ and $Y_n=\sqrt{n}(\widehat p_{n}-p)$, 
such that $Y_n \rightarrow Y$ almost surely in $\ell_2$.
 Fix $\eps >0$  and find $K\in\mc T_p$ such that $\sum_{x>K} p_x < \eps/4.$

Next, let $\mc T_p^K = \{0\leq x \leq K: x \in \mc T_p\},$ and let
$\delta = \min_{x\in \mc T_p^K} (p_x-p_{x+1})$.  Then, there exists an $n_0$ such that for all $n \geq n_0$
\begin{eqnarray}
\sup_{x\geq 0}|\widehat p_{n,x}-p_x| &<& \delta/3,\label{line:tight1}\\
\sup_{x\geq 0}|\sum_{0\leq y\leq x}\varphi(\widehat p_{n})_y-F(x)| &<& \delta/6,\label{line:tight2}
\end{eqnarray}
almost surely (see Corollary \ref{cor:glivenko}).

Now, consider any $m\in \mc T_p^K.$  It follows that any such $m$ is also
a touchpoint of the operator $\varphi$ on $\widehat p_n$.
Here, by touchpoint we mean that
$\sum_{x=0}^m \varphi(\widehat p_n)_x=\sum_{x=0}^m \widehat p_{n,x}.$
From \eqref{line:tight1}, it follows that
\begin{eqnarray*}
\inf_{x\leq m} \widehat p_{n,x} > \sup_{x>m} \widehat p_{n,x},
\end{eqnarray*}
which implies that $m$ is a touchpoint for the rearrangement estimator.
For the Grenander estimator, we require \eqref{line:tight2}.   Here,
\begin{eqnarray*}
\widehat F_n^G(m) - \widehat F_n^G(m-1) &>& F(m)-F(m-1) - \delta/3 \\
&=& p_m - \delta/3\\
&>& p_{m+1} + \delta/3\\
&=& F(m+1)-F(m) + \delta/3\\
&\geq & \widehat F_n^G(m+1) - \widehat F_n^G(m).
\end{eqnarray*}
Therefore, the slope of $\widehat F_n^G$ changes from $m$ to $m+1$,
which implies that $m$ is a touchpoint almost surely.
Let $\widehat p_{n}^{(s,r)}=\{\widehat p_{n,s}, \widehat p_{n,s+1}, \ldots , \widehat p_{n,r}\}$.
An important property of the $\varphi$ operator is if $m < m'$ are two
touchpoints of $\varphi$ applied to $\widehat p_n,$ then for all $m+1 \leq x \leq m'$,
$\varphi(\widehat p_n)_x = \varphi(\widehat p_n^{(m+1, m')})_x.$
Now, since $p$ takes constant values between the touchpoints $\mc T_p^K$,
 it follows that $\sqrt{n}(\varphi(\widehat p_n)-p)_x= \varphi_p(Y_n)_x$,
 for all $x\leq K.$

Therefore,  for all $n\geq n_0$
\begin{eqnarray*}
\bb E_n
&=& \sum_{x\geq 0} \left| \sqrt{n}(\varphi(\widehat p_n)-p)_x- \varphi_p(Y_n)_x\right|^2\\
&\leq & \sum_{x=0}^{K} \left( \sqrt{n}(\varphi(\widehat p_n)-p)_x- \varphi_p(Y_n)_x\right)^2 \\
&& \qquad + \ 2 \sum_{x> K} \left( \sqrt{n}(\varphi(\widehat p_n)-p)_x\right)^2
                  + 2 \sum_{x>K} \left( \varphi_p(Y_n)_x\right)^2\\
&\leq &  4 \sum_{x > K} \left( Y_{n,x}\right)^2,
\end{eqnarray*}
almost surely.  It follows that
\begin{eqnarray*}
\limsupp \, \bb E_n \leq 4 \sum_{x>K} \left( Y_{x}\right)^2,
\end{eqnarray*}
and hence
\begin{eqnarray*}
E\left[\limsupp \,  \bb E_n\right]
&\leq& 4 E\left[ \sum_{x>K} \left( Y_{x}\right)^2\right]  = 4 \sum_{x>K} p_x(1-p_x) < \eps.
\end{eqnarray*}
Since $\bb E_n \leq 2||Y_n||_2^2$, with $E[||Y_n||_2^2]\leq 1$, we may apply Fatou's lemma so that
\begin{eqnarray*}
0 \leq \limsupp\, E[\bb E_n] \leq E\left[\limsupp\, \bb E_n\right] \leq \eps.
\end{eqnarray*}
Letting $\eps\rightarrow 0$ completes the proof.
\end{proof}

Corollaries \ref{cor:decrasing} and \ref{cor:uniformprocess}
are obvious consequences of Theorem \ref{thm:process}.
Remark \ref{rem:prob_equal} is proved in the following section.

\subsection{Limiting distributions for metrics: proofs}

\begin{proof}[Proof of Corollary \ref{cor:ell2}]
We provide the details only in the $k=2$ setting.
The cases when $k>2$ follow in a similar manner, since here $||x||_k\leq ||x||_2$ for $x\in \ell_2$.

Convergence of $||Y_n||_2, ||Y_n^R||_2$ and $||Y_n^G||_2$
follows from Theorems \ref{thm:l2Gaussian} and \ref{thm:process}
by the continuous mapping theorem.  That $||Y||_2=||Y^R||_2$
is obvious from the definition of $Y^R.$
That $||Y^G||_2\leq ||Y||_2$ follows from Jensen's inequality and
the definition of the $\gren(\cdot)$ operator, since for any $r < s$,
$\gren(Y^{(r,s)})_x$ is equal to the average of $Y_y$ over some
subset of $\{r, \ldots, s\}$ containing the point $x$.  If $p$ is not strictly
decreasing, then there exists a region, which we denote again by
$\{r, \ldots, s\}$, where it is constant.  Then there is positive
probability that $(Y^G)^{(r,s)}$ is different from $Y^{(r,s)}.$  In this case, we have that
\begin{eqnarray*}
||(Y^G)^{(r,s)}||_2^2 < ||Y^{(r,s)}||_2^2,
\end{eqnarray*}
which finishes the proof of the stochastic ordering in the third statement.
Convergence in expectation is immediate since
\begin{eqnarray*}
E[||Y_n||_2^2]&=& \sum_{x\geq 0} p_x(1-p_x),
\end{eqnarray*}
and the same results for $Y_n^R, Y_n^G$ follow by the dominated
convergence theorem and the bounds in Theorem \ref{thm:BasicInequalities} (i).
Lastly, the bound $E[||Y^G||_2^2] \leq E[||Y||_2^2 ]$ with equality if and only if
$p$ is strictly monotone follows from the stochastic ordering.
\end{proof}

\begin{proof}[Proof of Corollary \ref{cor:ell1}]
The result of the corollary for the empirical estimator is essentially the
Borisov-Durst theorem (see e.g. \cite{MR1720712}, Theorem 7.3.1, page 244), which states that
\begin{eqnarray*}
\sup_{C\in 2^\NN}\left|\sum_{x\in C}Y_{n,x}\right|\Rightarrow \sup_{C\in 2^\NN}\left|\sum_{x\in C} Y_x\right|
\end{eqnarray*}
if $\sum_x \sqrt{p_x}<\infty$.
To complete the argument note that
$\sup_{C\in 2^\NN}|\sum_{x\in C}w_x| =||w||_1/2$ for any sequence $w$ such that 
$\sum_x w_x =0$ (note that the condition $\sum_x \sqrt{p_x}<\infty$ 
means that the sequences $Y_n$ and $Y$ are absolutely summable almost surely).
However, the result may also be proved by noting that the sequence
$Y_n$ is tight in $\ell_1$ using Lemma \ref{lem:tight}, since
\begin{eqnarray*}
E[||Y_n||_1]&\leq & \sum_{x\geq 0}\sqrt{p_x(1-p_x)},\\
\sum_{x\geq m} E[|Y_{n,x}|]  &\leq &\sum_{x\geq m}\sqrt{p_x(1-p_x)} \rightarrow 0,
\end{eqnarray*}
as $m\rightarrow\infty$ under the assumption $\sum_{x\geq 0} \sqrt{p_x} <\infty.$
The proof that $Y_n^G\Rightarrow Y^G$ and $Y_n^R\Rightarrow Y^R$ in 
$\ell_1$ is identical to the proof of Theorem \ref{thm:process}, and we omit the details.   
Convergence of expectations follows since $||Y_n||_1$ is uniformly integrable, as
\begin{eqnarray*}
E[||Y_n||_1 \mathbb I_{\{||Y_n||_1 > \alpha\}}]
&\leq&\frac{E[||Y_n||_1^2]}{\alpha}
            = \frac{1}{\alpha} \sum_{x,z}E[|Y_{n,x}||Y_{n,z}|]\\
&\leq & \frac{1}{\alpha}\left(\sum_{x\geq 0}\sqrt{p_x}\right)^2,
\end{eqnarray*}
by the Cauchy-Schwarz inequality.  All other details follow as in the proof of Corollary~\ref{cor:ell2}.
\end{proof}

\begin{proof}[Proof of Corollary \ref{cor:hell}]
If $\kappa<\infty,$ then we have that
\begin{eqnarray*}
8nH^2(\widehat p_n, p)
&=& 4n\sum_{x=0}^\kappa [\sqrt{\widehat p_{n,x}}-\sqrt{p_x}]^2\\
&=& 4\sum_{x=0}^\kappa \frac{[\sqrt{n}(\widehat p_{n,x}-p_x)]^2}{(\sqrt{\widehat p_{n,x}}+\sqrt{p_x})^2},
\end{eqnarray*}
which converges to
\begin{eqnarray}
4\sum_{x=0}^\kappa \frac{Y_x^2}{(2\sqrt{p_x})^2}=\sum_{x=0}^\kappa \frac{Y_x^2}{p_x}
\end{eqnarray}
by Theorem \ref{thm:l2Gaussian} and Theorem \ref{thm:GlobalConsistencyTheorem} 
for $k=\infty$. That this has a chi-squared distribution with $\kappa$ degrees of freedom 
is standard, and is shown for example, in \cite{ferg96}, Theorem~9.  
Convergence of means follows by the dominated convergence theorem 
from the bound $H(p,q)\leq \sqrt{||p-q||_1}$ (see e.g. \cite{lecam:69}, page 35) 
and Corollary \ref{cor:ell1}.  All other details follow as in the proof of Corollary \ref{cor:ell2}.
\end{proof}

\begin{proof}[Proof of Remark \ref{rem:wrongspace}]
Suppose first that $\sum_{x\geq 0}\sqrt{p_x}=\infty.$  Define $P$ to be the 
probability measure $P(A)=\sum_{x\in A}p_x,$ and let $W$ be the mean 
zero Gaussian field on $\ell^2$ such that $E[W_{x}W_{x'}]= p_x \delta_{x,x'}.$  
Then we may write $Y=_d \{W_x - p_x W_{\NN}\}_{x\geq 0}$, where $W_\NN=\sum_{x\geq 0}W_x.$

Now, since $\sum_{x\geq 0}P(|W_x|\geq \sqrt{p_x})=\infty$, by the 
Borel-Cantelli lemma we have that $\sum_{x\geq 0}|W_x|=\infty$ almost surely.  Since
\begin{eqnarray*}
\sum_{x\geq 0} |Y_x|&= & \sum_{x\geq 0} |W_x-p_x W_\NN|\\
&\geq & \sum_{x\geq 0} |W_x| - |W_\NN|,
\end{eqnarray*}
and $W_\NN$ is finite almost surely, it follows that $\sum_{x\geq 0} |Y_x|=\infty$ 
almost surely as well.  That is, if  $\sum_{x\geq 0}\sqrt{p_x}=\infty$, then the random variable 
$||Y||_1$ simply does not exist.  

A similar argument works for the Hellinger norm.  Assume that $\kappa=\infty.$  Then
\begin{eqnarray*}
\sum_{x\geq 0} \frac{Y_x^2}{p_x}&= &\left(\sum_{x\geq 0} \frac{W_x^2}{p_x}\right) - W_\NN^2,
\end{eqnarray*}
and the Borel-Cantelli lemma shows that $\sum_{x\geq 0} W_x^2/p_x$ is infinite almost surely.
\end{proof}

\begin{lem}\label{lem:l2_eq_touch}
Let $Z_1, \ldots, Z_k$ be i.i.d. N(0,1) random variables, and let $Z_i^G, i=1, \ldots, k$  
denote the left slopes of the least concave majorant of the graph of the cumulative sums 
$\sum_{i=1}^j Z_j$ with $j=0, \ldots, k.$  
Let $T$ denote the number of times that the LCM touches the cumulative sums 
(excluding the point zero, but including the point $k$).  Then
\begin{eqnarray*}
E\left[\sum_{i=1}^k (Z_i^G)^2\right]=E[T].
\end{eqnarray*}
\end{lem}

\begin{proof}
It is instructive to first consider some of the simple cases.  When $k=1,$ the result is obvious.   
Suppose then that $k=2.$  We have\\

\begin{center}
\begin{tabular}{ccc}
\toprule[1.5pt]
T & $\sum_{i=1}^k(Z_i^G)^2$ & if\\
\midrule
2 & $Z_1^2 + Z_2^2$ & $Z_1 > Z_2$\\
1 &  $\left(\frac{Z_1+Z_2}{\sqrt{2}}\right)^2$ & $Z_1 < \frac{Z_1+Z_2}{2}$\\
\bottomrule[1.5pt]\\
\end{tabular}
\end{center}
Note that we ignore all equalities, since these occur with probability zero.  It follows that
\begin{eqnarray*}
E\left[\sum_{i=1}^2(Z_i^G)^2\right]&=& E[(Z_1^2 + Z_2^2)1_{Z_1 > Z_2}]
     +E\left[\left(\frac{Z_1+Z_2}{\sqrt{2}}\right)^21_{Z_1< \frac{Z_1+Z_2}{2}}\right]
\end{eqnarray*}
where, by exchangeability it follows that
\begin{eqnarray*}
E[(Z_1^2 + Z_2^2)1_{Z_1 > Z_2}]&=& E[(Z_1^2 + Z_2^2)1_{Z_1 < Z_2}]\\
&=& E[(Z_1^2 + Z_2^2)]P(Z_1 > Z_2)\\
&=& 2 P(T=2).
\end{eqnarray*}
On the other hand, we also have that
\begin{eqnarray*}
E\left[\left(\frac{Z_1+Z_2}{\sqrt{2}}\right)^21_{Z_1< \frac{Z_1+Z_2}{2}}\right]
&=&E\left[\left(\frac{Z_1+Z_2}{\sqrt{2}}\right)^2\right] P\left(Z_1< \frac{Z_1+Z_2}{2}\right)\\ 
&=&1P(T=1),
\end{eqnarray*}
since the random variables $\bar Z=(Z_1+Z_2)/2$ and $Z_1-\bar Z$ are independent.  
The result follows.

Next, suppose that $k=3$.  Then we have the following.
\begin{center}
\begin{tabular}{cccc}
\toprule[1.5pt]
T & $\sum_{i=1}^k(Z_i^G)^2$ & \multicolumn{2}{c}{if}\\
\cmidrule(l){3-4}
& & (a) & (b)\\
\midrule
3 & $Z_1^2 + Z_2^2+Z_3^2$ & $Z_1> Z_2 > Z_3$ &\\
2 & $\left(\frac{Z_1+Z_2}{\sqrt{2}}\right)^2+Z_3^2$ & $\frac{Z_1+Z_2}{2}> Z_3$ &$\frac{Z_1+Z_2}{2} >Z_1$ \\
2 & $Z_1^2+\left(\frac{Z_2+Z_3}{\sqrt{2}}\right)^2$ & $Z_1>\frac{Z_2+Z_3}{2}$ & $\frac{Z_2+Z_3}{2} >Z_2$\\
1 &  $\left(\frac{Z_1+Z_2+Z_3}{\sqrt{3}}\right)^2$ & & $\frac{Z_1+Z_2+Z_3}{3}> Z_1, \frac{Z_1+Z_2}{2}$\\
\bottomrule[1.5pt]\\
\end{tabular}
\end{center}
The choice of splitting the conditions between columns (a) and (b) is key to 
our argument.  Note that the LCM creates a partition of the space $\{1, \ldots, k\}$, 
where within each subset the slope of the LCM is constant.  The number of partitions 
is equal to $T$.  Here, column (a) describes the necessary conditions on the order 
of the slopes on the partitions, while column (b) describes the necessary conditions 
that must hold within each partition.

In the first row of the table, we find by permuting across all orderings of (123) that
\begin{eqnarray*}
E[(Z_1^2 + Z_2^2+Z_3^2)\, 1_{Z_1> Z_2 > Z_3}] &= &E[(Z_1^2 + Z_2^2+Z_3^2)]P(Z_1> Z_2 > Z_3) \\
&=& 3 P(T=3).
\end{eqnarray*}
Next consider $T=2.$  Here, by permuting  $(123)$ to $(312)$, we find that
\begin{eqnarray*}
&& \hspace{-3cm}E\left[\left\{Z_1^2+\left(\frac{Z_2+Z_3}{\sqrt{2}}\right)^2\right\} 
         1_{Z_1>\frac{Z_2+Z_3}{2}}1_{\frac{Z_2+Z_3}{2} >Z_2}\right]\\
&=& E\left[\left\{\left(\frac{Z_1+Z_2}{\sqrt{2}}\right)^2+Z_3^2\right\} 
         1_{Z_3>\frac{Z_1+Z_2}{2}}1_{\frac{Z_1+Z_2}{2} >Z_1}\right].
\end{eqnarray*}
Note that the permutation $(123)$ to $(312)$ may be re-written as $(\{12\}\{3\})$ 
to $(\{3\}\{12\})$ which is really a permutation on the partitions formed by the LCM. Now,
\begin{eqnarray*}
E[T 1_{T=2}]
&=&   E\left[\left\{\left(\frac{Z_1+Z_2}{\sqrt{2}}\right)^2
            +Z_3^2\right\} 1_{\frac{Z_1+Z_2}{2}>Z_3}1_{\frac{Z_1+Z_2}{2} >Z_1}\right]\\\\
&& \ \ + \ \ E\left[\left\{Z_1^2+\left(\frac{Z_2+Z_3}{\sqrt{2}}\right)^2\right\} 
         1_{Z_1>\frac{Z_2+Z_3}{2}}1_{\frac{Z_2+Z_3}{2} >Z_2}\right]\\
&=& E\left[\left\{\left(\frac{Z_1+Z_2}{\sqrt{2}}\right)^2+Z_3^2\right\} 1_{\frac{Z_1+Z_2}{2} >Z_1}\right]\\
&=& E\left[\left\{\left(\frac{Z_1+Z_2}{\sqrt{2}}\right)^2+Z_3^2\right\}\right] P\left(\frac{Z_1+Z_2}{2} >Z_1\right)\\
&=& 2P(T=2),
\end{eqnarray*}
where in the penultimate line we use the fact that $Z_3, (Z_1+Z_2)/2$ and $Z_1-(Z_1+Z_2)/2$ are independent.

Lastly,
\begin{eqnarray*}
&&\hspace{-3cm}E\left[\left(\frac{Z_1+Z_2+Z_3}{\sqrt{3}}\right)^2 1_{\frac{Z_1+Z_2+Z_3}{3}>Z_1} 
        1_{\frac{Z_1+Z_2+Z_3}{3}>\frac{Z_1+Z_2}{2}}\right]\\
&=& E\left[\left(\frac{Z_1+Z_2+Z_3}{\sqrt{3}}\right)^2 1_{\frac{Z_1+Z_2+Z_3}{3}>Z_1} 
        1_{Z_3>\frac{Z_1+Z_2+Z_3}{3}}\right]\\
&=& E\left[\left(\frac{Z_1+Z_2+Z_3}{\sqrt{3}}\right)^2\right]E\left[ 1_{\frac{Z_1+Z_2+Z_3}{3}>Z_1} 
        1_{\frac{Z_1+Z_2+Z_3}{3}>\frac{Z_1+Z_2}{2}}\right]\\
&=& 1 P(T=1)
\end{eqnarray*}
as the variables $\overline{Z}=(Z_1+Z_2+Z_3)/3$ and 
$\{Z_1-\overline Z, Z_2-\overline Z, Z_3-\overline{Z}\}$ are independent.

The key to the general proof is the combination of two actions:
\begin{enumerate}
\item Permutations of subgroups (column (a)), and
\item independence of column (b) from the random variables $\sum_{i=1}^k(Z_i^G)^2$ 
and the indicator functions in column (a). Note that for any $k>j\geq 1$, 
letting $\bar Z = (Z_1+ Z_2 + \ldots + Z_k)/k$
\begin{eqnarray*}
\frac{Z_1+ Z_2 + \ldots + Z_j }{j}-\bar Z &=& \frac{(Z_1-\bar Z)+ (Z_2-\bar Z) + \ldots + (Z_j-\bar Z) }{j},
\end{eqnarray*}
which is independent of $\bar Z$ for any choice of $j<k.$
\end{enumerate}
To write down the proof for any $k$ we must first introduce some notation.
\begin{itemize}
\item For any $1\leq m\leq k$, we may create a collection $\mc P$ of partitions of 
$\{1, \ldots, k\}$ such that the total number of elements in each partition is $m$.  
For example, when $k=4$ and $m=2$, then the elements of $\mc P$ are the 
partitions $(\{1\}\{234\}),(\{12\}\{34\})$ and $(\{123\}\{4\})$.  Furthermore, for each partition, 
we may write down the number of elements in each subset of the partition.    
Here the sizes of the partitions are $1,3$ then $2,2$ and $3,1$.   These partitions may 
be grouped further by placing together all partitions such that their sizes are unique up to order.  
Thus, in the above example we would put together $1,3$ and $3,1$ as one group, and the 
second group would be made up of $2,2.$  From each subgroup we wish to choose a 
representative member, and the collection of these representatives will be denoted as $\tau(m).$  
We assume that the representative $\tau$ is chosen in such a way that the sizes of the 
partitions are given in increasing order.  
Let $r_1$ denote the number of subgroups with size 1, and so on.  
Thus, for $\tau=(\{1\}\{234\})$, we have $r_1=1, r_2=0, r_3=1, \ldots, r_k=0.$
\item Next, from $\tau(m)$ we wish to re--create the entire collection $\mc P$.  
To do this, it is sufficient to take each $\tau$ and re--create all of the partitions 
which had the same sizes.   Let $\sigma_m\tau$ denote the resulting collection for 
a fixed partition $\tau$.   Thus, $\mc P$ is equal to the union of $\sigma_m\tau$ over all $\tau \in \tau(m).$  
Note that the number of elements in $\sigma_m\tau$ is given by
\begin{eqnarray*}
{m \choose r_1\, r_2 \,\ldots \,r_k }.
\end{eqnarray*}
We also use the notation $R_j=\sum_{i=1}^j  r_i,$ with $R_0=0.$  Note that $R_k=m.$
\item For each partition $\sigma$, we write $\sigma_1, \ldots, \sigma_m$ to 
denote the individual subsets of the partition.  Thus, for $\sigma=(\{1\}\{234\}),$ 
we would have $\sigma_1=\{1\}$ and $\sigma_2 =\{2,3,4\}.$
\item For each $\sigma_j$ as defined above, we let
\begin{eqnarray*}
&AV_{\sigma_j} Z = \left(\sum_{i\in\sigma_j}Z_i\right)/|\sigma_j|, \mbox{ and } 
AV_{\sigma_j}^{-l} Z = \left(\sum_{i\in\sigma_j^{(l)}}Z_i\right)/|\sigma_j^{(l)}|,&
\end{eqnarray*}
where $\sigma_j^{(l)}$ denotes $\sigma_j$ with its \emph{last} $l$ elements removed.
\end{itemize}

We are now ready to calculate $E[\sum_{i=1}^k(Z_i^G)^2 1_{T=m}] $.  
By considering all possible partitions, this is equal to the sum over all $\tau\in \tau(m)$  of the following terms
\begin{eqnarray*}
\sum_{\sigma \in \sigma_m \tau}E\left[\left\{\sum_{j=1}^m |\sigma_j|\left(AV_{\sigma_j} Z \right)^2\right\} 
         1_{AV_{\sigma_1} Z > \ldots > AV_{\sigma _m} Z} 
         \prod_{j=1}^m 1_{AV_{\sigma_j} Z > \max \{AV_{\sigma_j}^{-1} Z, \ldots, AV_{\sigma_j}^{-(|\sigma_j|-1)} Z\}}\right].
\end{eqnarray*}
By permuting each $\sigma\in \sigma_m\tau$, and appealing to the 
exchangeability of the $Z_i$'s, this is equal to
\begin{eqnarray*}
&&\hspace{-1.5cm}E\left[\left\{\sum_{j=1}^m |\sigma_j|\left(AV_{\sigma_j} Z \right)^2\right\}
        \left\{ \prod_{i=1}^k 1_{AV_{\sigma_{R_{i-1}+1}} Z > \ldots > AV_{\sigma _{R_i}} Z} \right\}\right.\\
&&\left.\times\left\{\prod_{j=1}^m 
        1_{AV_{\sigma_j} Z > \max \{AV_{\sigma_j}^{-1} Z, \ldots, AV_{\sigma_j}^{-(|\sigma_j|-1)} Z\}}\right\}\right]\\
&=& E\left[\left\{\sum_{j=1}^m |\sigma_j|\left(AV_{\sigma_j} Z \right)^2\right\}
         \left\{ \prod_{i=1}^k 1_{AV_{\sigma_{R_{i-1}+1}} Z > \ldots > AV_{\sigma _{R_i}} Z} \right\}\right]\\
&&\hspace{1.5cm}\times E\left[\left\{
       \prod_{j=1}^m 1_{AV_{\sigma_j} Z > \max \{AV_{\sigma_j}^{-1} Z, \ldots, 
               AV_{\sigma_j}^{-(|\sigma_j|-1)} Z\}}\right\}\right],
\end{eqnarray*}
by independence of each $AV_{\sigma_j} Z$ and each $Z_i-AV_{\sigma_j} Z$ 
for $i\in \sigma_j.$  Notice that the permutations of $\sigma\in\sigma_m\tau$ 
do not account for permutations across all groups with equal ``size".    
By considering furthermore all permutations between groups of equal size, 
we further obtain that the last display above is equal to
\begin{eqnarray*}
&&E\left[\left\{\sum_{j=1}^m |\sigma_j|\left(AV_{\sigma_j} Z \right)^2\right\}\right]
      E\left[\left\{ \prod_{i=1}^k 1_{AV_{\sigma_{R_{i-1}+1}} Z > \ldots > AV_{\sigma _{R_i}} Z} \right\}\right]\\
&&\hspace{1.5cm}\times 
      E\left[\left\{\prod_{j=1}^m 1_{AV_{\sigma_j} Z > \max \{AV_{\sigma_j}^{-1} Z, \ldots, AV_{\sigma_j}^{-(|\sigma_j|-1)} Z\}}\right\}\right]\\
&=& m \  E\left[\left\{ \prod_{i=1}^k 1_{AV_{\sigma_{R_{i-1}+1}} Z > \ldots > AV_{\sigma _{R_i}} Z} \right\}\right] 
        E\left[\left\{\prod_{j=1}^m 1_{AV_{\sigma_j} Z > \max \{AV_{\sigma_j}^{-1} Z, \ldots, AV_{\sigma_j}^{-(|\sigma_j|-1)} Z\}}\right\}\right].
\end{eqnarray*}

Lastly, we collect terms to find that $E[\sum_{i=1}^k(Z_i^G)^2 1_{T=m}]$ is equal to $m$ times
\begin{eqnarray*}
&&\hspace{-1cm}\sum_{\tau\in \tau(m)} E\left[\left\{ \prod_{i=1}^k 
       1_{AV_{\sigma_{R_{i-1}+1}} Z > \ldots > AV_{\sigma _{R_i}} Z} \right\}\right] 
         E\left[\left\{\prod_{j=1}^m 
        1_{AV_{\sigma_j} Z > \max \{AV_{\sigma_j}^{-1} Z, \ldots, AV_{\sigma_j}^{-(|\sigma_j|-1)} Z\}}\right\}\right]\\
&=& \sum_{\tau\in \tau(m)}  \sum_{\sigma\in\sigma_m\tau} 
         E\left[ 1_{AV_{\sigma_1} Z > \ldots > AV_{\sigma _m} Z} 
         \left\{\prod_{j=1}^m 
         1_{AV_{\sigma_j} Z > \max \{AV_{\sigma_j}^{-1} Z, \ldots, AV_{\sigma_j}^{-(|\sigma_j|-1)} Z\}}\right\}\right]\\
&=& P(T=m),
\end{eqnarray*}
which concludes the proof.
\end{proof}

\begin{proof}[Proof of Proposition \ref{prop:means}]
In light of Proposition \ref{prop:carolandykstra} and the definition of $Y^G$ 
(along with some simple calculations), it is sufficient to prove that
\begin{eqnarray}\label{line:gren_toprove}
(s-r+1)E\left[\sum_{x=r}^s \gren(\widetilde Y^{(r,s)})_x^2\right]&=& \sum_{i=1}^{s-r}\frac{1}{i+1},
\end{eqnarray}
using the notation of the Proposition \ref{prop:carolandykstra}.  
Without loss of generality we may assume that $r=0$, and for simplicity 
we write $\widetilde Y$ for $\widetilde Y^{(r,s)}$.

Let $k=s+1$, and let $Z_1, \ldots, Z_k$ denote $k$ i.i.d. N(0,1) random variables, l
et $\bar Z$ denote their average, and let $\widetilde Z_i = Z_i-\bar Z$ (which is independent of $\bar Z$).  
We then have that
\begin{eqnarray*}
E\left[\sum_{x=1}^k \gren(Z)_x^2\right]&=& E\left[\sum_{x=1}^k \gren(\widetilde Z+\bar Z)_x^2\right]\\
&=&  E\left[\sum_{x=1}^k \left\{\gren(\widetilde Z)_x+\bar Z\right\}^2\right]\\
&=&E\left[\sum_{x=1}^k \left\{\gren(\widetilde Z)_x\right\}^2 + \sum_{x=1}^k \bar Z^2\right]\\
&=&E\left[\sum_{x=1}^k \left\{\gren(\widetilde Z)_x\right\}^2 \right] +1\\
&=&(y+1)E\left[\sum_{x=0}^y \left\{\gren(\widetilde Y)_x\right\}^2 \right] +1
\end{eqnarray*}
Therefore, by Lemma \ref{lem:l2_eq_touch}, to prove \eqref{line:gren_toprove},  it is sufficient to show that
\begin{eqnarray*}
E\left[\sum_{x=1}^k \gren(Z)_x^2\right]=E[T]=\sum_{i=1}^k \frac{1}{i},
\end{eqnarray*}
where $T$ denotes the number of touchpoints of the LCM with the cumulative sums of the $Z_i's.$

To do this, we use the results of \cite{MR0068154}.  He considers exchangeable 
random variables $X_1, X_2, \ldots$ and their partial sums $S_0=0, S_1, S_2, \ldots, S_n=\sum_{i=1}^n X_i$, 
and shows that the number $H_n$ of values $i\in\{1, \ldots, n-1\}$ for which $S_i$ coincides with the 
least concave majorant (equivalently the greatest convex minorant) of the sequence $S_0, \ldots, S_n$ has mean given by
\begin{eqnarray*}
E[H_n]&=& \sum_{i=1}^n \frac{1}{i+1},
\end{eqnarray*}
as long as the random variables $X_1, \ldots, X_n$ are symmetrically dependent and
\begin{eqnarray*}\label{line:SAcond}
P(S_i/i = S_j/j)=0, \ \ 1\leq i < j \leq n.
\end{eqnarray*}
The vector $X_1, \ldots, X_n$ is symmetrically dependent if its joint cumulative 
distribution function $P(X_i\leq x_i, i=1, \ldots, n)$ is a symmetric function of $x_1, \ldots, x_n.$
This result is Theorem 5 in \cite{MR0068154}.  Clearly, we have that 
$E[T-1]=E[H_{k}]$, for $X_1=Z_1, \ldots, X_n=Z_{k},$ which are exchangeable, 
and satisfy the required conditions.  The result follows.
\end{proof}

\begin{proof}[Proof of Remark \ref{rem:prob_equal}]
To prove this result we continue with the notation of the previous proof.  
Equality of $\gren(Y)$ with $Y$ holds if and only if the above partition 
$T=\{0, \ldots, y\}$.  By Theorem 5 of \cite{MR0068154}, this occurs with probability $1/(y+1)!.$
\end{proof}

\begin{proof}[Proof of Remark \ref{rem:pointwise_ineq}]
By Proposition \ref{prop:carolandykstra}  (and using the notation defined there), it is enough to prove that 
\begin{eqnarray*}
E[\gren(\widetilde Y)_x^2] &\leq&  \frac{1}{\tau} - \frac{1}{\tau^2},
\end{eqnarray*}
where for simplicity we write $\widetilde Y=\widetilde Y^{(s,r)}$.  
Let $\{\widetilde W_x\}_{x=r}^s$ be i.i.d. normal random variables with  
mean zero and variance $1/\tau,$ and let $\overline W = (\sum_{x=r}^s \widetilde W_x )/\tau$.  
Then $\widetilde Y \stackrel{d}{=} \widetilde Z - \overline Z,$ and also 
$\gren(\widetilde Z)_x =\gren(\widetilde Z - \overline Z)_x + \overline Z$.  
Notice also that $\widetilde Z-\overline Z$ and $\overline Z$ are independent.  
We therefore find that 
\begin{eqnarray*}
E[\gren(\widetilde Y)_x^2]+ 1/\tau^2&=& E[\gren(\widetilde Z)_x^2]\\
&\leq & E[\widetilde Z_x^2] = 1/\tau,
\end{eqnarray*}
the latter inequality following directly from Theorem 1.6.2 of \cite{MR961262}, 
since the elements of $\widetilde Z$ are independent.  
\end{proof}

\subsection{Estimating the mixing distribution: proofs}

\begin{proof}[Proof of Theorem \ref{thm:mixing_consistent}]
Since $||\tilde q_n-q||_k \leq ||\tilde q_n-q||_1$ and
$H(\tilde q_n,q) \leq \sqrt{||\tilde q_n-q||_1}$, it is sufficient to only
consider convergence in the $\ell_1$ norm.  Note that
\begin{eqnarray*}
|\tilde q_{n,x}-q_x|&\leq & (x+1) \left\{|\tilde p_{n,x+1}-p_{x+1}|+|\tilde p_{n,x}-p_{x}|\right\},
\end{eqnarray*}
and therefore we may further reduce the problem to showing that
$\sum_{x\geq 0} x |\tilde p_{n,x}-p_x|$ converges to zero.

For $\tilde p_n= \widehat p_{n,x}$, we have that for any large $K$
\begin{eqnarray*}
\sum_{x\geq 0} x |\widehat p_{n,x}-p_{x}|
&\leq & K \sup_{x < K} |\widehat{p}_{n,x} -p_x |
             + \sum_{x\geq K} x p_x + \sum_{x\geq K}  x \widehat p_{n,x},
\end{eqnarray*}
and since $E_p[X]$ exists by assumption, it follows from the law of large numbers that for any~$K$,
\begin{eqnarray*}
\sum_{x\geq K}  x \widehat p_{n,x} \rightarrow   \sum_{x\geq K} x \,p_x,
\end{eqnarray*}
almost surely. 
The proof now proceeds as in the proof of Theorem \ref{thm:GlobalConsistencyTheorem}.

For the rearrangement estimator and the MLE, we may use the same approach.
The key is to note that
$\sum_{x\geq K}x \tilde p_{n,x} \leq \sum_{x\geq K}x \widehat p_{n,x}$,
for any $K$ and for both $\tilde p_n = \widehat p_n^R, \widehat p_n^G$.  
This holds since $f_x = \bb I_{x\geq K} x$ is an increasing function and therefore 
\eqref{line:key_ineq1} of Lemma \ref{lem:key_ineq} applies.  
\end{proof}

\begin{proof}[Proof of Theorem \ref{thm:mixing_rate}]
Since $\kappa<\infty$ by assumption, the theorem follows directly from the results of Sections \ref{sec:LimitDistributions} and \ref{sec:limitdistributions_metrics}, as well as Theorem \ref{thm:mixing_consistent}.
\end{proof}

\bibliographystyle{ims}
\bibliography{DiscreteShape}

\end{document}